\documentclass{amsart}
\usepackage{amsmath}
\usepackage{amsthm}
\usepackage{amssymb}
\usepackage{color}
\usepackage{enumerate}

\newtheorem{thm}{Theorem}
\newtheorem{lem}[thm]{Lemma}

\newtheorem{defi}[thm]{Definition}
\newtheorem{prop}[thm]{Proposition}
\newtheorem{rk}[thm]{Remark}
\newtheorem{nota}[thm]{Notation}
\newtheorem{ass}[thm]{Assumption}

\newcommand{\Sd}{{\mathbb{S}_{d-1}}}
\newcommand{\rr}{{\mathbb{R}}}
\newcommand{\rd}{{\rr^d}}
\newcommand{\nn}{{\mathbb{N}}}
\newcommand{\cI}{{\mathcal{I}}}
\newcommand{\cZ}{{\mathcal{Z}}}
\newcommand{\cF}{{\mathcal{F}}}
\newcommand{\cG}{{\mathcal{G}}}
\newcommand{\cL}{{\mathcal{L}}}
\newcommand{\cP}{{\mathcal{P}}}
\newcommand{\cR}{{\mathcal{R}}}
\newcommand{\cV}{{\mathcal{V}}}
\newcommand{\cS}{{\mathcal{S}}}
\newcommand{\cH}{{\mathcal{H}}}
\newcommand{\cW}{{\mathcal{W}}}
\newcommand{\cE}{{\mathcal{E}}}
\newcommand{\cC}{{\mathcal{C}}}
\newcommand{\cD}{{\mathcal{D}}}
\newcommand{\cB}{{\mathcal{B}}}
\newcommand{\cU}{{\mathcal{U}}}
\newcommand{\chL}{\hat{\mathcal{L}}}
\newcommand{\hTheta}{{\hat \Theta}}
\newcommand{\hB}{{\hat B}}
\newcommand{\hZ}{{\hat Z}}
\newcommand{\tB}{{\tilde B}}
\newcommand{\tZ}{{\tilde Z}}
\newcommand{\tX}{{\tilde X}}
\newcommand{\e}{\epsilon}
\newcommand{\vip}{\vskip.2cm}
\newcommand{\indiq}{{\bf 1}}
\newcommand{\sm}{{s-}}
\newcommand{\E}{\mathbb{E}}
\newcommand{\PR}{\mathbb{P}}
\newcommand{\intot}{\int_0^t }
\newcommand{\intrd}{\int_\rd}
\newcommand{\dd}{{\rm d}}
\newcommand{\sg}{{\rm sg}}

\newcommand{\bB}{\bar B}
\newcommand{\bZ}{\bar Z}
\newcommand{\chZ}{\check Z}

\newcommand{\bM}{\mathbf M}
\newcommand{\bN}{\mathbf N}
\newcommand{\tbN}{\tilde{\bN}}
\newcommand{\ddiv}{{\rm div}}
\renewcommand{\mapsto}{\to}

\begin{document}

\title[Multi-dimensional critical kinetic Fokker-Planck equations]
{Anomalous diffusion for multi-dimensional critical Kinetic Fokker-Planck equations}

\author{Nicolas Fournier and Camille Tardif}

\address{Sorbonne Universit\'e - LPSM, Campus Pierre et Marie Curie, 
Case courrier 158, 4 place Jussieu, 75252 PARIS CEDEX 05,
{\tt nicolas.fournier@sorbonne-universite.fr, camille.tardif@sorbonne-universite.fr}.}

\thanks{This research was supported 
by the French  ANR-17- CE40-0030 EFI}



\begin{abstract}
We consider a particle moving in $d\geq 2$ dimensions, its 
velocity being a reversible diffusion process,
with identity   diffusion coefficient, of which the invariant measure behaves, roughly, like
$(1+|v|)^{-\beta}$ as $|v|\to \infty$, for some constant $\beta>0$.
We prove that for large times, after a suitable rescaling, 
the position process resembles a Brownian motion if $\beta\geq 4+d$,
a stable process if $\beta\in [d,4+d)$ and an integrated multi-dimensional generalization
of a Bessel process if $\beta\in (d-2,d)$.
The critical cases $\beta=d$, $\beta=1+d$  and $\beta=4+d$ require special rescalings.
%This continues \cite{ft} where the one-dimensional symmetric case was studied.
\end{abstract}

\maketitle

\section{Introduction and results}
\setcounter{equation}{0}

\subsection{Motivation and references}

Describing the motion of a particle with complex dynamics, after space-time rescaling, by a simple diffusion,
is a natural and classical subject. See for example Langevin \cite{l}, Larsen-Keller \cite{lk}, 
Bensoussans-Lions-Papanicolaou \cite{blp} and Bodineau-Gallagher-Saint-Raymond \cite{bgsr}.
Particles undergoing anomalous diffusion are often observed in physics, and many mathematical works
show how to modify some Boltzmann-like linear equations to asymptotically get some fractional diffusion limit
(i.e. a radially symmetric L\'evy stable jumping position process). See 
Mischler-Mouhot-Mellet \cite{mmm}, Jara-Komorowski-Olla \cite{jko}, 
Mellet \cite{m}, Ben Abdallah-Mellet-Puel \cite{bmp,bmp2}, etc.

\vip

The kinetic Fokker-Planck equation is also of constant use in physics, because it is
rather simpler than the Boltzmann equation:
assume that the density $f_t(x,v)$ of particles with position $x \in \rd$ and velocity $v\in \rd$
at time $t\geq 0$ solves
\begin{equation}\label{fpe}
\partial_t f_t(x,v) + v \cdot\nabla_xf_t(x,v)= \frac 12 \big(\Delta_{v}f_t(x,v) 
+ \beta {\rm div}_v[ F(v) f_t(x,v)]\big),
\end{equation}
for some force field $F:\rd\mapsto\rd$ and some constant $\beta>0$ that will be useful later.
We then try to understand the behavior of the {\it density} $\rho_t(x)=\int_{\rd} f_t(x,v) \dd v$
for large times.

\vip

The {\it trajectory} corresponding to \eqref{fpe}
is the following stochastic kinetic model:
\begin{equation}\label{eds}
V_t = v_0 + B_t - \frac{\beta}{2} \intot F(V_s) \dd s\quad \hbox{and} \quad X_t=x_0+\intot V_s \dd s.
\end{equation}
Here $(B_t)_{t\geq 0}$ is a $d$-dimensional Brownian motion.
More precisely, for $(V_t,X_t)_{t\geq 0}$ (with values in $\rd\times\rd$)
solving \eqref{eds}, the family of time-marginals $f_t=\cL(X_t,V_t)$ solves \eqref{fpe} in the sense
of distributions.

\vip

It is well-known that if $F$ is sufficiently confining, then the velocity process
$(V_t)_{t\geq 0}$ is close to equilibrium, its invariant distribution has a fast decay,
and after rescaling, the position process $(X_t)_{t\geq 0}$ resembles a Brownian motion in large time.
In other words, $(\rho_t)_{t\geq 0}$ is close to the solution to the heat equation. 

\vip

If on the contrary $F$ is not sufficiently confining, e.g. if $F\equiv 0$, then $(X_t)_{t\geq 0}$
cannot be reduced to an autonomous Markov process in large times. In other words,
$(\rho_t)_{t\geq 0}$ does not solve an autonomous time-homogeneous PDE.

\vip

The only way to hope for some anomalous diffusion limit, for a Fokker-Planck toy model like \eqref{fpe},
is to choose the force in such a way that the invariant measure of the velocity process has a fat tail.
One realizes that one has to choose $F$ behaving like $F(v)\sim 1/|v|$ as $|v|\to \infty$,
and the most natural choice is $F(v)=v/(1+|v|^2)$.
Now the asymptotic behavior of the model may depend on the value of $\beta>0$,
since the invariant distribution of the velocity process
is given by $(1+|v|^2)^{-\beta/2}$, up to some normalization constant.

\vip

The Fokker-Planck model \eqref{fpe}, with the force $F(v)=v/(1+|v|^2)$, is the object of the papers by
Nasreddine-Puel \cite{np} ($d\geq 1$ and $\beta>4+d$, diffusive regime), Cattiaux-Nasreddine-Puel \cite{cnp}
($d\geq 1$ and $\beta=4+d$, critical diffusive regime) and Lebeau-Puel \cite{lp}
($d=1$ and $\beta \in (1,5)\setminus\{2,3,4\}$). In this last paper, the authors show
that after time/space rescaling, the density $(\rho_t)_{t\geq 0}$ is close to the solution to the fractional
heat equation with index $\alpha/2$, where $\alpha=(\beta+1)/3$.
In other words, $(X_t)_{t\geq 0}$ resembles a symmetric $\alpha$-stable process.
This work relies on a spectral approach and involves many explicit computations.

\vip

Using an alternative probabilistic approach, we studied the one-dimensional case $d=1$
in \cite{ft}, treating all the cases $\beta\in (0,\infty)$ in a rather concise way.
We allowed for a more general (symmetric) force field $F$.

\vip

Physicists observed that atoms subjected to Sisyphus cooling
diffuse anomalously, see 
Castin-Dalibard-Cohen-Tannoudji \cite{cdct}, Sagi-Brook-Almog-Davidson \cite{sbad} and 
Marksteiner-Ellinger-Zoller \cite{mez}. 
A theoretical study has been proposed by Barkai-Aghion-Kessler \cite{bak}. 
They precisely model the motion of atoms by
\eqref{fpe} with the force $F(v)=v/(1+v^2)$ induced by the laser field, simplifying very slightly
the model derived in \cite{cdct}.
They predict, in dimension $d=1$ and with a quite high level of rigor,
the results of \cite[Theorem 1]{ft}, excluding the critical cases, with the following terminology:
normal diffusion when $\beta>5$, L\'evy diffusion when $\beta \in (1,5)$ and 
{\it Obukhov-Richardson phase} when $\beta \in (0,1)$. This last case is treated in a rather
confused way in \cite{bak}, mainly because no tractable explicit computation can be handled, since 
the limit process is an integrated symmetric Bessel process.

\vip

In \cite{kb}, Kessler-Barkai mention other fields of applications of this model, such as
single particle models for long-range interacting systems (Bouchet-Dauxois \cite{bd}), 
condensation describing a charged particle in the vicinity of a charged polymer (Manning, \cite{m}), and
motion of nanoparticles in an appropriately constructed force field (Cohen, \cite{co}).
We refer to \cite{np,cnp,lp} and especially \cite{kb,bak} 
for many other references and motivations.

\vip

The goal of the present paper is to study what happens in higher dimension.
We also allow for some non-radially symmetric force, to understand more deeply what happens,
in particular in the stable regime. To our knowledge, the results are completely new.
The proofs are technically much more involved than in dimension $1$.

\subsection{Main results}

In the whole paper, we assume that the initial condition  $(v_0,x_0) \in \rd\times\rd$ is deterministic
and, for simplicity, that $v_0 \neq 0$.
We also assume that the force is of the following form.

\begin{ass}\label{as}
There is a potential $U$ of the form $U(v)=\Gamma(|v|)\gamma(v/|v|)$, for some
$\gamma : \Sd \mapsto (0,\infty)$ of class $C^\infty$ and some $\Gamma:\rr_+\mapsto (0,\infty)$
of class $C^\infty$ satisfying $\Gamma(r) \sim r$ as $r\to\infty$, such that
for any $v\in \rr^d \setminus \{0\}$, $F(v)=\nabla[\log U(v)]=[U(v)]^{-1}\nabla U(v)$.
\end{ass}

Observe that $F$ is
of class $C^\infty$ on $\rr^d \setminus\{0\}$.
We will check the following well-posedness result.

\begin{prop}\label{eui}
Under Assumption \ref{as}, \eqref{eds} has a pathwise unique solution $(V_t,X_t)_{t\geq 0}$, 
which is furthermore
$(\rr^d\setminus\{0\})\times \rr^d$-valued.
\end{prop}

\begin{rk}\label{inv}
Assume that $\beta>d$. As we will see, the velocity process
$(V_t)_{t\geq 0}$ has a unique invariant probability measure 
given by $\mu_\beta(\dd v)=c_\beta [U(v)]^{-\beta}\dd v$, for 
$c_\beta = [\intrd[U(v)]^{-\beta}\dd v]^{-1}$.
\end{rk}

As already mentioned, the main example we have in mind is 
$\Gamma(r)=\sqrt{1+r^2}$ and $\gamma\equiv 1$, whence $U(v)=\sqrt{1+|v|^2}$ and $F(v)=v/(1+|v|^2)$.
We also allow for some non radially symmetric potentials to understand more deeply what may happen.

\vip

In the whole paper, we denote by $\cS_d^+$ the set of symmetric positive-definite $d\times d$ matrices.
We also denote by $\varsigma$ the uniform probability measure on $\Sd$.

\vip

For $((Z^\e_{t})_{t\geq 0})_{\e\geq 0}$ a family of $\rr^d$-valued processes, we
write $(Z^\e_{t})_{t\geq 0}\stackrel{f.d.}\longrightarrow (Z_t^0)_{t\geq 0}$ if for all 
finite subset $S\subset [0,\infty)$ the vector $(Z^\e_{t})_{t\in S}$ goes in 
law to $(Z_t^0)_{t\in S}$ as $\e\to 0$; and we write
$(Z^\e_{t})_{t\geq 0}\stackrel{d}\longrightarrow (Z^0_t)_{t\geq 0}$ if the convergence in law 
holds in the usual sense of continuous processes. Here is our main result.

\begin{thm}\label{mr}
Fix $\beta>0$, suppose Assumption \ref{as} and consider the solution $(V_t,X_t)_{t\geq 0}$ to \eqref{eds}.
We set $a_\beta=[\int_\Sd [\gamma(\theta)]^{-\beta}\varsigma(\dd \theta)]^{-1}>0$, as well as 
$M_\beta=a_\beta\int_\Sd \theta [\gamma(\theta)]^{-\beta}\varsigma(\dd \theta)\in\rd$ and, if 
$\beta>1+d$, $m_\beta=\intrd v \mu_\beta(\dd v) \in \rd$.

\vip

(a) If $\beta>4+d$,
there is $\Sigma \in \cS_d^+$ such that 
$$(\e^{1/2} [X_{t/\e}- m_\beta t /\e] )_{t\geq 0} \stackrel{f.d.} \longrightarrow (\Sigma B_t)_{t\geq 0},$$ 
where $(B_t)_{t\geq 0}$ is a $d$-dimensional Brownian motion.

\vip 

(b) If $\beta=4+d$ and if $\int_1^\infty r^{-1}|r \Gamma'(r)/\Gamma(r)-1|^2\dd r<\infty$,
then 
$$(\e^{1/2}|\log \e|^{-1/2} [X_{t/\e}- m_\beta t /\e] )_{t\geq 0} \stackrel{f.d.}\longrightarrow (\Sigma B_t)_{t\geq 0},$$ 
for some $\Sigma \in \cS_d^+$, where $(B_t)_{t\geq 0}$ is a $d$-dimensional Brownian motion.

\vip

(c) If $\beta \in (1+d,4+d)$, set $\alpha=(\beta+2-d)/3$. Then 
$$(\e^{1/\alpha} [X_{t/\e}-m_\beta t/\e] )_{t\geq 0}\stackrel{f.d.}
\longrightarrow (S_t)_{t\geq 0},$$ 
where $(S_t)_{t\geq 0}$ is a non-trivial $\alpha$-stable L\'evy process.

\vip

(d) If $\beta = 1+d$ and if $\int_1^\infty r^{-1}|r/\Gamma(r)-1|\dd r<\infty$
there exists a constant
$c>0$ such that 
$$(\e[X_{t/\e}- c M_\beta |\log \e| t/\e ])_{t\geq 0}\stackrel{f.d.} \longrightarrow (S_t)_{t\geq 0},$$ 
where $(S_t)_{t\geq 0}$ is a non-trivial $1$-stable L\'evy process.

\vip

(e) If $\beta \in (d,1+d)$, set $\alpha=(\beta+2-d)/3$. Then 
$$(\e^{1/\alpha} X_{t/\e} )_{t\geq 0}\stackrel{f.d.}
\longrightarrow (S_t)_{t\geq 0},$$ 
where $(S_t)_{t\geq 0}$ is a non-trivial $\alpha$-stable L\'evy process.

\vip

(f) If $\beta=d$, then 
$$(|\e \log \e|^{3/2} X_{t/\e} )_{t\geq 0}\stackrel{f.d.} \longrightarrow (S_t)_{t\geq 0},$$ 
where $(S_t)_{t\geq 0}$ is a non-trivial $2/3$-stable L\'evy process.

\vip

(g) If $\beta\in (d-2,d)$, 
$$(\e^{3/2} X_{t/\e})_{t\geq 0}\stackrel{d}\longrightarrow \Big(\int_0^t \cV_s \dd s\Big)_{t\geq 0},$$ 
where $(\cV_t)_{t\geq 0}$ is a $\rr^d$-valued continuous process (see Definition \ref{ddd}) of which the norm
$(|\cV_t|)_{t\geq 0}$ is a Bessel process with dimension
$d-\beta$ issued from $0$.
\end{thm}

The strong regularity of $U$ is only used to apply as simply as possible some classical PDE results.

\begin{rk}
(i) In the diffusive regimes (a) and (b), the matrix $\Sigma$ depends only on $U$ and $\beta$,
see Remarks \ref{rkd}-(i) and \ref{rkdc}-(i).
The additional condition when $\beta=4+d$ 
more or less imposes that $\Gamma'(r)\to 1$
as $r\to \infty$ and that this convergence does not occur too slowly. This is slightly restrictive, but
found no way to get rid of this assumption.

\vip

(ii) In cases (c), (d), (e) and (f), the L\'evy measure of the $\alpha$-stable process $(S_t)_{t\geq 0}$ 
only depends on $U$ and $\beta$: a complicated formula involving It\^o's excursion measure
can be found in Proposition \ref{tpfini}-(i). The additional condition 
when $\beta= 1+d$ requires that $r^{-1}\Gamma(r)$ does not converge too slowly to $1$ as $r\to \infty$
and is very weak. The constant $c>0$ in point (d) is explicit, see Remark \ref{rsc}.

\vip

(iii) In point (g), the law of $(\cV_t)_{t\geq 0}$ depends only on $\gamma$ and on $\beta$.

\vip

(iv) Actually, point (g) should extend to any value of $\beta \in (-\infty, d)$, with a rather simple proof,
the definition of the limit process $(\cV_t)_{t\geq 0}$ being less involved:
see Definition \ref{ddd} and observe that for $\beta\leq d-2$, the set of zeros of 
a Bessel process with dimension $d-\beta$ issued from $0$ is trivial.
We chose not to include this rather uninteresting case because the paper is already technical enough.
\end{rk}

For the main model we have in mind, Theorem \ref{mr} applies and its statement 
considerably simplifies. See Remarks \ref{rkd}-(ii) and \ref{rkdc}-(ii)
and Proposition \ref{tpfini}-(ii).

\begin{rk} Assume that
$\Gamma(r)=\sqrt{1+r^2}$ and $\gamma\equiv 1$, whence $F(v)=v/(1+|v|^2)$.

\vip

(a) If $\beta > 4+d$, then 
$(\e^{1/2} X_{t/\e})_{t\geq 0} \stackrel{f.d.} \longrightarrow (q B_t)_{t\geq 0},$
where $(B_t)_{t\geq 0}$ is a $d$-dimensional Brownian motion, for some explicit $q>0$.

\vip

(b) If $\beta=4+d$, then 
$(\e^{1/2}|\log \e|^{-1/2}X_{t/\e})_{t\geq 0}  \stackrel{f.d.} \longrightarrow (q B_t)_{t\geq 0},$
where $(B_t)_{t\geq 0}$ is a $d$-dimensional Brownian motion, for some explicit $q>0$.

\vip

(c)-(d)-(e) If $\beta\in (d,4+d)$, then 
$(\e^{1/\alpha}X_{t/\e})_{t\geq 0}  \stackrel{f.d.} \longrightarrow (S_t)_{t\geq 0},$
where $(S_t)_{t\geq 0}$ is a radially symmetric $\alpha$-stable process with non-explicit
multiplicative constant and where $\alpha=(\beta+2-d)/3$.

\vip

(f) If $\beta=d$, then
$(|\e\log\e|^{3/2}X_{t/\e})_{t\geq 0}  \stackrel{f.d.} \longrightarrow (S_t)_{t\geq 0},$
where $(S_t)_{t\geq 0}$ is a radially symmetric $2/3$-stable process with non-explicit
multiplicative constant.

\vip

(g) If $\beta\in (d-2,d)$, 
$(\e^{3/2} X_{t/\e})_{t\geq 0}\stackrel{d}\longrightarrow (\int_0^t \cV_s \dd s)_{t\geq 0},$
with $(\cV_t)_{t\geq 0}$ introduced in Definition \ref{ddd}.
\end{rk}

\subsection{Comments}

Pardoux-Veretennikov \cite{pv} studied in great generality the diffusive case, allowing for
some much more general SDEs with non-constant diffusion coefficient and general drift coefficient. 
Their results are sufficiently sharp to include the diffusive case $\beta>4+d$ when
$F(v)=v/(1+|v|^2)$. Hence the diffusive case (a)
is rather classical.

\vip

We studied the one-dimensional case $d=1$ with an even potential $U$
in \cite{ft}.
Many technical difficulties appear in higher dimension. In the diffusive and critical diffusive regime,
the main difficulty is that we cannot solve explicitly the Poisson equation
$\cL \phi(v)=v$ (with $\cL$ the generator of $(V_t)_{t\geq 0}$), while this is feasible
in dimension $1$. Observe that such a problem would disappear if dealing only with the
force $F(v)=v/(1+|v|^2)$.

\vip

We use a spherical decomposition $V_t=R_t\Theta_t$ of the velocity process.
This is of course very natural in this context, and we do not see how to proceed in another way.
However, since in some sense, after rescaling, the radius process $(R_t)_{t\geq 0}$ resembles a Bessel
process with dimension $d-\beta \in (-\infty,2)$, which hits $0$, spherical coordinates are 
rather difficult to deal with, the process $\Theta_t$ moving very fast each time $R_t$ touches $0$.

\vip

In dimension $1$, the most interesting {\it stable} regime is derived as follows.
We write $(V_t)_{t\geq 0}$ as a function of a time-changed Brownian motion $(W_t)_{t\geq 0}$,
using the classical {\it speed measures} and {\it scale functions} of one-dimensional SDEs and
express $\e^{1/\alpha} X_{t/\e}$ accordingly.
Passing to the limit as $\e\to 0$, we find the expression of the 
(symmetric) stable process in terms of the Brownian motion $(W_t)_{t\geq 0}$ and of its inverse local time at $0$
discovered by Biane-Yor \cite{by}, see also It\^o-McKean \cite[p 226]{imk} and Jeulin-Yor \cite{jy}.
In higher dimension, the situation is much more complicated, and we found no simpler way 
than writing our limiting stable processes using some {\it excursion Poisson point processes}.

\vip

Let us emphasize that our proofs are qualitative. 
On the contrary, even in dimension $1$,
the informal proofs of Barkai-Aghion-Kessler \cite{bak} rely on very explicit computations and
explicit solutions to O.D.E.s in terms of modified Bessel functions, and
Lebeau-Puel \cite{lp} also use rather explicit computations.

\subsection{Plan of the paper}

To start with, we explain informally in Section \ref{inf}
our proof of Theorem \ref{mr} in the most interesting case, that is
when $F(v)=v/(1+|v|^2)$ and when $\beta \in (d,4+d)$.
\vip

In Section \ref{notadebut}, we introduce some notation of constant use in the paper.
\vip

In Section \ref{Srep}, we write the velocity process $(V_t)_{t\geq 0}$
as $(R_t\Theta_t)_{t\geq 0}$, the radius process 
$(R_t)_{t\geq 0}$ solving an autonomous SDE, and the process $(\Theta_t)_{t\geq 0}$ being $\Sd$-valued.
We also write down a representation of the radius as a function of a time-changed Brownian motion,
using the classical theory of speed measures and scale functions  of one-dimensional SDEs.
\vip

We designed the other sections to be as independent as possible.
\vip

Sections \ref{Ss}, \ref{Sb},  \ref{Sd} and \ref{Scd} treat respectively the
stable regime (cases (c)-(d)-(e)-(f)), integrated Bessel regime (case (g)),
diffusive regime (case (a)) and critical diffusive regime (case (b)).
\vip

Finally, an appendix lies at the end of the paper and contains 
some more or less classical results about ergodicity of diffusion processes, about It\^o's excursion measure,
about Bessel processes, about convergence of inverse functions and, finally, a few technical estimates.

\section{Informal proof in the stable regime with a symmetric force}\label{inf}

We assume in this section that $F(v)=v/(1+|v|^2)$ and that  $\beta \in (d,4+d)$
and explain informally how to prove Theorem \ref{mr}-(c)-(d)-(e). We also assume, for example, that
$x_0=0$ and that $v_0=\theta_0\in\Sd$.

\vip

\underline{\it Step 1.} Writing the velocity process in spherical coordinates, we find that
$V_t=R_t\hTheta_{H_t}$, where
\begin{equation}\label{i1}
R_t=1+\tB_t+ \intot \Big(\frac{d-1}{2R_s}- \frac{\beta R_s}{1+R_s^2} \Big)\dd s,
\end{equation}
for some one-dimensional Brownian motion $(\tB_t)_{t\geq 0}$, independent of a spherical
$\Sd$-valued Brownian motion $(\hTheta_t)_{t\geq 0}$ starting from $\theta_0$, and where $H_t=\int_0^t R_s^{-2}\dd s$.

\vip

\underline{\it Step 2.} Using the classical {\it speed measure} and {\it scale function}, we may write
the radius process $(R_t)_{t\geq 0}$ as a space and time changed Brownian motion: set
$h(r)=(\beta+2-d)\int_{1}^r u^{1-d}[1+u^2]^{\beta/2} \dd u$, which is
an increasing bijection from $(0,\infty)$ into $\rr$.
We denote by $h^{-1}:\rr \mapsto (0,\infty)$ its inverse function and by
$\sigma(w)=h'(h^{-1}(w))$ from $\rr$ to $(0,\infty)$. For $(W_t)_{t\geq 0}$ a
one-dimensional Brownian motion, consider the continuous increasing process
$A_t=\intot [\sigma(W_s)]^{-2} \dd s$ and its inverse $(\rho_t)_{t\geq 0}$. 
One can classically check that $R_t=h^{-1}(W_{\rho_t})$ is a (weak) solution to \eqref{i1}, so that
we can write the position process as 
$$
X_t=\int_0^t h^{-1}(W_{\rho_s}) \hTheta_{H_s}\dd s = \int_0^{\rho_t}
\frac{h^{-1}(W_u)}{[\sigma(W_u)]^2}\hTheta_{H_{A_u}} \dd u.
$$
We used the change of variables $\rho_s=u$, i.e. $s=A_u$, whence $\dd s = [\sigma(W_u)]^{-2} \dd u$.
We next observe that $T_t=H_{A_t}=\int_0^{A_t}[h^{-1}(W_{\rho_s})]^{-2}\dd s=\int_0^t [\psi(W_u)]^{-2}\dd u$,
where we have set $\psi(w)=h^{-1}(w)\sigma(w)$. Finally,
$$
X_{t/\e}= \int_0^{\rho_{t/\e}} \frac{h^{-1}(W_u)}{[\sigma(W_u)]^2}\hTheta_{T_u} \dd u.
$$

\vip

\underline{\it Step 3.} To study the large time behavior of the position process, it is more convenient
to start from a fixed Brownian motion $(W_t)_{t\geq 0}$ and to use Step 2 with the Brownian
motion $(W^\e_t=(c\e)^{-1}W_{(c \e)^2 t})_{t\geq 0}$,
for some constant $c>0$ to be chosen later.
After a few computations, we find that
$$
X_{t/\e} =\int_0^{\rho^\e_{t}} \frac{h^{-1}(W_s/(c\e))\hTheta_{T^\e_s}}{(c \e)^2[ \sigma(W_s/(c\e))]^2}\dd s,\quad
\hbox{where}\quad T_t^\e=\int_0^t \frac{\dd u}{[c\e \psi(W_u/c\e)]^{2}},\quad
A^\e_t=\intot \frac{\dd u}{c^2\e [\sigma(W_u/(c\e))]^{2}},
$$
and where $(\rho^\e_t)_{t\geq 0}$ is the inverse of $(A^\e_t)_{t\geq 0}$.

\vip

\underline{\it Step 4.} If choosing $c=\int_\rr [\sigma(x)]^{-2}\dd x$, it holds that for all $t\geq0$,
$\lim_{\e \to 0} A^\e_t=L^0_t$ a.s., where $(L^0_t)_{t\geq 0}$ is the local time of $(W_t)_{t\geq 0}$:
by the occupation times formula, see Revuz-Yor \cite[Corollary 1.6 p 224]{ry},
$$
A^\e_t=\int_\rr \frac{L^x_t \dd x}{c^2\e [\sigma(x/(c\e))]^{2}}=\int_\rr \frac{L^{c \e y}_t \dd y}{c [\sigma(y)]^{2}}
\longrightarrow \int_\rr \frac{\dd y}{c [\sigma(y)]^{2}} L^0_t =L^0_t.
$$
As a consequence, $\rho^\e_t$ tends to $\tau_t$, the inverse of $L^0_t$.

\vip

\underline{\it Step 5.}
Studying the function $h$ near $0$ and $\infty$, and then $h^{-1}$, $\sigma$ 
and $\psi$ near $-\infty$ and $\infty$, we
find that, with $\alpha=(\beta+1-d)/3$ (see Lemma \ref{fcts}-(ix) and (v)),

\vip

$\bullet$ $\lim_{\e\to 0} \e^{1/\alpha} (c\e)^{-2}h^{-1}(w/(c\e)) [\sigma(w/(c\e))]^{-2}=
c' w^{1/\alpha-2}\indiq_{\{w>0\}}$,

\vip

$\bullet$ $\lim_{\e\to 0}[c\e \psi(w/c\e)]^{-2} = c'' w^{-2} \indiq_{\{w>0\}}+ \varphi(w)\indiq_{\{w\leq 0\}}$,

\vip
\noindent for some constants $c',c''>0$ and some unimportant function $\varphi\geq 0$.
Here appears the scaling $\e^{1/\alpha}$.

\vip

Passing to the limit informally in the expression of Step 3, we find that
\begin{gather*}
\e^{1/\alpha} X_{t/\e} \longrightarrow S_t= c' \int_0^{\tau_t} W_s^{1/\alpha-2} \indiq_{\{W_s>0\}} \hTheta_{U_s}\dd s,\\
\hbox{where}\quad  U_t=c'' \intot W_u^{-2}\indiq_{\{W_u>0\}} \dd u + \intot \varphi(W_u)\indiq_{\{W_u\leq 0\}} \dd u.
\end{gather*}
Unfortunately, this expression does not make sense, because $U_t=\infty$ for all $t>0$, since the
Brownian motion is (almost) $1/2$-H\"older continuous and since it hits $0$.
But in some sense, $U_t-U_s$ is well-defined if $W_u>0$ for all $u\in (s,t)$.
And in some sense, the processes $(\hTheta_{U_s})_{s\in [a,b]}$ and $(\hTheta_{U_s})_{s\in [a',b']}$ are independent
if $W_u>0$ on $[a,b]\cup[a',b']$ and if there exists $t\in (b,a')$ such that $W_t=0$, since then
$U_{a'}-U_b=\infty$, so that the spherical Brownian motion $\hTheta$, at time $U_{a'}$, has completely
forgotten the values it has taken during $[U_a,U_b]$.

\vip

Since $(\tau_t)_{t\geq 0}$ is the inverse local time of $(W_t)_{t\geq 0}$, it holds that $\tau_t$ is a
stopping-time and that $W_{\tau_{t-}}=W_{\tau_t}=0$ for each $t\geq 0$.
Hence by the strong Markov property, for any reasonable
function $f:\rr\mapsto \rr^d$, the process $Z_t=\int_0^{\tau_t} f(W_s)\dd s$ is L\'evy,
and its jumps are given by $\Delta Z_t= \int_{\tau_{t-}}^{\tau_t} f(W_s)\dd s$, for $t\in J
=\{s\geq 0 : \Delta\tau_s>0\}$.

\vip

The presence of $\hTheta_{U_s}$ in the expression of $(S_t)_{t\geq 0}$ does not affect its L\'evy character,
because $(\hTheta_t)_{t\geq 0}$ is independent of $(W_t)_{t\geq 0}$ and because in some sense, 
the family $\{(\hTheta_{U_u})_{u\in [\tau_{s-},\tau_s]} : s \in J\}$
is independent. Hence $(S_t)_{t\geq 0}$ is L\'evy and its jumps are given by
$$
\Delta S_t = c' \int_{\tau_{t-}}^{\tau_t} W_s^{1/\alpha-2} \indiq_{\{W_s>0\}} \hTheta^t_{[c''\int_{(\tau_t+\tau_{t-})/2}^{s} W_u^{-2}\dd u]}
\dd s, \qquad t\in J,
$$
for some i.i.d. family $\{(\hTheta^t_{u})_{u\in \rr} : t \in J\}$ of eternal spherical Brownian motions.
Informally, for each $t\in J$, we have set $\hTheta^t_{u}=\hTheta_{U_{(\tau_t+\tau_{t-})/2}+u}$ for all $u\in\rr$.
The choice of $(\tau_t+\tau_{t-})/2$ for the time origin of the eternal spherical Brownian motion $\hTheta^t$
is arbitrary, any time in $(\tau_{t-},\tau_t)$ would be suitable.
Observe that the clock $c''\int_{(\tau_t+\tau_{t-})/2}^{s} W_u^{-2}\dd u$ is well-defined for all $s\in (\tau_{t-},\tau_t)$
because $W_u$ is continuous and does not vanish on $u\in (\tau_{t-},\tau_t)$.
This clock tends to $\infty$ as $u\to \tau_t$, and to $-\infty$ as $u \to \tau_{t-}$.

\vip

It only remains to verify that the L\'evy measure $q$ of $(S_t)_{t\geq 0}$ is radially symmetric, which is more
or less obvious by symmetry of the law of the eternal spherical Brownian motion; and enjoys the scaling property
that $q(A_a)=a^\alpha q(A)$ for all $A\in \cB(\rd\setminus\{0\})$ and all $a>0$, where
$A_a=\{x \in \rd : ax \in A\}$. This property is inherited from the scaling property of the Brownian
motion (this uses that the clock in the spherical Brownian motion
is precisely proportional to $c''\int_{(\tau_t+\tau_{t-})/2}^{s} W_u^{-2}\dd u$).

\vip

To write all this properly, we have to use It\^o's excursion theory.

\vip

Let us also mention one last difficulty: when $\alpha\geq 1$, the integral
$\int_0^{t} W_s^{1/\alpha-2} \indiq_{\{W_s>0\}} \dd s$ is a.s. infinite for all $t>0$.
Hence to study $S_t$, one really has to use the symmetries of the spherical Brownian motion and that
the clock driving it explodes each time $W$ hits $0$.

\section{Notation}\label{notadebut}
In the whole paper, we suppose Assumption \ref{as}. 
We summarize here a few notation of constant use.

\vip

Recall that $\cS_d^+$ is the set of symmetric positive-definite $d\times d$ matrices.

\vip

We write the initial velocity as $v_0=r_0\theta_0$, with $r_0>0$ and $\theta_0\in\Sd$.

\vip

For $u \in \rd \setminus\{0\}$, let $\pi_{u^\perp}=(I_d - \frac{u u^*}{|u|^2})$ be
the $d\times d$-matrix of the orthogonal projection on $u^\perp$.

\vip

For $\Psi:\rd\mapsto\rd$, let $\nabla^*\Psi=(\nabla \Psi_1 \; \cdots \nabla \Psi_d)^*$.

\vip 

Recall that $a_\beta=[\int_\Sd [\gamma(\theta)]^{-\beta}\varsigma(\dd \theta)]^{-1}>0$, where
$\varsigma$ is the uniform probability measure on $\Sd$. We introduce the 
probability measure $\nu_\beta(\dd \theta)=a_\beta  [\gamma(\theta)]^{-\beta}\varsigma(\dd \theta)$ on $\Sd$.
It holds that $M_\beta=\int_{\Sd} \theta \nu_\beta(\dd \theta)\in \rd$.

\vip

If $\beta>d$, we set $b_{\beta}=[\int_0^\infty [\Gamma(r)]^{-\beta}r^{d-1}\dd r]^{-1}$ 
and introduce the probability measure $\nu'_\beta(\dd r)=b_\beta[\Gamma(r)]^{-\beta}r^{d-1}\dd r$ on $(0,\infty)$.
It has a finite mean $m_\beta'=\int_0^\infty r \nu'_\beta(\dd r)>0$ if $\beta>1+d$.

\vip

Still when $\beta>d$, we recall that $c_\beta = [\intrd[U(v)]^{-\beta}\dd v]^{-1}$ and that
$\mu_\beta(\dd v)=c_\beta [U(v)]^{-\beta}\dd v$ on $\rd$. It holds that $c_\beta=a_\beta b_\beta$ and 
for any measurable
$\varphi:\rd\mapsto\rr_+$, we have 
$$
\intrd \varphi(v)\mu_\beta(\dd v) = \int_0^\infty \int_{\Sd} \varphi(r\theta) \nu_\beta(\dd \theta) \nu_\beta'(\dd r).
$$
In particular, we have $m_\beta = M_\beta m'_\beta$ for all $\beta>1+d$.

\vip

In the whole paper, we implicitly extend all the functions on $\Sd$ to $\rd\setminus\{0\}$ as follows:
for $\psi:\Sd\mapsto \rr$ and $v\in \rd\setminus\{0\}$, we set $\psi(v)=\psi(v/|v|)$.

\vip

We endow $\Sd$ with its natural Riemannian metric, denote by $T\Sd$ its tangent bundle and by 
$\nabla_S$, $\ddiv_S$ and $\Delta_S$ the associated gradient, divergence and Laplace
operators. With the above convention, for a function $\psi:\Sd\mapsto \rr$ and 
a vector field $\Psi:\Sd\mapsto T\Sd$, it holds that, for $\theta \in \Sd\subset\rd\setminus\{0\}$,
$$
\nabla_S \psi(\theta)=\nabla \psi(\theta), \quad  
\ddiv_S \Psi(\theta)=\ddiv \Psi(\theta) 
\quad \hbox{and}\quad \Delta_S \psi(\theta)= \Delta \psi(\theta).
$$

\section{Representation of the solution}\label{Srep}

Here we show that \eqref{eds} is well-posed and explain how to build a solution (in law)
from some independent radial and spherical processes, in a way that will allow us 
to study the large time behavior of the position process by coupling.

\begin{lem}\label{eusph}
Consider a $d$-dimensional Brownian motion $(\hB_t)_{t\geq 0}$. The following equation, of which the unknown
$(\hTheta_t)_{t\geq 0}$ is $\rd\setminus\{0\}$-valued,
\begin{equation}\label{eqt}
\hTheta_t=\theta_0 + \intot \pi_{\hTheta_s^\perp} \dd \hB_s 
- \frac{d-1}2\intot \frac{\hTheta_s}{|\hTheta_s|^2}
\dd s - \frac \beta 2 \intot  \pi_{\hTheta_s^\perp} 
\frac{\nabla\gamma(\hTheta_s)}{\gamma(\hTheta_s)}\dd s
\end{equation}
has a unique strong solution, which is furthermore $\Sd$-valued.
\end{lem}

Recall that we have extended $\gamma$ to $\rd\setminus\{0\}$ by setting
$\gamma(v)=\gamma(v/|v|)$.

\begin{proof}
The coefficients of this equation being of class $C^1$ on $\rr^d\setminus\{0\}$, there classically
exists a unique maximal strong solution (defined until it reaches $0$ or explodes to infinity), 
and we only have to check that this solution a.s. remains in
$\Sd$ for all times. But a classical computation using the It\^o formula shows that a.s., 
$|\hTheta_t|^2 = |\theta_0|^2=1$ for all $t\geq 0$. This uses the fact that for $\phi(\theta)=|\theta|^2$
defined on $\rd$,
we have $\nabla\phi(\theta)=2 \theta$, so that $(\nabla\phi(\theta))^* \pi_{\theta^\perp}=0$
and we have $\partial_{ij} \phi(\theta)=2 \delta_{ij}$, from which 
$\frac12 \sum_{i,j=1}^d \partial_{ij}\phi(\theta) (\pi_{\theta^\perp})_{ij} 
- \frac{d-1}2 \nabla\phi(\theta)\cdot |\theta|^{-2}\theta =0$.
\end{proof}

The SDE \eqref{eqr} below has a unique strong solution: it has a unique local
strong solution (until it reaches $0$ or $\infty$) because its coefficients are $C^1$ on $(0,\infty)$ 
and we will see in Lemma \ref{exr} that
one can build a $(0,\infty)$-valued global weak solution, so that the unique strong solution is global.

\begin{lem}\label{rep}
For two independent Brownian motions $(\tB_t)_{t\geq 0}$ (in dimension $1$) and $(\hB_t)_{t\geq 0}$ (in dimension $d$),
consider the $\Sd$-valued process $(\hTheta_t)_{t\geq 0}$ solution to \eqref{eqt}
and the $(0,\infty)$-valued process $(R_t)_{t\geq 0}$ solution to 
\begin{equation}\label{eqr}
R_t=r_0 + \tB_t + \frac{d-1}2\intot \frac{\dd s}{R_s} 
- \frac{\beta}{2}\intot \frac{\Gamma'(R_s)}{\Gamma(R_s)}\dd s.
\end{equation}
Setting $H_t=\intot R_s^{-2}\dd s$, $V_t=R_t \hTheta_{H_t}$ and $X_t=x_0+\intot V_s\dd s$, 
the $(\rd\setminus\{0\})\times \rd$-valued process 
$(V_t,X_t)_{t\geq 0}$ is a weak solution to \eqref{eds}.
\end{lem}

\begin{proof}
For each $t\geq 0$, $\nu_t=\inf\{s>0 : H_s>t\}$ is a $(\tilde \cF_s)_{s\geq 0}$-stopping time, where
$\tilde \cF_s=\sigma(\tB_u : u\leq s)$, so that we can set $\cH_t=\tilde \cF_{\nu_t} \lor \sigma(\hB_s : s\leq t)$.
Now for each $t\geq 0$, $H_t=\inf\{s>0 : \nu_s>t\}$ 
is a $(\cH_s)_{s\geq 0}$-stopping time and we can define the filtration
$\cG_t=\cH_{H_t}$. One classically checks that

\vip

$(a)$ $(\tB_t)_{t\geq 0}$ is a $(\cG_t)_{t\geq 0}$-Brownian motion, because $(\tB_{\nu_t})_{t\geq 0}$ is a 
$(\cH_t)_{t\geq 0}$-martingale, so that $(\tB_t=\tB_{\nu_{H_t}})_{t\geq 0}$ is a $(\cH_{H_t}=\cG_t)_{t\geq 0}$-martingale,
and we have $\langle \tB\rangle_t=t$ because  $(\tB_t)_{t\geq 0}$ is a Brownian motion;

\vip

(b) $\bB_t=\int_0^{H_t} R_{\nu_s}\dd \hB_s$ is a $(\cG_t)_{t\geq 0}$-Brownian motion with dimension $d$, 
because $(\bB_{\nu_t})_{t\geq 0}$ is a 
$(\cH_t)_{t\geq 0}$-martingale, so that $(\bB_t)_{t\geq 0}$ is a $(\cG_t)_{t\geq 0}$-martingale, and because
$\langle \bB \rangle_t=I_d \int_0^{H_t} R_{\nu_s}^2\dd s = I_d t$;

\vip

(c) these two Brownian motions are independent
because for all $i=1,\dots,d$, $\langle \tB,\bB^i\rangle \equiv 0$;

\vip

(d) for all continuous $(\cH_t)_{t\geq 0}$-adapted $(S_t)_{t\geq 0}$, 
we have $\int_0^{H_t} S_s \dd \hB_s = \intot R_s^{-1} S_{H_s} \dd \bB_s$.
Indeed, it suffices that for all $(\cG_t)_{t\geq 0}$-martingale $(M_t)_{t\geq 0}$, 
$\langle \int_0^{H_\cdot} S_s \dd \hB_s,M\rangle_t= \intot R_s^{-1}S_{H_s} \dd \langle \bB,M\rangle_s$.
But $(N_t=M_{\nu_t})_{t\geq 0}$ is
a $(\cH_t)_{t\geq 0}$-martingale, and we have $\langle \int_0^{H_\cdot} S_s \dd \hB_s,M\rangle_t
=\langle \int_0^{H_\cdot} S_s \dd \hB_s, \int_0^{H_\cdot}\dd N_s\rangle_t
=\int_0^{H_t} S_s \dd \langle \hB,N \rangle_s=\int_0^t S_{H_u} \dd(\langle \hB,N \rangle_{H_u})
=\int_0^t S_{H_u} R_{u}^{-1} \dd  \langle \bB,M\rangle_u$, because
$R_u \dd(\langle \hB,N \rangle_{H_u}) = \dd  \langle \bB,M\rangle_u$,
since
$\langle \bB,M\rangle_t=\langle \int_0^{H_\cdot} R_{\nu_s}\dd \hB_s, \int_0^{H_\cdot}\dd N_s\rangle_t
=\int_0^{H_t} R_{\nu_s} \dd \langle \hB,N \rangle_{s}
=\int_0^t R_u\dd (\langle \hB,N \rangle_{H_u})$.

\vip

We observe that $\Theta_t=\hTheta_{H_t}$ is $(\cG_t)_{t\geq 0}$-adapted
and, recalling \eqref{eqt} and that $|\hTheta_t|=1$,
\begin{align}\label{thetapc}
\Theta_t=&\theta_0 + \intot R_s^{-1}\pi_{\Theta_s^\perp}\dd \bB_s 
- \frac{d-1}2\intot R_s^{-2}\Theta_s \dd s - \frac \beta 2 \intot  R_s^{-2}\pi_{\Theta_s^\perp}
\frac{\nabla\gamma(\Theta_s)}{\gamma(\Theta_s)}\dd s.
\end{align}
Applying the It\^o formula, we find, setting $V_t=R_t\Theta_t$ as in the statement,
\begin{align*}
V_t =& v_0 + \intot \Theta_s \dd \tB_s + \intot \pi_{\Theta_s^\perp} \dd \bB_s
+ \intot \Big(\frac{d-1} {2R_s} - \frac{\beta}{2}\frac{\Gamma'(R_s)}{\Gamma(R_s)} \Big)\Theta_s \dd s \\
&- \intot \Big(\frac{d-1}{2R_s}\Theta_s + \frac{\beta}2\pi_{\Theta_s^\perp}
\frac{\nabla\gamma(\Theta_s)}{R_s\gamma(\Theta_s)}\Big) \dd s\\
=& v_0 + B_t  - \frac{\beta}{2}\intot \Big(\frac{\Gamma'(R_s)}{\Gamma(R_s)}\Theta_s
+ \pi_{\Theta_s^\perp}
\frac{\nabla\gamma(\Theta_s)}{R_s\gamma(\Theta_s)} \Big) \dd s,
\end{align*}
where we have set $B_t=\intot \Theta_s \dd \tB_s + \intot \pi_{\Theta_s^\perp} \dd \bB_s$.
This is a $\rd$-valued $(\cG_t)_{t\geq 0}$-martingale with quadratic variation matrix  $\intot [\Theta_s\Theta_s^*+
\pi_{\Theta_s^\perp}]\dd s =I_d t$ and thus a Brownian motion.
It only remains to verify that, for $v=r\theta$ with $r>0$ and $\theta \in \Sd$, one has
\begin{equation}\label{fff}
F(v)=[\Gamma(r)]^{-1}\Gamma'(r) \theta+ [r\gamma(\theta)]^{-1}\pi_{\theta^\perp}\nabla\gamma(\theta),
\end{equation}
which follows from the fact that $F=\nabla [\log U]$ with 
$U(v)=\Gamma(|v|)\gamma(v/{|v|})$.
\end{proof}

We next build the radial process using classical tools, namely speed measures and scale functions, see
Revuz-Yor \cite[Chapter VII, Paragraph 3]{ry}. 

\begin{nota}\label{notaf} 
Fix $\beta>d-2$. We introduce $h(r)=(\beta+2-d)\int_{r_0}^r u^{1-d}[\Gamma(u)]^\beta \dd u$, which is
an increasing bijection from $(0,\infty)$ into $\rr$.
We denote by $h^{-1}:\rr \mapsto (0,\infty)$ its inverse function, for which $h^{-1}(0)=r_0$.
We also introduce $\sigma(w)=h'(h^{-1}(w))$ from $\rr$ to $(0,\infty)$ and $\psi(w)=[\sigma(w) h^{-1}(w)]^{2}$
from $\rr$ to $(0,\infty)$
\end{nota}

In the following statement, we introduce a parameter $\e\in(0,1)$, which may seem artificial 
at this stage, but this will be crucial to work by coupling.

\begin{lem}\label{exr} 
Fix $\beta>d-2$ and consider a Brownian motion $(W_t)_{t\geq 0}$. For $\e\in (0,1)$ and $a_\e>0$, 
introduce $A^\e_t=\e a_\e^{-2}\intot [\sigma(W_s/a_\e)]^{-2}\dd s$ and its inverse $\rho^\e_t$.
Set $R^\e_t = \sqrt \e h^{-1}(W_{\rho^\e_t}/a_\e)$.
For each $\e\in(0,1)$, the process $(S^\e_t=\e^{-1/2}R^\e_{\e t})_{t\geq 0}$
is $(0,\infty)$-valued and is a weak solution to \eqref{eqr}.
\end{lem}

This can be rephrased as follows: $(R^\e_t)_{t\geq 0}$ has the same law as $(\sqrt\e R_{t/\e})_{t\geq 0}$,
with $(R_t)_{t\geq 0}$ solving \eqref{eqr}. Of course, $(\sqrt\e R_{t/\e})_{t\geq 0}$ is a natural object
when studying the large time behavior of $(R_t)_{t\geq 0}$.

\begin{proof}
First, $(S^\e_t)_{t\geq 0}$ is $(0,\infty)$-valued by definition.
Next, there classically exists a Brownian motion $(\bB_t)_{t\geq 0}$, 
see e.g. Revuz-Yor \cite[Proposition 1.13 p 373]{ry}, such that 
$Y_t^\e=W_{\rho^\e_t}$ solves $Y^\e_t=\e^{-1/2}a_\e \int_0^t\sigma(Y^\e_s/a_\e) \dd \bB_s$, whence 
$Z^\e_t=a_\e^{-1}Y^\e_t=\e^{-1/2}\intot\sigma(Z^\e_s)\dd \bB_s$. Thus 
$$
R^\e_t=\sqrt \e h^{-1}(Z^\e_t)=\sqrt \e h^{-1}(0)+ \intot (h^{-1})'(Z^\e_s)\sigma(Z^\e_s)\dd \bB_s  
+ \frac1{2\sqrt \e}\intot (h^{-1})''(Z^\e_s)\sigma^2(Z^\e_s)\dd s.
$$
But $h^{-1}(0)=r_0$, $(h^{-1})'(z)\sigma(z)=1$ and $(h^{-1})''(z)\sigma^2(z)=-\sigma'(z)=
-h''(h^{-1}(z))/h'(h^{-1}(z))
= \frac{d-1}{h^{-1}(z)} - \beta\frac{\Gamma'(h^{-1}(z))}{\Gamma(h^{-1}(z))}$ because
$h''(u)/h'(u)=[\log (u^{1-d}\Gamma^\beta(u))]'=(1-d)/u + \beta \Gamma'(u)/\Gamma(u)$. Hence
\begin{align*}
R^\e_t=&\sqrt \e r_0 + \bB_t + \frac {d-1}{2\sqrt \e} \intot \frac1{h^{-1}(Z^\e_s)} \dd s
- \frac{\beta}{2\sqrt\e} \intot \frac{\Gamma'(h^{-1}(Z^\e_s))}{\Gamma(h^{-1}(Z^\e_s))}\dd s\\
=& \sqrt \e r_0 + \bB_t + \frac {d-1}{2} \intot \frac{\dd s}{R^\e_s} 
- \frac{\beta}{2} \intot \frac{\Gamma'(R^\e_s/\sqrt\e)}{\sqrt\e \Gamma(R^\e_s/\sqrt\e)}\dd s.
\end{align*}
We conclude that $S^\e_t=\e^{-1/2}R^\e_{\e t}$ solves \eqref{eqr} with the Brownian motion
$\tB_t=\e^{-1/2}\bB_{\e t}$.
\end{proof}

Finally, we can give the

\begin{proof}[Proof of Proposition \ref{eui}]
The global weak existence of a $\rd\setminus\{0\}$-valued solution proved in Lemma \ref{rep},
together with the local strong existence and pathwise uniqueness (until the velocity process reaches
$0$ or explodes to infinity), which classically follows from the fact that the drift $F$ is of class
$C^1$ on $\rd\setminus\{0\}$, implies the global strong existence and pathwise uniqueness for \eqref{rep}.
\end{proof}

\section{The stable regime}\label{Ss}

Here we prove Theorem \ref{mr}-(c)-(d)-(e)-(f). 
We introduce some notation that will be used during the whole section.
We fix $\beta\in [d,4+d)$ and set $\alpha=(\beta+2-d)/3$.
We recall Notation \ref{notaf}. We fix $\e\in (0,1)$ and introduce 
$$
a_\e=\kappa \e \quad \hbox{if} \quad  \beta\in(d,4+d) \quad \hbox{and} \quad a_\e=\frac{\e|\log\e|}4
\quad \hbox{if} \quad \e=d,
$$
where $\kappa = \int_\rr [\sigma(w)]^{-2}\dd w<\infty$ when $\beta>d$, see Lemma \ref{fcts}-(i).
We consider a one-dimensional Brownian motion $(W_t)_{t\geq 0}$, set 
$A^\e_t=\e a_\e^{-2}\intot [\sigma(W_s/a_\e)]^{-2}\dd s$, introduce its inverse $\rho^\e_t$ and put
$R^\e_t = \sqrt \e h^{-1}(W_{\rho^\e_t}/a_\e)$. We know from Lemma \ref{exr} that
$S^\e_t=\e^{-1/2}R^\e_{\e t}= h^{-1}(W_{\rho^\e_{\e t}}/a_\e)$ 
solves \eqref{eqr}. We also consider the solution $(\hTheta_t)_{t\geq 0}$ of 
\eqref{eqt}, independent of $(W_t)_{t\geq 0}$.

\begin{lem}\label{ettac}
For each $\e\in (0,1)$, $(X_{t/\e}-x_0)_{t\geq 0}\stackrel{d}=(\tX^\e_t)_{t\geq 0}$, where
\begin{equation}\label{ettacez}
\tX^\e_t=\frac1{a_\e^2}\int_0^{\rho^\e_t} \frac{h^{-1}(W_u/a_\e)\hTheta_{T^\e_u}}{[\sigma(W_u/a_\e)]^2}\dd u
\quad \hbox{where} \quad T^\e_t = \frac{1}{a_\e^2}\int_0^{t} 
\frac{\dd s }{\psi(W_s/a_\e)}.
\end{equation}
Furthermore, for any $m\in \rd$, any $t\geq 0$, it holds that 
\begin{equation}\label{ettace}
\tX_{t/\e}-m t/\e = \frac1{a_\e^2}\int_0^{\rho^\e_t} \frac{h^{-1}(W_u/a_\e)\hTheta_{T^\e_u}-m}{[\sigma(W_u/a_\e)]^2}\dd u.
\end{equation}
\end{lem}

\begin{proof}
We know from Lemma \ref{rep} that, setting $H_t^\e=\intot [S_s^\e]^{-2} \dd s$,
$(S^\e_t \hTheta_{H^\e_t})_{t\geq 0}\stackrel{d}=(V_t)_{t\geq 0}$. Recalling that 
$X_t-x_0=\intot V_s \dd s$, we conclude that  $(X_{t/\e}-x_0)_{t\geq 0}\stackrel{d}=(\tX^\e_t)_{t\geq 0}$,
where $\tX^\e_t=\int_0^{t/\e} S^\e_s \hTheta_{H^\e_s}\dd s = 
\int_0^{t/\e} h^{-1}(W_{\rho^\e_{\e s}}/a_\e)\hTheta_{H^\e_s}\dd s$. 
Performing the substitution $u=\rho^\e_{\e s}$, i.e. $s=\e^{-1}A^\e_u$, 
whence $\dd s = a_\e^{-2}[\sigma(W_u/a_\e)]^{-2}\dd u$, we find
$$
\tX^\e_t=\frac{1}{a_\e^2}\int_0^{\rho^\e_t} \frac{h^{-1}(W_u/a_\e) \hTheta_{H^\e_{\e^{-1}A^\e_u}}}
{[\sigma(W_u/a_\e)]^2}\dd u.
$$
Using the same change of variables, one verifies  that
$$
H^\e_{\e^{-1}A^\e_t}=\int_0^{\e^{-1}A^\e_t}\frac{\dd s}{[h^{-1}(W_{\rho^\e_{\e s}}/a_\e)]^{2}} = \frac{1}{a_\e^2}
\int_0^t \frac{\dd u}{[\sigma(W_u/a_\e)]^2[h^{-1}(W_u/a_\e)]^2}=\frac{1}{a_\e^2}\int_0^{t} 
\frac{\dd u }{\psi(W_u/a_\e)}
$$
as desired. The last claim follows from the fact that
$a_\e^{-2}\int_0^{\rho^\e_t} [\sigma(W_u/a_\e)]^{-2}\dd u = \e^{-1} A^\e_{\rho^\e_t}= \e^{-1}t$.
\end{proof}

We first study the convergence of the time-change.

\begin{lem}\label{tl}
(i) For all $T>0$, a.s., $\sup_{[0,T]}|A^\e_t-L^0_t| \to 0$ as $\e\to 0$, where $(L^0_t)_{t\geq 0}$
is the local time at $0$ of $(W_t)_{t\geq 0}$.

\vip

(ii) For all $t\geq 0$, a.s., $\rho^\e_t \to \tau_t=\inf\{u\geq 0 : L^0_u>t\}$,
the generalized inverse of $(L^0_s)_{s\geq 0}$.
\end{lem}

\begin{proof}
Point (ii) follows from point (i) by Lemma \ref{tc} and since $\PR(\tau_t\neq \tau_{t-})=0$.
Concerning point (i), we first assume that $\beta>d$.
Since $a_\e=\kappa \e$, 
the occupation times formula, see Revuz-Yor \cite[Corollary 1.6 p 224]{ry}, gives us
$$
A^{\e}_t=\frac{\e}{a_\e^2}\intot \frac{\dd s}{[\sigma(W_s/a_\e)]^{2}}=
\frac1{\kappa^2 \e}\int_\rr \frac{L^x_t \dd x}{\sigma^2(x/(\kappa \e))}=
\int_\rr \frac{L^{\kappa \e y }_t \dd y}{\kappa \sigma^2(y)},
$$
where $(L^x_t)_{t\geq 0}$ is the local time of $(W_t)_{t\geq 0}$ at $x$.
Recalling that $\kappa = \int_\rr [\sigma(w)]^{-2}\dd w$, which is finite by Lemma \ref{fcts}-(i),
we may write
$$
|A^\e_t - L^0_t| \leq 
\int_\rr \frac{|L^{\kappa \e y}_t-L^0_t| \dd y}{\kappa \sigma^2(y)},
$$
which a.s. tends uniformly (on $[0,T]$) to $0$ as $\e\to 0$ 
by dominated convergence, since $\sup_{[0,T]}|L^{\kappa \e y}_t-L^0_t|$ a.s. tends to $0$
for each fixed $y$ by \cite[Corollary 1.8 p 226]{ry} and since $\sup_{[0,T]\times \rr} L^{x}_t<\infty$
a.s.

\vip

We next treat the case where $\beta=d$, which is more complicated.
We recall that $a_\e=\e|\log \e|/4$.
By Lemma \ref{fcts}-(vi)-(vii), we know that $[\sigma(w)]^{-2} \leq C(1+|w|)^{-1}$ and that
\begin{equation}\label{eqlog2}
\int_{-x}^x \frac{\dd w}{[\sigma(w)]^2 } \stackrel{x\to \infty}\sim \frac{\log x}{4}.
\end{equation}
We fix $\delta>0$ and write $A^\e_t=J^{\e,\delta}_t+Q^{\e,\delta}_t$, where
$$
J^{\e,\delta}_t=\frac{\e}{a_\e^2}\int_0^t \frac{\indiq_{\{|W_s|> \delta\}}\dd s}
{[\sigma(W_s/a_\e)]^2} \quad\hbox{and}\quad
Q^{\e,\delta}_t=\frac{\e}{a_\e^2}\int_0^t \frac{\indiq_{\{|W_s|\leq \delta\}}\dd s}
{[\sigma(W_s/a_\e)]^2}.
$$
One checks that $\sup_{[0,T]}J^{\e,\delta}_t\leq C T \e /[a_\e^2(1+\delta/a_\e)]\leq C T \e/(\delta a_\e)$, 
which tends to $0$ as $\e\to 0$.
We next use the occupation times formula, see Revuz-Yor \cite[Corollary 1.6 p 224]{ry}, to write
\begin{align*}
Q^{\e,\delta}_t= \frac{\e}{a_\e^2}\int_{-\delta}^\delta \frac{L^x_t \dd x}
{[\sigma(x/a_\e)]^2} =\frac{\e}{a_\e^2}\int_{-\delta}^\delta \frac{\dd x}
{[\sigma(x/a_\e)]^2}L^0_t + \frac{\e}{a_\e^2}\int_{-\delta}^\delta \frac{(L^x_t-L^0_t) \dd x}
{[\sigma(x/a_\e)]^2}
=r_{\e,\delta} L_t^0+ R^{\e,\delta}_t,
\end{align*}
the last identity standing for a definition.
But a substitution and \eqref{eqlog2} allow us to write  
$$
r_{\e,\delta}=\frac{\e}{a_\e} \int_{-\delta/{a_\e}}^{\delta/{a_\e}} \frac{\dd y}
{[\sigma(y)]^2} \stackrel{\e\to 0}\sim
\frac{\e \log (\delta/a_\e)}{4 a_\e} \longrightarrow 1
$$
as $\e\to 0$. All this proves that a.s., for all $\delta>0$,
$\limsup_{\e\to 0}  \sup_{[0,T]}|A^\e_t - L^0_t|\leq \limsup_{\e\to 0}  \sup_{[0,T]}|R^{\e,\delta}_t|$.
But $|R^{\e,\delta}_t| \leq r_{\e,\delta} \times \sup_{[-\delta,\delta]}|L^x_t-L^0_t|$, whence
$\limsup_{\e\to 0}  \sup_{[0,T]}|A^\e_t - L^0_t|\leq \sup_{[0,T]\times[-\delta,\delta]}|L^x_t-L^0_t|$ a.s.,
and it suffices to let $\delta\to 0$, using Revuz-Yor \cite[Corollary 1.8 p 226]{ry}, to complete the proof.
\end{proof}

We next proceed to three first approximations: in the formula \eqref{ettace}, we show that
one may replace $\rho^\e_t$ by its limiting value $\tau_t$, that the negative values of
$W$ have a negligible influence, and that we may  introduce a cutoff that will allow us to neglect
the small jumps of the limiting stable process. All this is rather tedious in 
the infinite variation case $\alpha \in [1,2)$.
We recall that $m'_\beta>0$, $M_\beta\in\rd$ and $m_\beta=m'_\beta M_\beta$  were defined in 
Subsection \ref{notadebut}.

\begin{nota}\label{fap}
(i) If $\beta \in [d,1+d)$, we introduce, for $\delta\in(0,1]$ and $\e\in(0,1)$,
$$
Z^{\e,\delta}_t=a_\e^{1/\alpha-2}\int_0^{\tau_t} \frac{h^{-1}(W_u/a_\e)\hTheta_{T^\e_u}}{[\sigma(W_u/a_\e)]^2}
\indiq_{\{W_u>\delta \}}\dd u
\quad \hbox{and}\quad U^{\e,\delta}_t=a_\e^{1/\alpha}\tX^\e_t - Z^{\e,\delta}_t.
$$

(ii) If $\beta=1+d$, we put 
$$
\zeta_\e=\frac{\int_{-\infty}^1 h^{-1}(w/a_\e)[\sigma(w/a_\e)]^{-2}\dd w}{\int_{-\infty}^1 [\sigma(w/a_\e)]^{-2}\dd w}
=\frac{\int_{-\infty}^{1/a_\e} h^{-1}(w)[\sigma(w)]^{-2}\dd w}{\int_{-\infty}^{1/a_\e} [\sigma(w)]^{-2}\dd w}
$$
(so that $\kappa_{\e,1}$ defined below vanishes)
and we introduce, for $\delta\in(0,1]$ and $\e\in(0,1)$, 
\begin{gather*}
Z^{\e,\delta}_t=\frac 1 {a_\e}\int_0^{\tau_t} \frac{h^{-1}(W_u/a_\e)\hTheta_{T^\e_u}-\zeta_\e M_\beta}{[\sigma(W_u/a_\e)]^2}
\indiq_{\{W_u>\delta \}}\dd u, \quad
\kappa_{\e,\delta}=\frac 1 {a_\e}\int_{-\infty}^{\delta} \frac{h^{-1}(w/a_\e)- \zeta_\e}
{[\sigma(w/a_\e)]^2}\dd w,\\
\hbox{and}\quad U^{\e,\delta}_t=a_\e[\tX^\e_t- \zeta_\e M_\beta t/\e] - Z^{\e,\delta}_t- \kappa_{\e,\delta}M_\beta t.
\end{gather*}

(iii) If $\beta \in (1+d,4+d)$
we introduce, for $\delta\in(0,1]$ and $\e\in(0,1)$, 
\begin{gather*}
Z^{\e,\delta}_t=a_\e^{1/\alpha-2}\int_0^{\tau_t} \frac{h^{-1}(W_u/a_\e)\hTheta_{T^\e_u}-m_\beta}{[\sigma(W_u/a_\e)]^2}
\indiq_{\{W_u>\delta \}}\dd u, \quad
\kappa_{\e,\delta}=a_\e^{1/\alpha-2}\int_{-\infty}^{\delta} \frac{h^{-1}(w/a_\e)-m'_\beta}
{[\sigma(w/a_\e)]^2}\dd w,\\
\hbox{and}\quad U^{\e,\delta}_t=a_\e^{1/\alpha}[\tX^\e_t-m_\beta t/\e] - Z^{\e,\delta}_t- \kappa_{\e,\delta}M_\beta t.
\end{gather*}
\end{nota}

Observe that $\zeta_\e$ and $\kappa_{\delta,\e}$ are well defined by Lemma \ref{fcts}-(i)-(viii).

\begin{lem}\label{fapl}
For all $\beta \in [d,4+d)$, all $t\geq 0$, all $\eta>0$, 
$\lim_{\delta\to 0} \limsup_{\e \to 0}\PR[|U^{\e,\delta}_t|>\eta]=0$.
\end{lem}

\begin{proof}
\underline{\it Case (i):} $\beta \in [d,1+d)$, whence $\alpha \in [2/3,1)$. Recalling \eqref{ettacez}, we see that
$$
U^{\e,\delta}_t=a_\e^{1/\alpha-2}\int_0^{\tau_t} \frac{h^{-1}(W_u/a_\e)\hTheta_{T^\e_u}}{[\sigma(W_u/a_\e)]^2}
\indiq_{\{W_u\leq \delta \}}\dd u
+ a_\e^{1/\alpha-2} \int_{\tau_t}^{\rho^\e_t} 
\frac{h^{-1}(W_u/a_\e)\hTheta_{T^\e_u}}{[\sigma(W_u/a_\e)]^2}\dd u.
$$
Since $h^{-1}(w)[\sigma(w)]^{-2} \leq C (1+w)^{1/\alpha-2}\indiq_{\{w\geq 0\}}+C(1+|w|)^{-2}
\indiq_{\{w< 0\}}$ by Lemma \ref{fcts}-(viii), 
$$
a_\e^{1/\alpha-2}h^{-1}(w/a_\e)[\sigma(w/a_\e)]^{-2} 
\leq C w^{1/\alpha-2}\indiq_{\{w\geq 0\}}+C |w|^{1/\alpha-2}(1+|w|/a_\e)^{-1/\alpha}\indiq_{\{w< 0\}} \leq C|w|^{1/\alpha-2},
$$ 
and thus
\begin{align*}
|U^{\e,\delta}_t| \leq& C \int_0^{\tau_t} W_u^{1/\alpha-2}\indiq_{\{0\leq W_u\leq \delta \}}\dd u 
+C \int_0^{\tau_t} |W_u|^{1/\alpha-2}(1+|W_u|/a_\e)^{-1/\alpha}\indiq_{\{W_u<0\}}\dd u \\
&+ C \int_{\tau_t}^{\rho^\e_t}  |W_u|^{1/\alpha-2} \dd u.
\end{align*}
But $1/\alpha-2>-1$, so that the integral $\int_0^T |W_u|^{1/\alpha-2}\dd u$ is a.s. finite for all $T>0$
(because its expectation is finite). One concludes by dominated convergence, using 
that $\rho^\e_t\to \tau_t$ a.s. for each $t\geq 0$ fixed by Lemma \ref{tl}-(ii),
that a.s.,
$$
\lim_{\delta \to 0}\limsup_{\e \to 0}|U^{\e,\delta}_t| \leq C\lim_{\delta\to 0}
\int_0^{\tau_t} W_u^{1/\alpha-2}\indiq_{\{0\leq W_u\leq \delta \}}\dd u=0.
$$

\underline{\it Case (iii): $\beta \in (1+d,4+d)$.} This is much more complicated. 
Recalling \eqref{ettace}, we have
\begin{align*}
U^{\e,\delta}_t=& a_\e^{1/\alpha-2}\int_0^{\tau_t}
\frac{h^{-1}(W_u/a_\e)\hTheta_{T^\e_u}-m_\beta}{[\sigma(W_u/a_\e)]^2}
\indiq_{\{W_u\leq \delta \}}\dd u - \kappa_{\e,\delta}M_\beta t \\
&+ a_\e^{1/\alpha-2}\int_{\tau_t}^{\rho^\e_t} 
\frac{h^{-1}(W_u/a_\e)\hTheta_{T^\e_u}-m_\beta}{[\sigma(W_u/a_\e)]^2}\dd u\\
=& K^{\e,\delta}_{\tau_t} + M_\beta I^{\e,\delta}_{\tau_t} + [K^{\e,\infty}_{\rho^\e_t}-K^{\e,\infty}_{\tau_t}] + M_\beta 
[I^{\e,\infty}_{\rho^\e_t}-I^{\e,\infty}_{\tau_t}],
\end{align*}
where we have set (extending the definition of $\kappa_{\e,\delta}$ to all values of $\delta\in (0,\infty]$),
\begin{align*}
K^{\e,\delta}_t=&a_\e^{1/\alpha-2}\int_0^{t} \frac{h^{-1}(W_u/a_\e)[\hTheta_{T^\e_u}-M_\beta]}{[\sigma(W_u/a_\e)]^2}
\indiq_{\{W_u\leq \delta \}}\dd u ,\\
I^{\e,\delta}_t=&a_\e^{1/\alpha-2}\int_0^{t} \frac{h^{-1}(W_u/a_\e)-m'_\beta}{[\sigma(W_u/a_\e)]^2}
\indiq_{\{W_u\leq \delta \}}\dd u - \kappa_{\e,\delta} L^0_t.
\end{align*}
We used that $m_\beta=m_\beta'M_\beta$, that $L^0_{\tau_t}=t$ and that 
\begin{equation}\label{keiz}
\kappa_{\e,\infty}= a_\e^{1/\alpha-2}\int_{-\infty}^{\infty} \frac{h^{-1}(w/a_\e)-m'_\beta}{[\sigma(w/a_\e)]^{2}}\dd w =
a_\e^{1/\alpha-1}\int_{-\infty}^{\infty} \frac{h^{-1}(y)-m'_\beta}{[\sigma(y)]^{2}}\dd y =0
\end{equation}
by Lemma \ref{fcts}-(ii).

\vip

We first treat $I$.
By the occupation times formula, see 
Revuz-Yor \cite[Corollary 1.6 p 224]{ry}, and by definition of $\kappa_{\e,\delta}$,
$$
I^{\e,\delta}_t=a_\e^{1/\alpha-2}\int_{-\infty}^\delta \frac{h^{-1}(w/a_\e)-m'_\beta}{[\sigma(w/a_\e)]^2}
(L^w_t-L^0_t) \dd w.
$$
For each $\delta\in (0,\infty]$, 
each $T\geq 0$, we a.s. have $\lim_{\e\to 0} \sup_{[0,T]} |I^{\e,\delta}_t-I^{\delta}_t|=0$,
where we have set
$I^\delta_t=(\beta+2-d)^{-2}\int_0^\delta w^{1/\alpha-2}(L^w_t-L^0_t)\dd w$.
Indeed, this follows from dominated convergence, because 

\vip

\noindent $\bullet$ $a_\e^{1/\alpha-2} |h^{-1}(w/a_\e)-m_\beta'|[\sigma(w/a_\e)]^{-2}
\leq C |w|^{1/\alpha-2}$ by Lemma \ref{fcts}-(viii),

\vip

\noindent $\bullet$ $\lim_{\e\to 0} a_\e^{1/\alpha-2} [h^{-1}(w/a_\e)-m_\beta'][\sigma(w/a_\e)]^{-2}=
(\beta+2-d)^{-2}w^{1/\alpha-2}\indiq_{\{w\geq 0\}}$, see Lemma \ref{fcts}-(ix),

\vip

\noindent $\bullet$ a.s., 
$\int_\rr |w|^{1/\alpha-2} \sup_{[0,T]}|L^w_t-L^0_t|\dd w<\infty$, since $1/\alpha-2\in (-3/2,-1)$
and since $\sup_{[0,T]} |L^w_t-L^0_t|$ is a.s. bounded and almost $1/2$-Hold\"er continuous (as a function of
$w$), see \cite[Corollary 1.8 p 226]{ry}.

\vip

We conclude that $\lim_{\delta\to 0} \limsup_{\e\to 0} |I^{\e,\delta}_{\tau_t}|
=\lim_{\delta\to 0} |I^{\delta}_{\tau_t}|=0$ a.s. and, 
using that $\rho^\e_t\to \tau_t$ a.s. by Lemma \ref{tl}-(ii) (for each fixed $t\geq 0$)
and that $t\mapsto I^{\infty}_t$ is a.s.
continuous on $[0,\infty)$, that $\lim_{\delta\to 0} \limsup_{\e\to 0} |I^{\e,\infty}_{\rho^\e_t}-I^{\e,\infty}_{\tau_t} |=0$
a.s. All this proves that a.s.,
$$
\lim_{\delta\to 0} \limsup_{\e\to 0} [|I^{\e,\delta}_{\tau_t}| + |I^{\e,\infty}_{\rho^\e_t}-I^{\e,\infty}_{\tau_t} |]=0.
$$

We next treat $K$. We mention at once that all the computations below 
concerning $K$ are also valid when $\beta=1+d$, i.e. $\alpha=1$.
We introduce $\cW=\sigma(W_t,t\geq 0)$. Assume for a moment that there is $C>0$ such that
for any $\delta\in (0,\infty]$, any $\e\in(0,1)$, any $0\leq s \leq t$, a.s.,
\begin{equation}\label{obj}
\E[(K^{\e,\delta}_t-K^{\e,\delta}_s)^2|\cW]\leq C \int_s^t |W_u|^{2/\alpha-2}
\Big[\indiq_{\{0\leq W_u\leq \delta\}}+(1+|W_u|/\e)^{-1/\alpha}\Big] \dd u.
\end{equation}
Then, $\tau_t$ and $\rho^\e_t$ being $\cW$-measurable, we will deduce that
\begin{align*}
\E\big[ (K^{\e,\delta}_{\tau_t})^2+(K^{\e,\infty}_{\tau_t}-K^{\e,\infty}_{\rho^\e_t})^2 | \cW\big] \leq& 
C \int_0^{\tau_t} |W_u|^{2/\alpha-2} \Big[\indiq_{\{0\leq W_u\leq \delta\}}+(1+|W_u|/\e)^{-1/\alpha}\Big] \dd u \\
&+ C\Big|\int_{\rho^\e_t}^{\tau_t}|W_u|^{2/\alpha-2} \dd u \Big|.
\end{align*}
Since $\int_0^T |W_u|^{2/\alpha-2}\dd u <\infty$ a.s. for all $T>0$ because $2/\alpha-2>-1$ and since
$\rho^\e_t\to \tau_t$ a.s. (for $t\geq 0$ fixed) by Lemma \ref{tl}-(ii), conclude, by dominated convergence
that a.s.,
$$
\lim_{\delta\to 0} \limsup_{\e\to 0}
\E[(K^{\e,\delta}_{\tau_t})^2+(K^{\e,\infty}_{\tau_t}-K^{\e,\infty}_{\rho^\e_t})^2 |\cW]=0,
$$
from which the convergence
$\lim_{\delta\to 0} \limsup_{\e\to 0} |K^{\e,\delta}_{\tau_t}|+|K^{\e,\infty}_{\tau_t}-K^{\e,\infty}_{\rho^\e_t}|=0$ in probability 
follows.

\vip

We now check \eqref{obj}, starting from 
$$
(K^{\e,\delta}_t-K^{\e,\delta}_s)^2=a_\e^{2/\alpha-4}\int_s^t \int_s^t \frac{h^{-1}(W_a/a_\e)}{[\sigma(W_a/a_\e)]^2}
\frac{h^{-1}(W_b/a_\e)}{[\sigma(W_b/a_\e)]^2}
\indiq_{\{W_a\leq \delta\}}
\indiq_{\{W_b\leq \delta\}} (\hTheta_{T^\e_a}-M_\beta)(\hTheta_{T^\e_b}-M_\beta)\dd a \dd b.
$$
Since $(T^\e_t)_{t\geq 0}$ is $\cW$-measurable, since $(\hTheta_t)_{t\geq 0}$ is independent of $\cW$,
and since $M_\beta=\int_{\Sd} \theta \nu_\beta(\dd \theta)$, Lemma \ref{ergo2}-(ii) (and the Markov property) 
tells us that
there are $C>0$ and $\lambda>0$ such that
\begin{align*}
|\E([\hTheta_{T^\e_a}-M_\beta][\hTheta_{T^\e_b}-M_\beta]|\cW)|
\leq C \exp(-\lambda |T^\e_b-T^\e_a|).
\end{align*}
By Lemma \ref{fcts}-(viii) and since $a_\e=\kappa \e$, 
we have 
$$
a_\e^{1/\alpha-2}\frac{h^{-1}(w/a_\e)}{[\sigma(w/a_\e)]^2} \leq C (\e+|w|)^{1/\alpha-2}[\indiq_{\{w\geq 0\}}+
(1+|w|/\e)^{-1/\alpha}\indiq_{\{w<0\}}], 
$$
whence
\begin{align*}
\E[(K^{\e,\delta}_t-&K^{\e,\delta}_s)^2|\cW]
\leq  C \int_s^t \int_s^t (\e+|W_a|)^{1/\alpha-2}(\e+|W_b|)^{1/\alpha-2}\\
&\Big[\indiq_{\{0<W_a\leq \delta\}}+(1+|W_a|/\e)^{-1/\alpha}\Big]\Big[\indiq_{\{0\leq W_b\leq \delta\}}+(1+|W_b|/\e)^{-1/\alpha}\Big]
\exp(-\lambda |T^\e_a-T^\e_b|) \dd a \dd b.
\end{align*}
Next, we observe that, since $a_\e^2\psi(w/a_\e) \leq C(\e+|w|)^2$ by Lemma \ref{fcts}-(iv),
$$
\lambda |T^\e_a-T^\e_b|= \lambda \Big|\frac{1}{a_\e^2}\int_a^b
\frac{\dd s }{\psi(W_s/a_\e)}\dd s \Big|\geq c \Big|\int_a^b (\e+|W_s|)^{-2}\dd s \Big|
$$
for some $c>0$. Using that $(xy)^{1/\alpha}\leq x^{2/\alpha}+y^{2/\alpha}$ 
and a symmetry argument, we conclude that
\begin{align*}
\E[(K^{\e,\delta}_t-K^{\e,\delta}_s)^2|\cW]
\leq& C\int_s^t (\e+|W_b|)^{2/\alpha-2}\Big[\indiq_{\{0\leq W_b\leq \delta\}}+(1+|W_b|/\e)^{-1/\alpha}\Big]\\
&\hskip3cm \int_s^t  (\e+|W_a|)^{-2}
\exp\Big(  -   c\Big|\int_a^b (\e+|W_s|)^{-2}\dd s \Big|\Big)\dd a \dd b\\
\leq & C\int_s^t |W_b|^{2/\alpha-2}\Big[\indiq_{\{0\leq W_b\leq \delta\}}+(1+|W_b|/\e)^{-1/\alpha}\Big] \dd b
\end{align*}
as desired. 
We finally used that for all $b\in [0,t]$, all continuous function $\varphi:\rr_+\mapsto \rr_+$,
\begin{align}
\int_0^t \varphi(a)\exp\Big(-\Big|\int_a^b \varphi(s)\dd s\Big|\Big)\dd a \leq 2. \label{magic}
\end{align}

\vip

\underline{\it Case (ii): $\beta=1+d$.} Applying \eqref{ettace} with $m=\zeta_\e M_\beta$, we see that
\begin{align*}
U^{\e,\delta}_t=&\frac 1 {a_\e}\int_0^{\tau_t} \frac{h^{-1}(W_u/a_\e)\hTheta_{T^\e_u}-\zeta_\e M_\beta}
{[\sigma(W_u/a_\e)]^2} \indiq_{\{W_u\leq \delta \}}\dd u -\kappa_{\e,\delta}M_\beta t
+ \frac 1 {a_\e}\int_{\tau_t}^{\rho^\e_t} \frac{h^{-1}(W_u/a_\e)\hTheta_{T^\e_u}-\zeta_\e M_\beta}
{[\sigma(W_u/a_\e)]^2}\dd u\\
=& K^{\e,\delta}_{\tau_t} + [K^{\e,\infty}_{\rho^\e_t}-K^{\e,\infty}_{\tau_t}] + M_\beta 
I^{\e,\delta}_{\tau_t} + M_\beta [I^{\e,\infty}_{\rho^\e_t}-I^{\e,\infty}_{\tau_t}],
\end{align*}
where we have set, for $\delta\in (0,1)\cup\{\infty\}$, with the convention that $\kappa_{\e,\infty}=0$,
\begin{align*}
K^{\e,\delta}_t=&\frac1{a_\e}\int_0^{t} \frac{h^{-1}(W_u/a_\e)[\hTheta_{T^\e_u}-M_\beta]}{[\sigma(W_u/a_\e)]^2}
\indiq_{\{W_u\leq \delta \}}\dd u ,\\
I^{\e,\delta}_t=&\frac1{a_\e}\int_0^{t} \frac{h^{-1}(W_u/a_\e)-\zeta_\e}{[\sigma(W_u/a_\e)]^2}
\indiq_{\{W_u\leq \delta \}}\dd u - \kappa_{\e,\delta} L^0_t.
\end{align*}

Exactly as in Case (iii), 
$\lim_{\delta\to 0} \limsup_{\e\to 0}[|K^{\e,\delta}_{\tau_t}|+|K^{\e,\infty}_{\tau_t}-K^{\e,\infty}_{\rho^\e_t}|]=0$
in probability.

\vip

We also have, for any $\delta \in (0,1)\cup\{\infty\}$, by definition of $\kappa_{\e,\delta}$
(in particular since $\kappa_{\e,1}=\kappa_{\e,\infty}=0$),
$$
I^{\e,\delta}_t=\frac1{a_\e}\int_{-\infty}^\delta \frac{h^{-1}(w/a_\e)-\zeta_\e}{[\sigma(w/a_\e)]^2}
(L^w_t-L^0_t\indiq_{\{w\leq 1\}}) \dd w.
$$
As in Case (iii), it is sufficient to verify that for each $\delta\in (0,1)\cup\{\infty\}$,
each $T\geq 0$, we a.s. have $\lim_{\e\to 0} \sup_{[0,T]} |I^{\e,\delta}_t-I^{\delta}_t|=0$,
where we have set
$I^\delta_t=9^{-2}\int_0^\delta w^{-1}(L^w_t-L^0_t\indiq_{\{w\leq 1\}})\dd w$.
This, here again, follows from dominated convergence, because, recalling
that $a_\e = \kappa \e$,

\vip

\noindent $\bullet$ $a_\e^{-1} h^{-1}(w/a_\e)[\sigma(w/a_\e)]^{-2}
\leq C w^{-1}\indiq_{\{w\geq 0\}}+ C|w|^{-1}(1+|w|)^{-1 }\indiq_{\{w<0\}}$ by Lemma \ref{fcts}-(viii),

\vip 

\noindent $\bullet$ $\lim_{\e\to 0}  a_\e^{-1} h^{-1}(w/a_\e)/[\sigma(w/a_\e)]^2 = 
9^{-1}w^{-1}\indiq_{\{w\geq 0\}}$, see Lemma \ref{fcts}-(ix),

\vip

\noindent $\bullet$ $\zeta_\e \leq C \int_{-\infty}^{1/a_\e} h^{-1}(w)[\sigma(w)]^{-2}\dd w
\leq C(1+ |\log \e|)$ by Lemma \ref{fcts}-(viii),

\vip

\noindent $\bullet$ $a_\e^{-1} \zeta_\e [\sigma(w/a_\e)]^{-2}
\leq C \e^{-1}(1+|\log\e|)(1+|w|/\e)^{-4/3}\leq C \e^{1/3}(1+|\log \e|) |w|^{-4/3}$ by Lemma \ref{fcts}-(vi),

\vip

\noindent $\bullet$ the integral 
$$ \int_\rr [|w|^{-1}\indiq_{\{w>0\}}+|w|^{-1}(1+|w|)^{-1 }\indiq_{\{w<0\}}+|w|^{-4/3}] 
\sup_{[0,T]}|L^w_t-L^0_t\indiq_{\{w\leq 1\}}|\dd w
$$
is a.s. finite, 
since $\sup_{[0,T]} |L^w_t-L^0_t\indiq_{\{w\leq 1\}}|$ is a.s. bounded, vanishes for $w$ sufficiently large
(namely, for $w>\sup_{[0,T]} W_s$) and is a.s. almost $1/2$-Hold\"er 
continuous near $0$, see \cite[Corollary 1.8 p 226]{ry}.
\end{proof}

We need the excursion theory for the Brownian motion, see Revuz-Yor \cite[Chapter XII, Part 2]{ry}.
We introduce a few notation and briefly summarize what we will use.

\begin{nota}\label{notaex}
Recall that $(W_t)_{t\geq 0}$ is a Brownian motion, that $(L^0_t)_{t\geq 0}$ is its local time at $0$,
that $\tau_t=\inf\{u\geq 0 : L^0_u>t\}$ is its inverse. We introduce $J=\{s>0 : \tau_s>\tau_{s-}\}$
and, for $s \in J$, 
$$
e_s=\Big(W_{\tau_\sm+r}\indiq_{\{r \in [0,\tau_s-\tau_\sm]\}}  \Big)_{r\geq 0} \in \cE,
$$
where $\cE$ is the set of continuous functions $e$ from $\rr_+$ into $\rr$ such that $e(0)=0$,
such that 
$$
\ell(e)=\sup\{r>0 : e(r)\neq 0\}\in (0,\infty)
$$ 
and such that $e(r)$ does not vanish on $(0,\ell(e))$.
For $e\in \cE$, we denote by $x(e)=\sg(e(\ell(e)/2))\in \{-1,1\}$ and observe that 
$\sg(e(r))=x(e)$ for all $r\in (0,\ell(e))$.

\vip

We introduce $\bM=\sum_{s\in J} \delta_{(s,e_s)}$,
which is a Poisson measure on $[0,\infty)\times \cE$ with intensity measure
$\dd s \Xi(\dd e)$, where $\Xi$ is a $\sigma$-finite measure on $\cE$ known as
It\^o's measure and that can be decomposed as follows:
denoting by $\cE_1=\{e \in \cE : \ell(e)=1$ and $x(e)=1\}$ and by
$\Xi_1 \in \cP(\cE_1)$ the law of the normalized Brownian excursion, 
for all measurable $A\subset \cE$,
\begin{equation}\label{tictac}
\Xi(A)=\int_0^\infty \frac{\dd \ell}{\sqrt{2\pi \ell^3}} \int_{\{-1,1\}} 
\frac12(\delta_{-1}+\delta_1)(\dd x) \int_{\cE_1} \Xi_1(\dd e) \indiq_{\{(x\sqrt\ell e(r/\ell))_{r\geq 0} \in A\}}.
\end{equation}

It holds that $\tau_t=\intot \int_\cE \ell(e) \bM(\dd s,\dd e)$ and for
all $t\in J$, all $s\in [\tau_{t-},\tau_t]$, we have 
$
W_s=e_t(s-\tau_{t-}).
$
For any $\phi : \rr \mapsto \rr_+$, any $t\geq 0$, we have
\begin{equation}\label{superform}
\int_0^{\tau_t}\phi(W_u)\dd u = \sum_{s \in J \cap [0,t]} \int_{\tau_\sm}^{\tau_s} \phi(W_u)\dd u=
\intot \int_\cE \Big[\int_0^{\ell(e)} \phi(e(u)) \dd u\Big]
\bM(\dd s,\dd e).
\end{equation}
\end{nota}

We now rewrite the processes introduced in Notation \ref{fap} in terms of the excursion Poisson measure.
We recall that $\psi,h,\sigma$ were defined in Notation \ref{notaf}.

\begin{nota}\label{fgen}
Fix $\e\in(0,1)$ and $0 \leq \delta < A\leq \infty$.
For $e\in\cE$, and $\theta=(\theta_r)_{r \in \rr}$ in $\cH=C(\rr,\Sd)$, let
\begin{align*}
F_{\e,\delta,A}(e,\theta)=&a_\e^{1/\alpha-2} \int_0^{\ell(e)}
\frac{h^{-1}(e(u)/a_\e)\theta_{r_{\e,u}(e)}-m_{\beta,\e}}
{[\sigma(e(u)/a_\e)]^2} \indiq_{\{\delta< e(u)< A\}} \dd u,
\end{align*}
where $m_{\beta,\e}=0$ if $\beta\in [d,1+d)$, $m_{\beta,\e}=\zeta_\e M_\beta$ if $\beta =1+d$ 
and $m_{\beta,\e}=m_\beta$ if $\beta\in (1+d,4+d)$ and where, for $u\in (0,\ell(e))$,
$$
r_{\e,u}(e)=\frac1{a_\e^2}\int_{\ell(e)/2}^u \frac{\dd v}{\psi(e(v)/a_\e)}.
$$
\end{nota}

Observe that $F_{\e,\delta,A}(e,\theta)=0$ if $x(e)=-1$.
Also, we make start the clock $r_{\e,u}(e)$ from the middle $\ell(e)/2$ of the excursion because at the limit,
$a_\e^2\psi(x/a_\e)$ vanishes at $x=0$ sufficiently fast so that
both $a_\e^{-2}\int_{0+}[\psi(e(v)/a_\e)]^{-1}\dd v$ and 
$a_\e^{-2}\int^{\ell(e) -}[\psi(e(v)/a_\e)]^{-1}\dd v$ will tend to infinity
as $\e\to 0$.

\begin{rk}
For all $\e\in (0,1)$, all $\delta \in (0,1)$, all $t\geq 0$, we have
\begin{align}\label{ttt}
Z^{\e,\delta}_t =& \intot   \int_\cE
F_{\e,\delta,\infty}\Big( e, \Big(\hTheta_{[P^\e_\sm+r-r_{\e,0}(e)]\lor 0} 
\Big)_{r \in \rr} \Big) \bM(\dd s,\dd e),
\end{align}
where 
$$
P^\e_t = \intot  \int_\cE
\Big[\frac1{a_\e^2}\int_0^{\ell(e)}\frac{\dd u}{\psi(e(u)/a_\e)} \Big] \bM(\dd s,\dd e).
$$
\end{rk}

\begin{proof}
For any reasonable $\phi_1: \rr\times\rr\mapsto \rr$
and $\phi_2:\rr\mapsto \rr$, if setting $\nu_t=\intot \phi_2(W_s)\dd s$, we have
\begin{align*}
\int_0^{\tau_t} \phi_1(W_s,\nu_s)\dd s =&
\sum_{s\in J\cap [0,t]} \int_{\tau_\sm}^{\tau_s} \phi_1\Big(W_u,\nu_{\tau_\sm}+\int_{\tau_\sm}^u \phi_2(W_v)\dd v\Big)\dd u\\
=& \intot \int_\cE \Big[\int_0^{\ell(e)} \phi_1\Big(e(u),\nu_{\tau_\sm}+ 
\int_0^{u} \phi_2(e(v))\dd v \Big)\dd u\Big] \bM(\dd s,\dd e).
\end{align*}
With $\phi_1(w,\nu)=a_\e^{1/\alpha-2}[\sigma(w/a_\e)]^{-2}[h^{-1}(w/a_\e)\hTheta_\nu-m_{\beta,\e}]\indiq_{\{w>\delta\}}$
and $\phi_2(w)=a_\e^{-2}[\psi(w/a_\e)]^{-1}$, so that $T^\e_t=\intot \phi_2(W_s)\dd s$
and $P^\e_t=T^\e_{\tau_t}$ by \eqref{superform}, this gives
\begin{align*}
Z^{\e,\delta}_t=& \intot   \int_\cE \Big[a_\e^{1/\alpha-2}\int_0^{\ell(e)} 
\frac{h^{-1}(e(u)/a_\e)\hTheta_{P^\e_\sm+ a_\e^{-2}\int_0^{u} [\psi(e(v))]^{-2}\dd v} - m_{\beta,\e}}
{[\sigma(e(u)/a_\e)]^2} \indiq_{\{e(u)\geq \delta \}}\dd u\Big] \bM(\dd s,\dd e),
\end{align*}
from which the result follows because by definition of $F_{\e,\delta,\infty}$, we have
\begin{align*}
&F_{\e,\delta,\infty}\Big(e, \Big(\hTheta_{[P^\e_\sm+r-r_{\e,0}(e)]\lor 0} 
\Big)_{r \in \rr} \Big)\\
=& a_\e^{1/\alpha-2} \int_0^{\ell(e)}
\frac{h^{-1}(e(u)/a_\e)\hTheta_{[P^\e_\sm+r_{\e,u}(e)-r_{\e,0}(e)]\lor 0}-m_{\beta,\e}}
{[\sigma(e(u)/a_\e)]^2} \indiq_{\{e(u)\geq \delta\}} \dd u
\end{align*}
and because $P^\e_\sm$ is positive, as well as $r_{\e,u}(e)-r_{\e,0}(e)$ which equals
$$
\frac1{a_\e^2}\int_{\ell(e)/2}^u \frac{\dd v}{\psi(e(v)/a_\e)}+  
\frac1{a_\e^2}\int_0^{\ell(e)/2} \frac{\dd v}{\psi(e(v)/a_\e)}=
\frac1{a_\e^2}\int_0^u \frac{\dd v}{\psi(e(v)/a_\e)}
$$
as desired.
\end{proof}

We now get rid of the correlation in the spherical process.

\begin{lem}\label{dec}
Let $\bN$ be a Poisson measure on $[0,\infty)\times  \cE \times \cH$ 
with intensity measure
$$
\pi(\dd s,\dd e,\dd \theta)= \dd s \Xi(\dd e)
\Lambda(\dd \theta),
$$
for $\Xi \in \cP(\cE)$ the law of the normalized Brownian excursion
and $\Lambda \in \cP(\cH)$ the law of the stationary eternal
spherical process built in Lemma \ref{ergo2}. 
For $\e\in(0,1)$ and $\delta\in(0,1)$, we introduce the process
$$
\bZ^{\e,\delta}_t= \intot \int_\cE \int_{\cH}
F_{\e,\delta,\infty}(e,\theta) \bN(\dd s,\dd e,\dd\theta).
$$
For all $T>0$, all $\delta>0$, there exists a function $q_{T,\delta}:(0,1)\mapsto \rr_+$
such that $\lim_{\e\to 0} q_{T,\delta}(\e)=1$ and such that for any $\e\in(0,1)$, we can find a coupling
between $(Z^{\e,\delta}_t)_{t\in [0,T]}$ and $(\bZ^{\e,\delta}_t)_{t\in [0,T]}$ such that
$\PR[(Z^{\e,\delta}_t)_{t\in [0,T]}=(\bZ^{\e,\delta}_t)_{t\in [0,T]}] \geq q_{T,\delta}(\e)$.
\end{lem}

Observe that the process $(\bZ^{\e,\delta}_t)_{t\geq 0}$ is L\'evy.

\begin{proof} The proof is tedious, but rather simple in its principle:
the main idea is that the clock of $\hTheta$ in \eqref{ttt} runs a very long way (asymptotically infinite 
when $\e\to 0$) between two excursions, so that we can apply Lemma \ref{ergo2}-(iv).

\vip

\underline{\it Step 1.} For all $\delta\in(0,1)$, all $e\in\cE$, 
there is $s_\delta(e)>0$ such that
for all $\e\in(0,1)$, all $\theta,\theta' \in \cH$, we have 
$F_{\e,\delta,\infty}(e,\theta)=F_{\e,\delta,\infty}(e,\theta')$  
as soon as $\theta_r=\theta'_r$ for
all $r\in[-s_\delta(e),s_\delta(e)]$.

\vip

We recall that $F_{\e,\delta,\infty}(e,\theta)=0$ if $x(e)=-1$, so that it suffices to treat the case
of positive excursions. We have $F_{\e,\delta,\infty}(e,\theta)=F_{\e,\delta,\infty}(e,\theta')$  
if $\theta_u=\theta'_u$ for all $u\in[-s_{\delta,\e}(e),s_{\delta,\e}(e)]$, where
$$
s_{\delta,\e}(e)=\max\{-r_{\e, \inf\{v>0 : e(v) > \delta\}\land(\ell(e)/2) }(e),r_{\e, \sup\{v>0 : e(v) > \delta\} \lor(\ell(e)/2) }(e)\}
$$
because then for all $u\in (0,\ell(e))$ such that
$\theta_{r_{\e,u}(e)}\neq \theta'_{r_{\e,u}(e)}$, we have 
$$
\hbox{either} \quad
r_{\e,u}(e)>r_{\e, \sup\{v>0 : e(v) > \delta\} \lor(\ell(e)/2) }(e) 
\quad \hbox{or} \quad 
r_{\e,u}(e)<r_{\e, \inf\{v>0 : e(v) > \delta\} \land(\ell(e)/2)}(e),
$$ 
whence in both cases $e(u)<\delta$, which makes vanish the indicator
function $\indiq_{\{e(u)\geq \delta\}}$.
Using now that $a_\e^{-2}[\psi(w/a_\e)]^{-1} \leq C w^{-2}$ for all $w>0$ by Lemma \ref{fcts}-(iv),
we realize that 
$$
s_{\delta,\e}(e)\leq 
C \int_{\inf\{v>0 : e(v) > \delta\}\land(\ell(e)/2)}^{\sup\{v>0 : e(v) > \delta\} \lor(\ell(e)/2)} \frac{\dd u}{[e(u)]^2}.
$$
Denoting by $s_\delta(e)$ this last quantity, which is finite because $e$ does not 
vanish during the interval $[\inf\{v>0 : e(v) > \delta\}\land(\ell(e)/2),
\sup\{v>0 : e(v) > \delta\} \lor(\ell(e)/2)]$, completes the step.

\vip

\underline{\it Step 2.} Since only a finite number of excursions exceed $\delta$ per unit of time
we may rewrite \eqref{ttt} as 
$$
Z^{\e,\delta}_t = \sum_{i=1}^{N^\delta_t}
F_{\e,\delta,\infty} \Big(e_i^\delta,(\hTheta_{[T_i^{\e,\delta}+r]\lor 0})_{r\geq 0} \Big),
$$
where $\cE_\delta=\{e\in \cE : \sup_{u\in[0,\ell(e)]} e(u) >\delta \}$,
$N_t^\delta=\bM([0,t]\times \cE_\delta)$, of which we denote by $(s_i^\delta)_{i\geq 1}$
the chronologically ordered instants of jump. For each $i \geq 1$, we have introduced by 
$e_i^\delta \in \cE_\delta$ the mark associated to $s_i^\delta$,
uniquely defined by the fact that $\bM(\{(s_i^\delta,e_i^\delta)\})=1$.
We also have set, for each $i\geq 1$,
$$
T_i^{\e,\delta}=P^\e_{s_i^\delta-}-r_{\e,0}(e_i^\delta).
$$

\underline{\it Step 3.} Here we show that for all $\delta\in(0,1)$, all $T>0$, a.s., 
$\min_{i=1,\dots,N^\delta_T} (T_i^{\e,\delta}-T_{i-1}^{\e,\delta} )\to \infty$ as $\e\to 0$.
It suffices to observe that, since $\psi(u)\leq C(1+|u|^2)$ by Lemma \ref{fcts}-(iv)
and since $P^\e_{s_i^\delta-}\geq T_{i-1}^{\e,\delta}$,
$$
T_i^{\e,\delta}-T_{i-1}^{\e,\delta} \geq -r_{\e,0}(e_i^\delta) =
\frac{1}{a_\e^2}\int_0^{\ell(e_i^\delta)/2}\frac{\dd v }{\psi[e_i^\delta(v)/a_\e]}
\geq c \int_0^{\ell(e_i^\delta)/2}\frac{\dd v }{a_\e^2 + [e_i^\delta(v)]^2}.
$$
By monotone convergence, we conclude that
$$
\liminf_{\e\to 0} (T_i^{\e,\delta}-T_{i-1}^{\e,\delta})
\geq c \int_0^{\ell(e_i^\delta)/2}\frac{\dd v }{[e_i^\delta(v)]^2}= \infty \quad \hbox{a.s.}
$$
by Lemma \ref{exc}-(i).

\vip

\underline{\it Step 4.} We work conditionally on $\bM$ and set 
$A_{T,\delta}= \sup_{i=1,\dots,N^\delta_T} s_\delta(e_i^\delta)$.
By Lemma \ref{ergo2}-(iv), we can find, for each $\e\in (0,1)$, an i.i.d. family
of $\Lambda$-distributed eternal processes $(\hTheta_r^{\star,1,\e})_{r\in\rr}$, ..., 
$(\hTheta_r^{\star,N^\delta_T,\e})_{r\in\rr}$ such that the probability (conditionally on $\bM$) that
$(\hTheta_{[T_i^{\e,\delta}+r]\lor 0})_{r\in[-A_{T,\delta},A_{T,\delta}]}=
(\hTheta_r^{\star,i,\e})_{r\in[-A_{T,\delta},A_{T,\delta}]}$ for all $i=1,\dots,N^\delta_T$
is greater than 
$p_{T,\delta,\e}=p_{A_{T,\delta}}(T_1^{\e,\delta}, T_2^{\e,\delta}-T_{1}^{\e,\delta},\dots,
T_{N^\delta_t}^{\e,\delta}-T_{N^\delta_t-1}^{\e,\delta})$,
which a.s. tends to $1$ as $\e\to 0$ by Step 3.

\vip

\underline{\it Step 5.}
We set, for $t\in [0,T]$, 
$$
\bZ^{\e,\delta}_t=\sum_{i=1}^{N^\delta_t}
F_{\e,\delta,\infty} \Big(e_i^\delta,(\hTheta_r^{\star,i,\e})_{r\geq 0} \Big).
$$
This process has the same law as the process $(\bZ^{\e,\delta}_t)_{t\in [0,T]}$ of the statement.
Furthermore, we know from Step 1 that $Z^{\e,\delta}_t=\bZ^{\e,\delta}_t$ for all $t\in [0,T]$ as
soon as $(\hTheta_{[T_i^{\e,\delta}+r]\lor 0})_{r\in[-A_{T,\delta},A_{T,\delta}]}=
(\hTheta_t^{\star,i,\e})_{r\in[-A_{T,\delta},A_{T,\delta}]}$ for all $i=1,\dots,N^\delta_T$.
This occurs with probability $q_{T,\delta}(\e)=\E[p_{T,\delta,\e}]$, which tends to $1$ as $\e\to 0$ by dominated
convergence. 
\end{proof}

We introduce the compensated Poisson measure
$\tbN=\bN-\pi$.

\begin{lem} \label{jj2}
We fix $\delta\in(0,1]$ and $\e\in(0,1)$.

\vip

(i) If $\beta \in [d,1+d)$, we simply set $\hZ^{\e,\delta}_t=\bZ^{\e,\delta}_t$ for all $t\geq 0$.

\vip

(ii) If $\beta=1+d$, we define $\hZ^{\e,\delta}_t=\bZ^{\e,\delta}_t+\kappa_{\e,\delta} M_\beta t$ 
for all $t\geq 0$ and we have
\begin{align*}
\hZ^{\e,\delta}_t= \intot \int_\cE \int_{\cH}
F_{\e,\delta,1}(e,\theta) \tbN(\dd s,\dd e,\dd\theta)
+ \intot \int_\cE \int_{\cH}
F_{\e,1,\infty}(e,\theta) \bN(\dd s,\dd e,\dd\theta).
\end{align*}

(iii) If $\beta \in (1+d,4+d)$, we set $\hZ^{\e,\delta}_t=\bZ^{\e,\delta}_t+\kappa_{\e,\delta} M_\beta t$ 
for all $t\geq 0$ and we have
\begin{align*}
\hZ^{\e,\delta}_t=& \intot \int_\cE  \int_{\cH}
F_{\e,\delta,\infty}(e,\theta) \tbN(\dd s,\dd e,\dd\theta).
\end{align*}
\end{lem}

\begin{proof} We recall that $\int_\rr \phi(w)\dd w=
\int_\cE \Big[\int_0^{\ell(e)} \phi(e(u))\dd u \Big] \Xi(\dd e)$ for all $\phi \in L^1(\rr)$,
see Lemma \ref{exc}-(ii).

\vip

To verify (iii), we have to check that
$$
I=\int_\cE\int_\cH 
F_{\e,\delta,\infty}(e,\theta)\Lambda(\dd\theta)\Xi(\dd e)=-\kappa_{\e,\delta} M_\beta.
$$
Recalling the expression of $F_{\e,\delta,\infty}$ and that $\Lambda$ is the law of the eternal
stationary spherical process, see Lemma \ref{ergo2}, of which the invariant measure is
$\nu_\beta$, which satisfies $\int_{\Sd}\theta\nu_\beta(\dd \theta)=M_\beta$, we find
\begin{align*}
I=&\int_\cE \Big[a_\e^{1/\alpha-2}\int_0^{\ell(e)}
\frac{h^{-1}(e(u)/a_\e))M_\beta-m_\beta}{[\sigma(e(u)/a_\e)]^2}
\indiq_{\{ e(u)>\delta  \}}\dd u \Big]\Xi(\dd e) \\
=& a_\e^{1/\alpha-2} \int_\delta^\infty \frac{h^{-1}(w/a_\e)M_\beta-m_\beta}{[\sigma(w/a_\e)]^2} 
\dd w.
\end{align*}
Recalling that $m_\beta=M_\beta m_\beta'$, the definition of $\kappa_{\e,\delta}$ (see Notation \ref{fap}-(iii))
and that $\kappa_{\e,\infty}=0$, see \eqref{keiz},
$$
I=M_\beta a_\e^{1/\alpha-2} \int_\delta^\infty \frac{h^{-1}(w/a_\e)-m'_\beta}{[\sigma(w/a_\e)]^2}
\dd w = - M_\beta a_\e^{1/\alpha-2} \int_{-\infty}^\delta \frac{h^{-1}(w/a_\e)-m'_\beta}{[\sigma(w/a_\e)]^2}\dd w,
$$
which equals $-M_\beta \kappa_{\e,\delta}$ as desired.

\vip

Concerning (ii), since $F_{\e,\delta,\infty}=F_{\e,\delta,1}+ F_{\e,1,\infty}$, we have to verify that
$$
J=\int_\cE\int_\cH 
F_{\e,\delta,1}(e,\theta)\Lambda(\dd\theta)\Xi(\dd e)=-\kappa_{\e,\delta} M_\beta.
$$
Proceeding as above, we find
$$
J=M_\beta a_\e^{-1} \int_\delta^1 \frac{h^{-1}(w/a_\e)-\zeta_\e}{[\sigma(w/a_\e)]^2} \dd w
=-M_\beta a_\e^{-1} \int_{-\infty}^\delta \frac{h^{-1}(w/a_\e)-\zeta_\e}{[\sigma(w/a_\e)]^2}
\dd w=-M_\beta \kappa_{\e,\delta}
$$
by definition of $\kappa_{\e,\delta}$ and since
$\kappa_{\e,1}=0$, recall Notation \ref{fap}-(ii).
\end{proof}

We now introduce the limit (as $\e\to 0$) of the function defined in Notation \ref{fgen}.

\begin{nota}\label{fgenl}
Fix $0 \leq \delta < A\leq \infty$.
For $e\in\cE$ and $\theta=(\theta_r)_{r \in \rr}$ in $\cH=C(\rr,\Sd)$,
we set
\begin{align*}
F_{\delta,A}(e,\theta)=& \frac1{(\beta+2-d)^2} \int_0^{\ell(e)}
[e(u)]^{1/\alpha-2}\theta_{r_{u}(e)} \indiq_{\{\delta \leq e(u) \leq A\}}\dd u,
\end{align*}
where, for $u\in (0,\ell(e))$, 
$$
r_{u}(e)=\frac1{(\beta+2-d)^2}\int_{\ell(e)/2}^u \frac{\dd v}{[e(v)]^{2}}.
$$
\end{nota}

Finally, we make tend $\e$ and $\delta$ to $0$.

\begin{lem}\label{ppfini}
We consider the processes $(\hZ^{\e,\delta}_t)_{t\geq 0}$ introduced in Lemma \ref{jj2}, built
with the same Poisson measure $\bN$ for all values of $\e\in(0,1)$ and $\delta\in(0,1)$.
For all $T>0$, $\sup_{[0,T]} |\hZ^{\e,\delta}_t- Z_t|$ goes to $0$ in probability as $(\e,\delta)\to(0,0)$, where

\vip

(i) $Z_t=\intot \int_\cE \int_{\cH} F_{0,\infty}(e,\theta) \bN(\dd s,\dd e,\dd\theta)$  if $\beta \in [d,1+d)$, 

\vip

(ii) $Z_t=\intot \int_\cE \int_{\cH} F_{0,1}(e,\theta) 
\tbN(\dd s,\dd e,\dd\theta)+\intot \int_\cE \int_{\cH} F_{1,\infty}(e,\theta) 
\bN(\dd s,\dd e,\dd\theta)$  if $\beta=1+d$,

\vip

(iii) $Z_t=\intot \int_\cE \int_{\cH}
F_{0,\infty}(e,\theta) \tbN(\dd s,\dd e,\dd\theta)$     if $\beta \in (1+d,4+d)$.
\end{lem}

\begin{proof}
We  divide the proof in several steps.

\vip 

\underline{\it Step 1.} There is $C>0$ such that for all
$\e\in (0,1]$, all $0\leq \delta \leq A\leq \infty$, 
all $e\in\cE$, all $\theta\in \cH$,
\begin{gather*}
|F_{\e,\delta,A}(e,\theta)| \leq C \int_0^{\ell(e)} ([e(u)]^{1/\alpha-2} + \indiq_{\{\beta=1+d\}}[e(u)]^{-4/3} )
\indiq_{\{ \delta \leq e(u) \leq A\}} \dd u.
\end{gather*}
Indeed, we know from Lemma \ref{fcts}-(viii) that $[1+h^{-1}(w)][\sigma(w)]^{-2} \leq C(1+|w|)^{1/\alpha-2}$, 
which implies that $a_\e^{1/\alpha-2}[1+h^{-1}(w/a_\e)][\sigma(w/a_\e)]^{-2} \leq C|w|^{1/\alpha-2}$, and it
only remains to note that when $\beta=1+d$ (so that $\alpha=1$),
\begin{equation}\label{benvoila}
\Big|\frac{m_{\beta,\e}}{a_\e [\sigma(w/a_\e)]^{2}}\Big|=
\frac {|M_\beta|\zeta_\e}{a_\e [\sigma(w/a_\e)]^{2}} \leq C \frac{ 1+|\log \e|}{\e (1+|w|/\e)^{4/3}}
\leq C \frac{ \e^{1/3}(1+|\log \e|)}{|w|^{4/3}}
\end{equation}
by Lemma \ref{fcts}-(vi), since $a_\e=\kappa\e$ 
and since $\zeta_\e \leq C(1+|\log \e|)$, see the end of the proof 
of Lemma \ref{fapl}.

\vip

\underline{\it Step 2.} We fix $0\leq \delta_0<A\leq\infty$ and verify that 
for all $\theta\in \cH$ and $\Xi$-almost every $e\in \cE$, we have
$$
\lim_{(\e,\delta)\to (0,\delta_0)} F_{\e, \delta,A}(e,\theta) = F_{\delta_0,A}(e,\theta).
$$
Using precisely the same bounds as in Step 1, the
result follows from dominated convergence, because 

\vip

$\bullet$ $a_\e^{1/\alpha-2} h^{-1}(w/a_\e) [\sigma(w/a_\e)]^{-2} \to (\beta+2-d)^{-2} w^{1/\alpha-2}$
for each fixed $w>0$ by Lemma \ref{fcts}-(ix), 

\vip

$\bullet$ $\theta\in \cH$ is continuous and
$r_{\e,u}(e)=a_\e^{-2}\int_{\ell(e)/2}^u [\psi(e(v)/a_\e)]^{-1}\dd v \to r_u(e)$ for each
$u\in(0,\ell(e))$ by Lemma \ref{fcts}-(v) (and by dominated convergence),

\vip

$\bullet$ $a_\e^{1/\alpha-2} m_{\beta,\e} [\sigma(w/a_\e)]^{-2} \to 0$ for each fixed $w>0$,
because

\vip

\hskip0.3cm $\star$ if $\beta\in[d,1+d)$, $m_{\beta,\e}=0$,

\vip

\hskip0.3cm $\star$ if $\beta=1+d$, see \eqref{benvoila},

\vip

\hskip0.3cm $\star$  if $\beta\in(1+d,4+d)$, then
$a_\e^{1/\alpha-2} |m_{\beta,\e}| [\sigma(w/a_\e)]^{-2} \leq C \e^{1/\alpha-2}(1+w/\e)^{-2(\beta+1-d)/(\beta+2-d)}\to 0$,
by Lemma \ref{fcts}-(vi), since $m_{\beta,\e}=m_\beta$ and since $2(\beta+1-d)/(\beta+2-d)>2-1/\alpha$,

\vip
$\bullet$ $\int_0^{\ell(e)} ([e(u)]^{1/\alpha-2}+[e(u)]^{-4/3}) \dd u<\infty$ for $\Xi$-almost every
$e\in \cE$ by Lemma \ref{exc}-(iv).
\vip

\underline{\it Step 3.} We now conclude. We write $\hZ^{\e,\delta}_t=Y^{\e,1}_t-Y^{\e,2}_t
+Y^{\e,\delta,3}_t$  and $Z_t=Y_t^1-Y_t^2+Y_t^3$, where

\vip

\noindent $\bullet$  $Y_t^1=\intot \int_\cE \int_{\cH} F_{1,\infty}(e,\theta) \bN(\dd s,\dd e,\dd\theta)$
and $Y^{\e,1}_t=\intot \int_\cE \int_{\cH} F_{\e,1,\infty}(e,\theta) \bN(\dd s,\dd e,\dd\theta)$,

\vip

\noindent $\bullet$ $Y_t^2=t\int_\cE \int_{\cH} F_{1,\infty}(e,\theta) \Lambda(\dd \theta)\Xi(\dd e)$
and $Y^{\e,2}_t=t\int_\cE \int_{\cH} F_{\e,1,\infty}(e,\theta) \Lambda(\dd \theta)\Xi(\dd e)$
if $\beta\in (1+d,4+d)$,

\vip

\noindent $\bullet$ $Y_t^2=Y^{\e,2}_t=0$ if $\beta \in [d,1+d]$,

\vip

\noindent $\bullet$  $Y_t^3=\intot \int_\cE \int_{\cH} F_{0,1}(e,\theta) \bN(\dd s,\dd e,\dd\theta)$
and $Y^{\e,\delta,3}_t=\intot \int_\cE \int_{\cH} F_{\e,\delta,1}(e,\theta) \bN(\dd s,\dd e,\dd\theta)$
if $\beta \in [d,1+d)$,

\vip

\noindent $\bullet$  $Y_t^3=\intot \int_\cE \int_{\cH} F_{0,1}(e,\theta) \tbN(\dd s,\dd e,\dd\theta)$
and $Y^{\e,\delta,3}_t\!=\intot \int_\cE \int_{\cH} F_{\e,\delta,1}(e,\theta) \tbN(\dd s,\dd e,\dd\theta)$
if $\beta \in [1+d,4+d)$.

\vip

\underline{\it Step 3.1.} For any $\beta \in [d,4+d)$, it holds that a.s.,
$$
\lim_{\e\to 0} \sup_{[0,T]} |Y^1_t-Y^{\e,1}_t | \leq \lim_{\e\to 0}
\int_0^T \int_\cE \int_{\cH} |F_{1,\infty}(e,\theta) -F_{\e,1,\infty}(e,\theta)|\bN(\dd s,\dd e,\dd\theta)=0.
$$
This uses only the facts that $F_{1,\infty}(e,\theta)=F_{\e,1,\infty}(e,\theta)=0$ as soon as 
$\sup_{r\in[0,\ell(e)]} e(r)<1$, 
that $\bN(\{(s,e,\theta)\in [0,T]\times\cE\times\cH : \sup_{[0,\ell(e)]} e\geq 1\})$ is a.s. finite, and 
that $\lim_{\e\to 0}F_{\e,1,\infty}(e,\theta)=F_{1,\infty}(e,\theta)$ for $\Xi\otimes\Lambda$-almost every 
$(e,\theta)\in\cE\times\cH$ by Step 2.

\vip

\underline{\it Step 3.2.} If $\beta \in (1+d,4+d)$, it holds that
$$
\lim_{\e\to 0} \sup_{[0,T]} |Y^2_t-Y^{\e,2}_t |\leq  T \lim_{\e\to 0}
\int_\cE \int_{\cH} |F_{1,\infty}(e,\theta) -F_{\e,1,\infty}(e,\theta)| \Lambda(\dd\theta)\Xi(\dd e)=0
$$
by dominated convergence, thanks to Steps 1 and 2 and since
$\int_\cE [\int_0^{\ell(e)} (e(u))^{1/\alpha-2}\indiq_{\{e(u)\geq 1\}}\dd u  
]\Xi(\dd e)=\int_1^\infty x^{1/\alpha-2}\dd x <\infty$ 
by Lemma \ref{exc}-(ii) and since $1/\alpha-2<-1$ because $\alpha=(\beta+2-d)/3>1$.

\vip

\underline{\it Step 3.3.} If $\beta\in [d,1+d)$,
\begin{align*}
\lim_{(\e,\delta)\to(0,0)}\E\Big[\sup_{[0,T]} |Y^3_t-Y^{\e,\delta,3}_t|\Big]\leq T \lim_{(\e,\delta\to(0,0)}
\int_\cE \int_{\cH} |F_{0,1}(e,\theta) -F_{\e,\delta,1}(e,\theta)| \Lambda(\dd \theta)\Xi(\dd e)=0
\end{align*}
by dominated convergence, using Steps 1 and 2 and
that
$\int_\cE [\int_0^{\ell(e)} [e(u)]^{1/\alpha-2}\indiq_{\{0\leq e(u)\leq 1\}} \dd u] \Xi(\dd e) 
= \int_0^1 w^{1/\alpha-2}\dd w <\infty$ by Lemma \ref{exc}-(ii) and since $1/\alpha-2>-1$ 
because $\alpha=(\beta+2-d)/3<1$.

\vip

\underline{\it Step 3.4.} If finally $\beta \in [1+d,4+d)$, by Doob's inequality,
\begin{align*}
\lim_{(\e,\delta)\to(0,0)}\E\Big[\sup_{[0,T]} |Y^3_t-Y^{\e,\delta,3}_t|^2\Big]\leq
4T \lim_{(\e,\delta\to(0,0)}
\int_\cE \int_{\cH} |F_{0,1}(e,\theta) -F_{\e,\delta,1}(e,\theta)|^2 \Lambda(\dd \theta)\Xi(\dd e)=0
\end{align*}
by dominated convergence, thanks to Steps 1 and 2 and
since we know from Lemma \ref{exc}-(iii) 
that $\int_\cE [\int_0^{\ell(e)}  (|e(u)|^{1/\alpha-2}+|e(u)|^{-4/3})\indiq_{\{0\leq e(u)\leq 1\}}\dd u]^2 \Xi(\dd e)
\leq 4 [\int_0^1 \sqrt x (x^{1/\alpha-2}+x^{-4/3}) \dd x]^2<\infty$.
\end{proof}

Gathering all the previous lemmas, one 
obtains the following convergence result.

\begin{prop}\label{pfini}
Consider the process $(Z_t)_{t\geq 0}$ defined in Lemma \ref{ppfini} (its definition 
depending on $\beta$) and set $S_t= \kappa^{-1/\alpha}Z_t$ if $\beta \in (d,4+d)$
and $S_t=8Z_t$ if $\beta=d$.

\vip

(i) If $\beta \in (1+d,4+d)$, then $(\e^{1/\alpha} [X_{t/\e}-m_\beta t/\e])_{t\geq 0} 
\stackrel{f.d.} \longrightarrow 
(S_t)_{t\geq 0}$.

\vip

(ii) If $\beta=1+d$, then  $(\e[X_{t/\e}-M_\beta \zeta_\e t/\e])_{t\geq 0} \stackrel{f.d.} \longrightarrow 
(S_t)_{t\geq 0}$.

\vip

(iii) If $\beta \in (d,1+d)$, then $(\e^{1/\alpha} X_{t/\e})_{t\geq 0} \stackrel{f.d.} \longrightarrow 
(S_t)_{t\geq 0}$.

\vip

(iv) If $\beta=d$, then $([\e |\log \e|]^{3/2} X_{t/\e})_{t\geq 0} \stackrel{f.d.} \longrightarrow 
(S_t)_{t\geq 0}$.
\end{prop}

\begin{proof} Since $a_\e=\kappa \e$ when $\beta\in (d,4+d)$ and
$a_\e=\e|\log\e|/4$ when $\beta=d$, it is sufficient to prove that,
setting $m_{\beta,\e}=m_{\beta}$ if $\beta \in (1+d,4+d)$,
$m_{\beta,\e}=M_\beta \zeta_\e$ if $\beta=1+d$ and $m_{\beta,\e}=0$ if
$\beta \in [d,1+d)$, it holds that
$(a_\e^{1/\alpha} [X_{t/\e}-m_{\beta,\e} t/\e])_{t\geq 0} \stackrel{f.d.} \longrightarrow 
(Z_t)_{t\geq 0}$.

\vip

We know from Lemma \ref{ettac} that   $(X_{t/\e})_{t\geq 0}\stackrel{d}=
(x_0+\tX^\e_t)_{t\geq 0}$. Since $a_\e^{1/\alpha} x_0\to 0$, it thus suffices to verify that
$(\chZ^\e_t)_{t\geq 0} \stackrel{f.d.} \longrightarrow (Z_t)_{t\geq 0}$, where we have set
$\chZ^\e_t=a_\e^{1/\alpha} [\tX_{t/\e}-m_{\beta,\e} t/\e]$.

\vip

We consider $\Phi : \mathbb{D}([0,\infty),\rr^d)\mapsto \rr$
of the form $\Phi(x)=\phi(x_{t_1},\dots,x_{t_n})$ for some continuous and bounded $\phi: \rr^n\mapsto \rr$.
Our goal is to check that 
$I_\e=\E[\Phi((\chZ^\e_t)_{t\geq 0})] \to \E[\Phi((Z_t)_{t\geq 0})]=I$ as $\e\to 0$.

\vip

We know from Lemma \ref{fapl} that for all $t\geq0$, all $\eta>0$,
$$
\lim_{\delta \to 0}\limsup_{\e\to 0}\PR(|\chZ^{\e}_t - [Z^{\e,\delta}_t
+\kappa_{\e,\delta}M_\beta t]|\geq \eta)=0,
$$
with the convention that $\kappa_{\e,\delta}=0$ if $\beta \in [d,1+d)$.
Setting $I_{\epsilon,\delta}=\E[\Phi(([Z^{\e,\delta}_t
+\kappa_{\e,\delta}M_\beta t])_{t\geq 0})]$, we deduce that $\lim_{\delta \to 0} \limsup_{\e\to 0} |I_{\e,\delta}-I_\e|=0$.
We thus have to check that  $\lim_{\delta \to 0} \limsup_{\e\to 0} |I_{\e,\delta}-I|=0$.

\vip

By Lemma \ref{dec}, we know that for each $\delta>0$, $\lim_{\e\to0}|I_{\e,\delta}-J_{\e,\delta}|= 0$
for each $\delta>0$,  where we have set $J_{\e,\delta}=\E[\Phi((\bZ^{\e,\delta}_t
+\kappa_{\e,\delta}M_\beta t)_{t\geq 0})]$. It thus suffices to verify that 
$\lim_{\delta \to 0} \limsup_{\e\to 0} |J_{\e,\delta}-I|=0$.

\vip

By Lemma \ref{jj2}, it holds that $J_{\e,\delta}=\E[\Phi((\hZ^{\e,\delta}_t)_{t\geq 0})]$.

\vip

Finally, it follows from Lemma \ref{ppfini} that $\lim_{(\e,\delta)\to(0,0)} J_{\e,\delta}=I$, 
which completes the proof.
\end{proof}

We still have to study a little our limiting processes.

\begin{prop}\label{tpfini}
For any $\beta \in [d,4+d)$, set $\alpha = (\beta+2-d)/3$ and consider 
the limit process $(S_t)_{t\geq 0}$ introduced in Proposition \ref{pfini} (its definition 
depending on $\beta$).

\vip

(i) The process $(S_t)_{t\geq 0}$ is an $\alpha$-stable L\'evy process of which the L\'evy measure
$q$, depending only $\beta$ and $U$, is given, for all $A\in \cB(\rd\setminus\{0\})$, by
$$
q(A)={\mathfrak a} \int_0^\infty u^{-1-\alpha}\PR(u Y \in A)\dd u ,
$$
where ${\mathfrak a}=\alpha/[\kappa \sqrt{2\pi}(\beta+2-d)^{2\alpha}]$ with $\kappa=
(\beta+2-d)^{-1}\int_0^\infty u^{d-1}[\Gamma(u)]^{-\beta}\dd u$ (see Lemma \ref{fcts}-(i)) if $\beta \in (d,4+d)$, 
where ${\mathfrak a}= 2^{7/6}/[3\sqrt \pi]$ if $\beta=d$ and where
the $\rd$-valued random variable $Y$ is defined as follows.
Consider a normalized Brownian excursion ${\bf e}$ (with unit length),
independent of an eternal stationary spherical process $(\hTheta^\star_t)_{t\in \rr}$ 
as in Lemma \ref{ergo2}-(iii) and set
$$
Y=\int_0^1 [{\bf e}(u)]^{1/\alpha-2} \hTheta^\star_{[(\beta+2-d)^{-2} \int_{1/2}^u [{\bf e}(v)]^{-2}\dd v]} \dd u.
$$

\vip

(ii) Assume now that $\gamma\equiv 1$ (recall Assumption \ref{as}). Then
$$
(\e^{1/\alpha}X_{t/\e})_{t\geq 0} \stackrel{f.d.} \longrightarrow (S_t)_{t\geq 0}
\quad
\hbox{if $\beta \in (d,4+d)$ and}
\quad
([\e|\log \e|]^{3/2}X_{t/\e})_{t\geq 0} \stackrel{f.d.} \longrightarrow (S_t)_{t\geq 0}
\quad
\hbox{if $\beta=d$}
$$
and in any case, $(S_t)_{t\geq 0}$ is a radially symmetric $\alpha$-stable L\'evy process, 
that is, there is a constant ${\mathfrak b}>0$ depending on $\Gamma$, $\beta$ and $d$ such that 
$q(\dd z)={\mathfrak b} |z|^{-d-\alpha}\dd z$ and thus 
$\E[\exp(i \xi \cdot S_t)]=\exp(- {\mathfrak b}' t |\xi|^\alpha)$
for all $\xi \in \rd$, all $t\geq 0$, for some other constant ${\mathfrak b}'>0$.
\end{prop}

Observe that in (i), the random variable $Y$ is well-defined thanks to Lemma
\ref{exc}-(iv).

\begin{proof}
We start with point (i).
It readily follows from its definition (see Proposition \ref{pfini} and Lemma \ref{ppfini})
that $(S_t)_{t\geq 0}$ is a L\'evy process with L\'evy measure given by
$$
q(A)= \int_{\cE} \Xi(\dd e) \int_{\cH} \Lambda(\dd \theta) 
\indiq_{\{{\mathfrak c} F_{0,\infty}(e,\theta)\in A\}}, \quad A\in\cB(\rd\setminus\{0\}),
$$
where ${\mathfrak c}=\kappa^{-1/\alpha}$ if $\beta\in (d,4+d)$ and ${\mathfrak c}=8$ if
$\beta=d$. Using the decomposition \eqref{tictac} of $\Xi$
and that $F_0(e,\theta)=0$ if $x(e)=-1$, we have
\begin{align*}
q(A)=& \int_0^\infty \frac{\dd \ell}{2\sqrt{2\pi\ell^3}}\int_{\cE_1} \Xi_1(\dd e) \int_{\cH} \Lambda(\dd \theta) 
\indiq_{\{{\mathfrak c} F_{0,\infty}(\sqrt{\ell}e(\cdot/\ell),\theta)\in A\}}\\
=& \int_0^\infty \frac{\dd \ell}{2\sqrt{2\pi\ell^3}} 
\PR\Big({\mathfrak c}F_{0,\infty}(\sqrt{\ell}{\bf e}(\cdot/\ell),\hTheta^\star)
\in A \Big)
\end{align*}
with the notation of the statement. But recalling Notation \ref{fgenl},
\begin{align*}
F_{0,\infty}(\sqrt{\ell}{\bf e}(\cdot/\ell),\hTheta^\star)=&\frac1{(\beta+2-d)^2}
\int_0^\ell [\sqrt{\ell}{\bf e}(u/\ell)]^{1/\alpha-2}
\hTheta^\star_{[(\beta+2-d)^{-2} \int_{\ell/2}^u [\sqrt \ell {\bf e}(v/\ell)]^{-2}\dd v]}\dd u
= \frac{\ell^{1/(2\alpha)}Y}{(\beta+2-d)^2},
\end{align*}
whence
$$
q(A)=\int_0^\infty \frac{\dd \ell}{2\sqrt{2\pi\ell^3}} 
\PR\Big( \frac{{\mathfrak c}\ell^{1/(2\alpha)}}{(\beta+2-d)^2} Y\in A  \Big)
= \int_0^\infty \frac{{\mathfrak a}\dd u}{u^{1+\alpha}} \PR(u Y \in A).
$$

Let us check that $(S_t)_{t\geq 0}$ is $\alpha$-stable, i.e. that
its L\'evy measure $q$ satisfies $q( A_c)=c^\alpha q(A)$,
for all $A \in \cB(\rd\setminus\{0\})$, all $c>0$, where we have set
$A_c=\{z \in \rd : cz \in A   \}$. But
$$
q(A_c)=\int_0^\infty \frac{{\mathfrak a}\dd u}{u^{1+\alpha}} \PR(c u Y \in A)
=c^{\alpha}\int_0^\infty \frac{{\mathfrak a}\dd u}{u^{1+\alpha}} \PR(u Y \in A)=c^\alpha q(A).
$$

We now turn to point (ii). If $\gamma\equiv 1$, then $M_\beta=m_\beta=0$, so that
the announced convergence to $(S_t)_{t\geq 0}$ follows from 
Proposition \ref{pfini}. Moreover, $(S_t)_{t\geq 0}$ is radially symmetric
by definition, recalling Proposition \ref{pfini}, Lemma \ref{ppfini} and
that $\bN(\dd s,\dd e,\dd \theta)$ is a Poisson measure with intensity $\dd s \Xi(\dd e) \Lambda(\dd \theta)$
and observing that $\Lambda\in \cP(\cH)$ is the law of $\hTheta^\star$, which is nothing but a
stationary $\Sd$-valued Brownian motion (because $\gamma\equiv 1$, see Lemma \ref{ergo2}).
\end{proof}

We can finally handle the

\begin{proof}[Proof of Theorem \ref{mr}-(c)-(d)-(e)-(f)]
Points (c)-(e)-(f) immediately follow from Propositions \ref{pfini} and \ref{tpfini}.
For point (d), which concerns the case where $\beta=1+d$, we know that
$(\e[X_{t/\e}-M_\beta \zeta_\e t/\e])_{t\geq 0} \stackrel{f.d.} \longrightarrow (S_t)_{t\geq 0}$,
where $(S_t)_{t\geq 0}$ is a $1$-stable L\'evy process.
We claim that under the additional condition $\int_1^\infty r^{-1}|[\Gamma(r)]^{-1}r-1|\dd r<\infty$, there 
is $b\in\rr$ such that 
\begin{equation}\label{oob}
\lim_{\e\to 0} \Big(\zeta_\e - \frac{1}{9\kappa}|\log \e|\Big)=b,
\end{equation}
whence $(\e[X_{t/\e}- M_\beta |\log \e| t/(9\kappa\e)])_{t\geq 0} \stackrel{f.d.} 
\longrightarrow (S_t+b M_\beta t)_{t\geq 0}$. This completes 
the proof because the L\'evy process $(S_t+bM_\beta t)_{t\geq 0}$ is also a $1$-stable.

\vip

To check \eqref{oob}, we recall Notation \ref{fap} to write
$\zeta_\e=C_\e/D_\e$, where
$$
C_\e=\int_{-\infty}^{1/a_\e} h^{-1}(w)[\sigma(w)]^{-2}\dd w
\quad \hbox{and}\quad
D_\e= \int_{-\infty}^{1/a_\e} [\sigma(w)]^{-2}\dd w.
$$
By Lemma \ref{fcts}-(i)-(vi), we have $|D_\e - \kappa|\leq C \int_{1/a_\e}^\infty (1+|w|)^{-4/3 }\dd w
\leq C a_\e^{1/3} \leq C \e^{1/3}$ since $a_\e=\kappa\e$.

\vip

We thus only have to verify that $\lim_{\e \to 0} (C_\e - |\log \e|/9)$ exists. 
Recalling Notation \ref{notaf} and using the substitution $r=h^{-1}(w)$, we find
\begin{align*}
C_\e=\int_{0}^{h^{-1}(1/a_\e)} r[h'(r)]^{-1}\dd r=\frac13\int_0^{A_\e} r^d [\Gamma(r)]^{-1-d}\dd r
\end{align*}
where we have set $A_\e=h^{-1}(1/a_\e)$. Since $h(r)=3\int_{r_0}^r u^{1-d}[\Gamma(u)]^{1+d} \dd u \sim r^3$
as $r\to \infty$ and since $a_\e=\kappa \e$, we deduce that $A_\e\sim [\kappa \e]^{-1/3}$ as 
$\e\to 0$, so that $\lim_{\e \to 0} (|\log \e|/9 - (\log A_\e)/3)= (\log\kappa)/9$,
and we are reduced to check that $\lim_{\e\to 0} (C_\e - (\log A_\e)/3)$ exists. But
\begin{align*}
C_\e-\frac13\log A_\e= \frac13
\int_0^{A_\e} \Big[\Big(\frac r {\Gamma(r)}\Big)^{1+d} - \indiq_{\{r\geq 1\}}\Big]\frac{\dd r}r
\to \frac13
\int_0^\infty \Big[\Big(\frac r {\Gamma(r)}\Big)^{1+d} - \indiq_{\{r\geq 1\}}\Big]\frac{\dd r}r
\end{align*}
as $\e\to 0$. This last quantity is well-defined and finite, because $\Gamma:[0,\infty)\mapsto (0,\infty)$
is bounded from below, because $\Gamma(r)\sim r$ as $r\to \infty$, and because 
$\int_1^\infty r^{-1}|(r/\Gamma(r))-1|\dd r<\infty$ by assumption.
\end{proof}

\begin{rk}\label{rsc}
In Theorem \ref{mr}-(d), i.e. in the critical case $\beta=1+d$, the constant $c$ is given by $c=1/(9\kappa)=
(3\int_0^\infty u^{d-1}[\Gamma(u)]^{-1-d}\dd u)^{-1}$ by Lemma \ref{fcts}-(i).
\end{rk}

\section{The integrated Bessel regime}\label{Sb}

Here we give the proof of Theorem \ref{mr}-(g). 
We first define properly the limit process $(\cV_t)_{t\geq 0}$.

\begin{defi}\label{ddd} 
We fix $\beta \in (d-2,d)$ and 
consider a Bessel process $(\cR_t)_{t\geq 0}$ starting from $0$
with dimension $d-\beta \in (0,2)$, as well as an i.i.d. family $\{(\hTheta^{\star,i}_t)_{t\in \rr}, i \geq 1\}$
with common law $\Lambda$, see Lemma \ref{ergo2}-(iii), independent of $(\cR_t)_{t\geq 0}$.
We set $\cZ=\{t\geq 0 : \cR_t=0\}$ and we write $\cZ^c=\cup_{i\geq 1}(\ell_i,r_i)$
as the (countable) union of its connected components: for all $i\geq 1$, 
we have $\cR_{\ell_i}=\cR_{r_i}=0$ and $\cR_t>0$ for all $t\in (\ell_i,r_i)$.
We then define
$$
\cV_t = \sum_{i\geq 1} \indiq_{\{t \in (\ell_i,r_i)\}}\cR_t \hTheta^{\star,i}_{[\int_{(\ell_i+r_i)/2}^t\cR_s^{-2}\dd s]}.
$$
\end{defi} 

\begin{rk}
In some sense to be precised, $(\cV_t)_{t\geq 0}$ is the unique (in law) solution to
$$
\cV_t=B_t - \frac\beta 2 \intot \cF(\cV_s)\dd s,
$$
where $(B_t)_{t\geq 0}$ is a $d$-dimensional Brownian motion and where $\cF(v)=\cU^{-1}(v)\nabla \cU(v)$,
with $\cU(v)=|v| \gamma(v/|v|)$ (if $\gamma \equiv 1$, one finds $\cF(v)=|v|^{-2}v$).
This equation is what one gets when informally searching
for the limit of $\sqrt\e V_{t/\e}$ as $\e\to 0$, $(V_t)_{t\geq 0}$ being the solution to \eqref{eds}.
But it is not clearly well-defined because $\cF$ is singular at $0$. 
See \cite[Section 6]{fj} for the detailed study of such an equation in dimension $d=2$ and when $\gamma\equiv 1$.
\end{rk}

We now introduce some notation that will be used during the whole section.
We fix $\beta\in (d-2,d)$, recall Notation \ref{notaf} and set, for $\e\in(0,1)$,
$$
a_\e= \e^{(\beta+2-d)/2}.
$$
We consider a one-dimensional Brownian motion $(W_t)_{t\geq 0}$, set 
$A^\e_t= \e a_\e^{-2}\intot [\sigma(W_s/a_\e)]^{-2}\dd s$, introduce its inverse $\rho^\e_t$ and put
$R^\e_t = \sqrt \e h^{-1}(W_{\rho^\e_t}/a_\e)$
and $T^\e_t=\int_0^t [R^\e_s]^{-2}\dd s$.
We also consider the solution $(\hTheta_t)_{t\geq 0}$
of \eqref{eqt}, independent of $(W_t)_{t\geq 0}$.

\begin{lem}\label{repul}
For all $\e\in(0,1)$, $(\sqrt \e V_{t/\e})_{t\geq 0} \stackrel{d}=(R^\e_t \hTheta_{T^\e_t})_{t\geq 0}$,
for $(V_t)_{t\geq 0}$ the velocity process of \eqref{eds}.
\end{lem}

\begin{proof}
We know from Lemmas \ref{rep} and \ref{exr} that setting $S^\e_t= \e^{-1/2}R^\e_{\e t}$
and $\bar T^\e_t=\int_0^t [S^\e_s]^{-2} \dd s$, it holds that
$(S^\e_t \hTheta_{\bar T^\e_t})_{t\geq 0} \stackrel{d}= (V_t)_{t\geq 0}$, whence
$(\sqrt\e S^\e_{t/\e} \hTheta_{\bar T^\e_{t/\e}})_{t\geq 0} \stackrel{d}= (\sqrt\e V_{t/\e})_{t\geq 0}$.
To conclude, it suffices to observe that
$\sqrt\e S^\e_{t/\e}=R^\e_t$ and that $\bar T^\e_{t/\e}=\int_0^{t/\e} [\e^{-1/2}R^\e_{\e s}]^{-2} \dd s
=\int_0^t [R^\e_s]^{-2}\dd s =T^\e_t$.
\end{proof}

We first study the convergence of the radius process.

\begin{lem}\label{rays}
There is a Bessel process $(\cR_t)_{t\geq 0}$ with dimension $d-\beta$ issued from $0$ such that
$(R^\e_t)_{t\geq 0}$ a.s. converges to $(\cR_t)_{t\geq 0}$, uniformly on compact time intervals.
\end{lem}

\begin{proof}
Since $[\sigma(w)]^{-2} \leq C (1+|w|)^{-2(\beta+1-d)/(\beta+2-d)}$ by
Lemma \ref{fcts}-(vi) and  since
$$
\lim_{\e\to 0} \e [a_\e \sigma(w/a_\e)]^{-2} = (\beta+2-d)^{-2} w^{-2(\beta+1-d)/(\beta+2-d)}\indiq_{\{w>0\}}
$$
by Lemma \ref{fcts}-(xi), since moreover we a.s. have, for all $T>0$, 
$\int_0^T |W_s|^{-2(\beta+1-d)/(\beta+2-d)}\dd s <\infty$ because $2(\beta+1-d)/(\beta+2-d)<1$, 
we conclude, by dominated convergence, 
that a.s., for all $t\geq 0$, $(A^\e_t)_{t\geq 0}$ converges to  
$$
A_t=(\beta+2-d)^{-2} \intot W_s^{-2(\beta+1-d)/(\beta+2-d)}\indiq_{\{W_s>0\}} \dd s.
$$
Let $\rho_t=\inf\{s>0 : A_s>t\}$ be its generalized inverse and let $J=\{t>0 : \rho_t>\rho_{t-}\}$.
We now verify that a.s., for all $T>0$,
\begin{equation}\label{ggg}
\lim_{\e\to 0}\sup_{u\in[0,T]} |(W_{\rho^\e_u})_+-(W_{\rho_u})_+| =0.
\end{equation}

(a) By Lemma \ref{tc}, we know that a.s., for all $t\in [0,\infty)\setminus J$, $\rho^\e_t \to \rho_t$.

\vip
  
(b) We a.s. have, for all $t\geq 0$, $A_{\rho_{t-}}=A_{\rho_t}=t$ (since $A$ is continuous) and
$$
\rho_{A_t}= \inf\{s>t : W_s>0\}=\left\{\begin{array}{lll} t & \hbox{ if } &W_t \geq 0, \\ \inf\{s>t : W_s=0\}&
\hbox{ if } & W_t<0.  \end{array} \right. 
$$
Indeed, the second equality is clear and, setting $\nu_t= \inf\{s>t : W_s>0\}$, it holds that $\rho_{A_t}=
\inf\{s>0 : A_s>A_t\}=\inf\{s>t : A_s>A_{\nu_t}\}$ (because $A_{\nu_t}=A_t$ by definition of $A$),
whence clearly $\rho_{A_t}=\inf\{s>t : W_s>0\}$ (again by definition of $A$).

\vip

(c) Since $A$ is continuous, we deduce from (a) that a.s., for a.e. $t\geq 0$, $A_{\rho^\e_t}\to A_{\rho_t}$.
Since moreover $t\mapsto A_{\rho_t}$ is a.s. continuous (by (b)) and nondecreasing (as well as
$t\mapsto A_{\rho_t^\e}$ for each $\e>0$), we conclude from the Dini
theorem that a.s., $\sup_{[0,T]} |A_{\rho^\e_t}- A_{\rho_t}| \to 0$.

\vip

(d) By (b), we a.s. have $(W_u)_+=W_{\rho_{A_u}}$ for all $u\geq 0$.

\vip

(e) Almost surely, $u\mapsto W_{\rho_u}$ is nonnegative and continuous. First, by (b),
we have $W_{\rho_u}=W_{\rho_{A_{\rho_u}}}$, which is nonnegative by (d). Next, it suffices to prove that
$W_{\rho_{u-}}=W_{\rho_u}$ for all $u\geq 0$. Setting $t=\rho_{u-}$, we see that $W_t=W_{\rho_{A_t}}$ (by (b)
and since $W_t\geq 0$). Hence $W_t=W_{\rho_{A_{\rho_{u-}}}}=W_{\rho_{A_{\rho_{u}}}}$ by (b), whence $W_t=W_{\rho_u}$
as desired.

\vip

(f) To complete the proof of \eqref{ggg}, it suffices to note that $(W_{\rho^\e_u})_+-(W_{\rho_u})_+=
W_{\rho_{A_{\rho^\e_u}}}-W_{\rho_u}$ by (d) and (e), that $u\mapsto W_{\rho_u}$ is continuous by (e), and finally to use
point (c). 

\vip

By Lemma \ref{fcts}-(x), we have $\sqrt\e h^{-1}(w/a_\e) \to  w_+^{1/(\beta+2-d)}$, uniformly on 
compact subsets of $\rr$. Together with \eqref{ggg}, this implies that
$(R^\e_t=\sqrt \e h^{-1}(W_{\rho^\e_t}/a_\e))_{t\geq 0}$ a.s. converges, 
uniformly on compact time intervals, to
$((W_{\rho_t})_+^{1/(\beta+2-d)})_{t\geq 0}$, which is a Bessel process with dimension $d-\beta$
issued from $0$ by Lemma \ref{bess}.
\end{proof}

We can now give the

\begin{proof}[Proof of Theorem \ref{mr}-(g)]
Our goal is to verify that 
$(R^\e_t \hTheta_{T^\e_t})_{t\geq 0}$
goes in law to  $(\cV_t)_{t\geq 0}$, for the usual convergence of continuous processes.
This implies that $(\e^{3/2}X_{t/\e})_{t\geq 0}$ goes in law to $(\intot \cV_s\dd s)_{t\geq 0}$, since
by Lemma \ref{repul}, $(\e^{3/2}X_{t/\e}=\e^{3/2}x_0+\intot \sqrt\e V_{s/\e} \dd s)_{t\geq 0}$ 
has the same law as $(\e^{3/2}x_0+\intot R^\e_s \hTheta_{T^\e_s} \dd s)_{t\geq 0}$.
We already know from Lemma \ref{rays} that a.s., $\sup_{[0,T]} |R^\e_t-\cR_t|\to 0$ for all $T>0$,
where $(\cR_t)_{t\geq 0}$ is a Bessel process as in Definition \ref{ddd} and we introduce 
$\cZ=\{t\geq 0 : \cR_t=0\}$ and write $\cZ^c=\cup_{i\geq 1}(\ell_i,r_i)$
with, for all $i\geq 1$, $\cR_{\ell_i}=\cR_{r_i}=0$ and $\cR_t>0$ for all $t\in (\ell_i,r_i)$.
Finally, we set $\cW=\sigma(W_s,s\geq 0)$ and observe that
$\cW=\sigma(R^\e_t,\cR_t,t\geq 0,\e\in(0,1))$ is independent of $(\hTheta_t)_{t\geq 0}$.

\vip

\underline{\it Step 1.} For all $i>j\geq 1$, we have $\lim_{\e\to 0} (\tau_i^\e-\tau_j^\e)=\infty$ a.s.,
where we have set
$$
\tau_i^\e= T^\e_{(\ell_i+r_i)/2}=\int_0^{(\ell_i+r_i)/2} \frac{\dd s}{[R^\e_s]^2}.
$$
Indeed, by the Fatou Lemma, we know that a.s.,
$$
\liminf_{\e\to 0}(\tau_i^\e-\tau_j^\e)\geq \int_{(\ell_j+r_j)/2}^{(\ell_i+r_i)/2} \frac{\dd s}{[\cR_s]^2}
\geq \int_{\ell_i}^{(\ell_i+r_i)/2} \frac{\dd s}{[\cR_s]^2}=\infty
$$
by Lemma \ref{bess}-(ii).

\vip

\underline{\it Step 2.} For $T>0$ and $\delta>0$, we consider the (a.s. finite) set of indices
$$
\cI_{\delta,T}=\Big\{ i \geq 1 : \ell_i\leq T \;\; \hbox{ and } \sup_{s\in (\ell_i,r_i)}\cR_s >\delta\Big\}
$$
and for $i\in \cI_{\delta,T}$, we introduce $\ell_i<\ell_i^\delta<r_i^\delta<r_i$ defined by
$$
\ell_i^\delta=\inf\{s>\ell_i : \cR_s>\delta\}\quad\hbox{and}\quad r_i^\delta=\sup\{s<r_i : \cR_s>\delta\}.
$$
We also set
$$
A_{\delta,T}= 2 \max_{i\in \cI_{\delta,T}} \Big[\Big|\int_{(\ell_i+r_i)/2}^{\ell_i^\delta}\frac{\dd s}{[\cR_s]^2}\Big|  + 
\Big|\int_{(\ell_i+r_i)/2}^{r_i^\delta}\frac{\dd s}{[\cR_s]^2}\Big|\Big].
$$
By Lemma \ref{ergo2}-(iv), conditionally on $\cW$, we can find an i.i.d. family
$((\hTheta^{\star,i,\e,\delta}_t)_{t\in \rr})_{i\in \cI_{\delta,T}}$ of $\Lambda$-distributed processes
such that, setting 
$$
\Omega_{\e,\delta,T}=\Big\{\forall i\in \cI_{\delta,T}, \;\; (\hTheta^{\star,i,\e,\delta}_t)_{t\in [-A_{\delta,T},A_{\delta,T}]}
=(\hTheta_{(\tau_i^\e+t)\lor 0})_{t\in [-A_{\delta,T},A_{\delta,T}]}\},
$$
we have $\Pr(\Omega_{\e,\delta,T}|\cW)={\bf p}_{\delta,T}(\e)$,
where ${\bf p}_{\delta,T}(\e)=p_{A_{\delta,T}}(\tau^\e_{i_1},\tau^\e_{i_2}-\tau^\e_{i_1},\dots,\tau^\e_{i_n}-\tau^\e_{i_{n-1}})$
and where we have written $\cI_{\delta,T}=\{i_1,\dots,i_n\}$.
We know that ${\bf p}_{\delta,T}(\e)$ a.s. tends to $1$ as $\e\to 0$, so that
$r_{\delta,T}(\e)=\PR(\Omega_{\e,\delta,T})=\E[{\bf p}_{\delta,T}(\e)]$
also tends to $1$ as $\e\to 0$.

\vip

\underline{\it Step 3.} Conditionally on $\cW$, we then consider an i.i.d. family 
$((\hTheta^{\star,i,\e,\delta}_t)_{t\in \rr})_{i\in \nn^* \setminus \cI_{\delta,T}}$, independent of 
$((\hTheta^{\star,i,\e,\delta}_t)_{t\in \rr})_{i\in \nn^* \setminus \cI_{\delta,T}}$, and we consider the process
$(\cV_t^{\e,\delta})_{t\geq 0}$ built from $(\cR_t)_{t\geq 0}$ and the i.i.d. family 
$((\hTheta^{\star,i,\e,\delta}_t)_{t\in \rr})_{i\geq 1}$ as in Definition \ref{ddd}, that is,
$$
\cV_t^{\e,\delta}=\sum_{i\geq 1} \indiq_{\{t \in (\ell_i,r_i)\}}\cR_t \hTheta^{\star,i,\e,\delta}_{[\int_{(\ell_i+r_i)/2}^t\cR_s^{-2}\dd s]}.
$$
For all $\e\in(0,1)$ and all $\delta \in (0,1)$, $(\cV_t^{\e,\delta})_{t\geq 0} \stackrel{d}=
(\cV_t)_{t\geq 0}$. We will show that for all $\eta>0$,
$$
\lim_{\delta\to 0} \limsup_{\e\to 0} \PR[\Delta_{\e,\delta,T}>\eta]=0 \quad \hbox{where}\quad \Delta_{T,\delta,\e}= 
\sup_{[0,T]} \Big|R^\e_t \hTheta_{T^\e_t}- \cV_t^{\e,\delta} \Big|
$$
and this will conclude the proof. 
Recalling that $|\cV_t^{\e,\delta}|=\cR_t$,
\begin{align*}
\Delta_{\e,\delta,T}\leq \sup_{[0,T]} |R^\e_t - \cR_t|
+ \sup_{[0,T]} \Big|\cR_t\hTheta_{T^\e_t}- \cV_t^{\e,\delta} \Big|\indiq_{\{\cR_t\leq \delta\}}
+\sup_{[0,T]} \Big|\cR_t \hTheta_{T^\e_t}- \cV_t^{\e,\delta} \Big|\indiq_{\{\cR_t> \delta\}}.
\end{align*}
We already know that the first term a.s. tends to $0$ as $\e\to 0$, the second one is bounded by $2\delta$ and
the third one is bounded by $(\sup_{[0,T]}\cR_t)\Delta'_{\e,\delta,T}$, where
$\Delta'_{\e,\delta,T}=\sup_{[0,T]} |\hTheta_{T^\e_t}- \cR_t^{-1}\cV_t^{\e,\delta} \Big|\indiq_{\{\cR_t> \delta\}}$.
All in all, we only have to check that 
$\lim_{\delta\to 0} \limsup_{\e\to 0} \PR[\Delta'_{\e,\delta,T}>\eta]=0$.

\vip

\underline{\it Step 4.} For all $t\in [0,T]$, $\cR_t>\delta$ implies that 
$t\in \cup_{i\in \cI_{\delta,T}}(\ell_i^\delta,r_i^\delta)$, whence
$$
\cR_t^{-1}\cV_t^{\e,\delta}- \hTheta_{T^\e_t}=\sum_{i\in \cI_{\delta,T}} \indiq_{\{t \in (\ell_i^\delta,r_i^\delta)\}}
\Big(\hTheta^{\star,i,\e,\delta}_{[\int_{(\ell_i+r_i)/2}^t\cR_s^{-2}\dd s]} - \hTheta_{[\tau_i^\e+\int_{(\ell_i+r_i)/2}^t [R_s^\e]^{-2}\dd s]}\Big).
$$
because $T^\e_t=\tau_i^\e+\int_{(\ell_i+r_i)/2}^t [R_s^\e]^{-2}\dd s$.
For $x\in (0,1)$, we have $\lim_{\e\to 0}\PR(\Omega'_{\e,\delta,T}(x))=1$, where
$$
\Omega'_{\e,\delta,T}(x)=\Big\{\forall i\in \cI_{\delta,T}, \;\;\forall t\in (\ell_{i}^\delta,r_i^\delta), \;\;
\Big|\int_{(\ell_i+r_i)/2}^t \cR_s^{-2}\dd s - \int_{(\ell_i+r_i)/2}^t [R_s^\e]^{-2}\dd s\Big| \leq x \Big\}.
$$
Indeed, for each $i\in \cI_{\delta,T}$,
$\cR_s$ is continuous and positive on $(\ell_{i}^\delta,r_i^\delta)$ and we have already seen that
$\lim_{\e\to 0} \sup_{[0,T]} |R^\e_t-\cR_t|=0$. For the same reasons, it holds that
$\lim_{\e\to 0}\PR(\Omega''_{\e,\delta,T})=1$
$$
\Omega''_{\e,\delta,T}=\Big\{\forall i\in \cI_{\delta,T}, \;\;\forall t\in (\ell_{i}^\delta,r_i^\delta), \;\;
\Big|\int_{(\ell_i+r_i)/2}^t \cR_s^{-2}\dd s\Big|\lor
\Big|\int_{(\ell_i+r_i)/2}^t [R_s^\e]^{-2}\dd s\Big| \leq A_{\delta,T} \Big\}.
$$
Now on $\bar \Omega_{\e,\delta,T}(x)= \Omega_{\e,\delta,T}\cap \Omega'_{\e,\delta,T}(x)\cap \Omega''_{\e,\delta,T}$, 
we have, for any $t\in [0,T]$,
$$
(\cR_t^{-1}\cV_t^{\e,\delta}- \hTheta_{T^\e_t})\indiq_{\{\cR_t> \delta\}}=
 \sum_{i\in \cI_{\delta,T}} \indiq_{\{t \in (\ell_i^\delta,r_i^\delta)\}}
\Big(\hTheta^{\star,i,\e,\delta}_{[\int_{(\ell_i+r_i)/2}^t\cR_s^{-2}\dd s]} 
- \hTheta^{\star,i,\e,\delta}_{[\int_{(\ell_i+r_i)/2}^t [R_s^\e]^{-2}\dd s]}\Big),
$$
whence
$$
\Delta'_{\e,\delta,T}
\leq \#(\cI_{\delta,T}) \sup\Big\{\Big|\hTheta^{\star,i,\e,\delta}_a-\hTheta^{\star,i,\e,\delta}_b\Big| \; : \;
i\in\cI_{\delta,T}, \;a,b\in [-A_{\delta,T},A_{\delta,T}], \;|a-b|<x\Big\}=M_{\delta,T}^\e(x),
$$
the last equality standing for a definition. But the law of $M_{\delta,T}^\e(x)$ does not depend on $\e\in (0,1)$
(because conditionally on $\cW$, the family $((\hTheta^{\star,i,\e,\delta}_t)_{t\in\rr})_{i\in \cI_{\delta,T}}$ is i.i.d.
and $\Lambda$-distributed. All in all, we have proved that for all $\delta>0$, all $T>0$, all $\eta>0$,
$x>0$, with a small abuse of notation,
$$
\limsup_{\e\to 0} \PR(\Delta'_{\e,\delta,T}>\eta)\leq \PR(M_{\delta,T}(x)>\eta)+ \limsup_{\e\to 0}
\PR((\bar \Omega_{\e,\delta,T}(x))^c) = \PR(M_{\delta,T}(x)>\eta)
$$
But $\lim_{x\to 0} \PR(M_{\delta,T}(x)>\eta)=0$, because the
$\Lambda$-distributed processes are continuous. We thus have 
$\limsup_{\e\to 0} \PR(\Delta'_{\e,\delta,T}>\eta)=0$
for each $\delta>0$, which completes the proof.
\end{proof}

%CI DESSOUS, VIRER PROBABLEMENT LE CAS $\beta \leq d$.
%We now set $b_\beta=1$ if $\beta\in(0,d]$ and $b_{\beta}=[\int_0^\infty [\Gamma(r)]^{-\beta}r^{d-1}\dd r]^{-1}$ 
%if $\beta>d$. We introduce the  measure $\nu'_\beta(\dd r)
%=b_\beta[\Gamma(r)]^{-\beta}r^{d-1}\dd r$ on $(0,\infty)%,
%which has a finite mean $m_\beta'=\int_0^\infty r \nu'_\beta(\dd r)>0$ if $\beta>1+d$.
%\vip
%Set $c_\beta=a_\beta$ if $\beta\in(0,d]$ and recall that $c_\beta = [\intrd[U(v)]^{-\beta}\dd v]^{-1}$ 
%if $\beta>d$. We
%introduce $\mu_\beta(\dd v)=c_\beta [U(v)]^{-\beta}\dd v$ on $\rd$.
%\vip
%For any $\beta \in (0,d]$, it holds that $c_\beta=a_\beta b_\beta$ and for any measurable
%$\varphi:\rd\mapsto\rr_+$, we have 
%$$
%\intrd \varphi(v)\mu_\beta(\dd v) = \int_0^\infty \int_{\Sd} \varphi(r\theta) 
%\nu_\beta(\dd \theta) \nu_\beta'(\dd r).
%$$
%In particular, we have $m_\beta = M_\beta m'_\beta$ for all $\beta>1+d$.

\section{The diffusive regime}\label{Sd}

The goal of this section is to prove Theorem \ref{mr}-(a). As already mentioned, 
this regime is almost treated in Pardoux-Veretennikov \cite{pv}, which consider much more general
problems. However, we can not strictly apply their result because $F$ is not locally bounded
(except if $\gamma\equiv 1$). Moreover, our proof is much simpler (because our model is much simpler).
First, we adapt to our context a Poincar\'e inequality found in Cattiaux-Gozlan-Guillin-Roberto \cite{cggr}.

\begin{lem}\label{poinca}
For any $\beta>2+d$,
there is a constant $C>0$ such that for all $f \in H^1_{loc}(\rd)\cap L^1(\rd,\mu_\beta)$ satisfying
$\intrd f(v) \mu_\beta(\dd v)=0$,
$$
\intrd [f(v)]^2 (1+|v|)^{-2}\mu_\beta(\dd v) \leq C \intrd |\nabla f(v)|^2 \mu_\beta(\dd v).
$$
\end{lem}

\begin{proof}
The constants below are allowed to depend only on $U$, $\beta$ and $d$.
By Assumption \ref{as}, there are
$0<C_1<C_2$ such that
$C_1 (1+|v|)^{-\beta}\dd v \leq \mu_\beta(\dd v)\leq C_2(1+|v|)^{-\beta}\dd v$.

\vip

We know from \cite[Proposition 5.5]{cggr} that for any $\alpha>d$, there is a constant $C$ 
such that for $g \in H^1_{loc}(\rd)\cap L^1(\rd,(1+|v|)^{-\alpha}\dd v)$ 
satisfying $\intrd g(v) (1+|v|)^{-\alpha}\dd v=0$, we have the inequality
$\intrd [g(v)]^2 (1+|v|)^{-\alpha}\dd v \leq C \intrd |\nabla g(v)|^2 (1+|v|)^{2-\alpha}\dd v$.

\vip

For $f$ as in the statement, we apply this inequality with 
$\alpha=\beta+2>d$ and $g=f-a$, the constant $a\in \rr$ being such that
$\intrd g(v) (1+|v|)^{-\beta-2}\dd v=0$. We find that 
$\intrd [g(v)]^2 (1+|v|)^{-2-\beta}\dd v \leq C_3 \intrd |\nabla g(v)|^2 (1+|v|)^{-\beta}\dd v$.

\vip

But $\intrd f(v) \mu_\beta(\dd v)=0$, whence $a=-\intrd g(v) \mu_\beta(\dd v)$ and thus
\begin{align*}
a^2 \leq  C_2^2\Big[\intrd g(v)(1+|v|)^{-\beta} \dd v\Big]^2 
=  C_2^2\Big[\intrd (1+|v|)^{-\beta/2-1} g(v)(1+|v|)^{1-\beta/2} \dd v\Big]^2,
\end{align*}
whence $a^2 \leq C_2^2 C_4  \intrd [g(v)]^2(1+|v|)^{-\beta-2} \dd v$
by the Cauchy-Schwarz inequality, where the constant $C_4=\intrd (1+|v|)^{2-\beta} \dd v$ is finite
because $\beta>2+d$.

\vip

Using that $f^2 \leq 2g^2+2a^2$ and setting $C_5=\intrd (1+|v|)^{-2-\beta} \dd v$, we find that
\begin{align*}
\intrd [f(v)]^2 (1+|v|)^{-2}\mu_\beta(\dd v)\leq& 2C_2 \intrd [g(v)]^2(1+|v|)^{-2-\beta}\dd v
+2C_2 a^2 \intrd (1+|v|)^{-2-\beta} \dd v\\
\leq & 2C_2[1+ C_2^2 C_4 C_5]\intrd [g(v)]^2(1+|v|)^{-2-\beta}\dd v\\
\leq & 2C_2C_3[1+ C_2^2 C_4 C_5]\intrd |\nabla g(v)|^2(1+|v|)^{-\beta}\dd v\\
\leq & 2 C_1^{-1}C_2C_3[1+ C_2^2 C_4 C_5]\intrd |\nabla f(v)|^2 \mu_\beta(\dd v).
\end{align*}
We finally used that $\nabla g=\nabla f$. 
\end{proof}

We next state a lemma that will allow us to solve the Poisson equation 
$\cL f(v) = v-m_\beta$, where $\cL$ is the generator of the velocity process.
We state a slightly more general version, that will be needed when treating the critical case $\beta=4+d$

\begin{lem}\label{lmlm}
Suppose that 
$\beta>2+d$. Let $g: \rd\mapsto \rr$ be of class $C^\infty$ and satisfy 
\begin{equation}\label{condilm}
\intrd g(v) \mu_\beta(\dd v) =0 \quad \hbox{and}\quad \intrd [g(v)]^2 (1+|v|)^{2}\mu_\beta(\dd v)<\infty.
\end{equation}
There exists $f : \rd\setminus\{0\}\mapsto \rr$, of class $C^\infty$, such that
$\intrd |\nabla f (v)|^2 \mu_\beta(\dd v) <\infty$ and solving the equation
$\frac 12[\Delta f -\beta F \cdot \nabla f]=g$ on $\rd\setminus\{0\}$.
\end{lem}

\begin{proof}
We divide the proof in three steps.

\vip

\underline{\it Step 1.} We introduce the weighted Sobolev space $H^1_\beta=\{\varphi\in H^1_{loc}(\rd)\; 
: \; |||\varphi|||_\beta<\infty$ and $\intrd \varphi(v) \mu_\beta(\dd v)=0 \}$, where
$|||\varphi|||_\beta^2 = \intrd [\varphi(v)]^2 (1+|v|)^{-2}\mu_\beta(\dd v) +  \intrd |\nabla\varphi(v)|^2 
\mu_\beta(\dd v)$.
By the Lax-Milgram theorem, there is a unique $f \in H^1_\beta$
such that for all $\varphi \in H^1_\beta$, $\intrd \nabla f(v) \cdot \nabla \varphi(v) \mu_\beta(\dd v) = 
- 2 \intrd\varphi(v) g(v) \mu_\beta(\dd v)$.

\vip

Indeed, $H^1_\beta$ is Hilbert, the quadratic form 
$A(\varphi,\phi)=\intrd \nabla \varphi(v) \cdot \nabla \phi(v) \mu_\beta(\dd v)$ is continuous
on $H^1_\beta$, as well as the linear form $L(\varphi)= 2 \intrd\varphi(v) g(v) \mu_\beta(\dd v)$
(here we use the moment condition on $g$),
and $A$ is coercive 
(i.e. there is $c>0$ such that $A(\varphi,\varphi)\geq c|||\varphi|||_\beta$ for all $\varphi \in H^1_\beta$)
by Lemma \ref{poinca}.

\vip

\underline{\it Step 2.} Using that $\int_\rd g(v)\mu_\beta(\dd v)=0$, we deduce from Step 1 that
$\intrd \nabla f(v) \cdot \nabla \varphi(v) \mu_\beta(\dd v) = 
- 2 \intrd\varphi(v) g(v) \mu_\beta(\dd v)$ for all
$\varphi \in H^1_{loc}(\rd)$ with $|||\varphi|||_\beta<\infty$ (without the centering condition on
$\varphi$).

\vip

\underline{\it Step 3.} We can now apply Gilbarg-Trudinger \cite[Corollary 8.11 p 186]{gt}:
$F$ being of class $C^\infty$ on $\rd\setminus \{0\}$, as well as $g$, and $f$ being a weak solution to
$\frac 12[\Delta f -\beta F \cdot \nabla f]= g$, it is of class $C^\infty$ on $\rd\setminus\{0\}$.
More precisely, we fix $v \in \rd\setminus\{0\}$ and we apply the cited corollary
on the open ball $B(v,|v|/2)$ to conclude that $f$ is of class $C^\infty$ on $B(v,|v|/2)$.

\vip

\underline{\it Step 4.} We thus can proceed rigorously to some integrations by parts to deduce that for all
$\varphi \in C^\infty_c(\rd\setminus\{0\})$, recalling that
$\mu_\beta(\dd v)=c_\beta [U(v)]^{-\beta}\dd v$, we have 
$\intrd {\rm div}[(U(v))^{-\beta}\nabla f(v)] \varphi(v) \dd v = 2\intrd \varphi(v) g(v) [U(v)]^{-\beta}\dd v$.
Hence ${\rm div}[U^{-\beta}\nabla f]=
2 g U^{-\beta}$ on $\rd \setminus\{0\}$ by continuity, whence the conclusion, since
$F(v)=[U(v)]^{-1}\nabla U(v)$.
\end{proof}

We can now give the

\begin{proof}[Proof of Theorem \ref{mr}-(a)]
We fix $\beta>4+d$.
We consider, for each $i=1,\dots,d$, a $C^\infty$ function $f_i:\rd\setminus\{0\}\mapsto \rr$ such that 
$\intrd |\nabla f_i(v)|^2\mu_\beta(\dd v) <\infty$ and $\frac12[\Delta f_i(v) -\beta F(v) \cdot \nabla f_i(v)]
=v_i-m_\beta^i$, where $m_\beta^i=\intrd v_i \mu_\beta(\dd v)$ is the $i$-th coordinate of $m_\beta$.
Such a function $f_i$ exists by Lemma \ref{lmlm}, because $g_i(v)=v_i-m_\beta^i$ is $C^\infty$, 
$\mu_\beta$-centered and $\intrd [g_i(v)]^2(1+|v|)^{2}\mu_\beta(\dd v) <\infty$ because $\beta> 4+d$.

\vip

We now set $f=(f_1 \; f_2 \;\dots\; f_d)^*: \rd\mapsto \rd$ and apply the It\^o formula, which is licit
because $f$ is of class $C^\infty$ on $\rd\setminus\{0\}$ and because $(V_t)_{t\geq 0}$ never visits
$0$: recalling \eqref{eds} and that $\nabla^* f = (\nabla f_1 \; \nabla f_2 \;\dots\; \nabla f_d)^*$,
$$
f(V_t)=f(v_0)+\intot \nabla^* f(V_s) \dd B_s + \intot (V_s - m_\beta)\dd s =
f(v_0) + \intot  \nabla^* f(V_s) \dd B_s + X_t-m_\beta t -x_0.
$$
Hence we have $\sqrt\e(X_{t/\e} - m_\beta t/\e)=M^\e_t + Y^\e_t$, 
where $M^\e_t=-\sqrt\e \int_0^{t/\e} \nabla^* f(V_s) \dd B_s$
and where $Y^\e_t= \sqrt{\e}[x_0 + f(V_{t/\e}) - f(v_0)]$.

\vip

For each $t\geq 0$, $Y^\e_t$ goes to $0$ in law (and thus in probability) as $\e\to 0$:
this immediately follows from the fact that $f(V_{t/\e})$ converges in law as $\e\to 0$,
see Lemma \ref{ergo1}-(iii). It is not clear (and probably false)
that $\sup_{[0,t]} |Y^\e_s| \to 0$, which explains why we deal with finite-dimensional 
distributions.

\vip

Next, $(M^\e_t)_{t\geq 0}$ converges in law, in the usual sense of continuous processes,
to $(\Sigma B_t)_{t\geq 0}$, where $\Sigma \in \cS_d^+$ is the square root
of $\Sigma^2=\intrd \nabla^* f(v) \nabla f(v) \mu_\beta(\dd v) \in \cS_d^+$ (see below).
Indeed, since $(M^\e_t)_{t\geq 0}$ is a continuous $\rd$-valued martingale,
it suffices, by Jacod-Shiryaev \cite[Theorem VIII-3.11 p 473]{js}, to verify that
for all $i,j \in \{1,\dots,d\}$, $\langle M^{\e,i}, M^{\e,j}\rangle_t \to \Sigma^2_{ij} t$ in
probability for each $t\geq 0$. But this follows from 
the fact that $\langle M^{\e,i}, M^{\e,j}\rangle_t= \e \int_0^{t/\e} \nabla f_i(V_s)\cdot \nabla f_j(V_s) \dd s$,
from Lemma \ref{ergo1}-(ii)
and from the fact that $\intrd |\nabla f(v)|^2\mu_\beta(\dd v) <\infty$.

\vip

All this proves that indeed, $(\sqrt\e(X_{t/\e} - m_\beta t/\e))_{t\geq 0}$ converges, in the sense
of finite dimensional distributions, to $(\Sigma B_t)_{t\geq 0}$, as $\e\to 0$.

\vip

Let us finally explain why $\Sigma^2$ is positive definite.
For $\xi \in \rd \setminus\{0\}$, we have, setting $f_\xi(v)=f(v)\cdot \xi$,
$$
\xi^*\Sigma^2\xi = \intrd |\nabla f(v)\xi|^2\mu_\beta(\dd v)= \intrd |\nabla f_\xi(v)|^2\mu_\beta(\dd v),
$$
which is strictly positive because else we would have $\nabla f_\xi(v)=0$ for a.e. $v\in\rd$,
so that $f_\xi$ would be constant on $\rd\setminus\{0\}$ (recall that $f$ is smooth on $\rd\setminus\{0\}$).
This is impossible, because $\Delta f_\xi(v) - \beta F(v)\cdot \nabla f_\xi(v)= 2(v-m_\beta)\cdot \xi$ 
on $\rd\setminus \{0\}$ and because constants do not solve this equation.
\end{proof}

\begin{rk}\label{rkd} Consider some $\beta>4+d$.

\vip

(i) In Theorem \ref{mr}-(a), $\Sigma\in \cS_d^+$ is the square root of 
$\intrd \nabla^* f(v) \nabla f(v) \mu_\beta(\dd v)$, with $\mu_\beta$ defined in Remark \ref{inv} and
with $f=(f_1,\dots,f_d)$, where $f_i:\rd\setminus\{0\}\mapsto \rr$ is the (unique) $C^\infty$ solution
to $\frac12[\Delta f_i(v) -\beta F(v) \cdot \nabla f_i(v)]=v_i-m_\beta^i$ such that
$\intrd |\nabla f_i(v)|^2\mu_\beta(\dd v) <\infty$.

\vip

(ii)
If $U(v)=(1+|v|^2)^{1/2}$, then $\mu_\beta(\dd v)=c_\beta (1+|v|^2)^{-\beta/2}\dd v$
and $m_\beta=0$,
so that $(\sqrt \e X_{t/\e})_{t\geq 0} \stackrel{f.d.} \longrightarrow (\Sigma B_t)_{t\geq 0}$.
Furthermore, it holds that $f_i(v)= - a (|v|^2+3)v_i$, with $a=2/(3\beta-4-2d)$, and
a computation shows that $\Sigma=q I_d$, with 
\begin{align*}
q^2=&\int_{\rd}|\nabla f_1(v)|^2\mu_\beta(\dd v)=-\int_{\rd} f_1(v)\Big[\Delta f_1(v) -\beta F(v) 
\cdot \nabla f_1(v)\Big]   \mu_\beta(\dd v)=-2\int_{\rd} f_1(v)v_1\mu_\beta(\dd v)\\
=& 2ac_\beta \int_{\rd} (|v|^2+3)v_1^2 (1+|v|^2)^{-\beta/2}\dd v= \frac{2ac_\beta}d
\intrd (|v|^2+3)|v|^2(1+|v|^2)^{-\beta/2}\dd v.
\end{align*}
\end{rk}

\section{The critical diffusive regime}\label{Scd}

The goal of this section is to prove Theorem \ref{mr}-(b).
We have not been able to solve the Poisson equation, so that we adopt 
a rather complicated strategy. This would not be necessary if considering only the case
$U(v)=(1+|v|^2)^{1/2}$ where the solution to the Poisson equation is explicit:
we could omit Lemmas \ref{Psi} and \ref{dc2} below.

\begin{lem}\label{Psi}
Fix $\beta>0$.
There is $\Psi: \Sd\mapsto\rd$, of class $C^\infty$,
such that for all $\theta \in \Sd$, all $k=1,\dots,d$,
$$
\frac 12 \Delta_S \Psi_k (\theta)- \frac \beta 2 \frac{\nabla_S\gamma(\theta)}{\gamma(\theta)}
\cdot\nabla_S\Psi_k(\theta)= \frac 9 2 \Psi_k(\theta)+\theta_k.
$$
\end{lem}

\begin{proof}
By Aubin \cite[Theorem 4.18 p 114]{a}, for any $\lambda>0$ and any smooth 
$g:\Sd \mapsto \rr$, there is a unique smooth solution $f:\Sd\mapsto \rr$
to $\ddiv_S(\gamma^{-\beta} \nabla_S f)= 2 \gamma^{-\beta}(\lambda f + g)$. This uses that $\gamma^{-\beta}$ is smooth
and positive on $\Sd$. This equation rewrites as $\frac12\Delta_S f 
- \frac\beta 2\gamma^{-1}\nabla_S\gamma\cdot\nabla_S f
= \lambda f + g$.
Applying this result, for each fixed $k=1,\dots,d$, with $\lambda=9/2$ and $g(\theta)=\theta_k$, 
completes the proof.
\end{proof}

We now introduce some notation for the rest of the section.
We write $V_t=R_t\hTheta_{H_t}$ as in Lemma \ref{rep} and we set $\Theta_t=\hTheta_{H_t}$.
We know that $(R_t)_{t\geq 0}$ solves \eqref{eqr} for some one-dimensional Brownian motion $(\tB_t)_{t\geq 0}$,
that $(\Theta_t)_{t\geq 0}$ solves \eqref{thetapc} for some $d$-dimensional Brownian motion $(\bB_t)_{t\geq 0}$,
and that these two Brownian motions are independent.

\begin{lem}\label{dc1}
Assume that $\beta=4+d$ and consider
the function $\Psi$ introduced in Lemma \ref{Psi}. We have 
$R_t^3\Psi(\Theta_t)=r_0^3\Psi(\theta_0)-x_0+(X_t-m_\beta t)+M_t+Y_t$, where
\begin{align*}
M_t=&\intot R_s^2\nabla^*_S\Psi(\Theta_s)\dd \bar B_s+3\intot R_s^2\Psi(\Theta_s)\dd \tB_s,\\
Y_t=&m_\beta t + \frac{3(4+d)}{2}\intot \Big(R_s-\frac{R_s^2\Gamma'(R_s)}{\Gamma(R_s)} \Big)\Psi(\Theta_s) \dd s.
\end{align*}
\end{lem}

\begin{proof}
Applying It\^o's formula with the function $\Psi$
(extended to $\rd\setminus\{0\}$ as in Subsection \ref{notadebut} so that
we can use the usual derivatives of $\rd$), we find
\begin{align*}
\Psi(\Theta_t)=&\Psi(\theta_0)+\intot R_s^{-1}\nabla^*\Psi(\Theta_s)\pi_{\Theta_s^\perp}\dd \bar B_s
-\frac{d-1}{2}\intot R_s^{-2}\nabla^*\Psi(\Theta_s)\Theta_s\dd s \\
&- \frac\beta2\intot R_s^{-2}\nabla^*\Psi(\Theta_s)\pi_{\Theta_s^\perp} \frac{\nabla \gamma(\Theta_s)}
{\gamma(\Theta_s)}\dd s
+ \frac12 \intot R_s^{-2} \sum_{i,j=1}^d (\pi_{\Theta_s^\perp})_{ij}\partial_{ij}\Psi (\Theta_s) \dd s.
\end{align*}
But the way $\Psi$ has been extended to $\rd\setminus\{0\}$ implies that 
$\pi_{\theta^\perp}\nabla\Psi(\theta)=\nabla\Psi(\theta)=\nabla_S\Psi(\theta)$, that $\nabla^*\Psi(\theta)\theta=0$ 
and that
$\sum_{i,j=1}^d (\pi_{\theta^\perp})_{ij}\partial_{ij}\Psi(\theta)=\Delta \Psi(\theta)-\sum_{i,j=1}^d \theta_i\theta_j
\partial_{ij}\Psi(\theta)
=\Delta \Psi(\theta)=\Delta_S\Psi(\theta)$. Consequently,
\begin{align*}
\Psi(\Theta_t)=&\Psi(\theta_0)+\intot R_s^{-1}\nabla_S^*\Psi(\Theta_s)\dd \bar B_s
- \frac\beta2\intot R_s^{-2}\nabla_S^*\Psi(\Theta_s)\frac{\nabla_S \gamma(\Theta_s)}{\gamma(\Theta_s)}\dd s
+ \frac12 \intot R_s^{-2} \Delta_S\Psi (\Theta_s) \dd s\\
=&\Psi(\theta_0)+\intot R_s^{-1}\nabla_S^*\Psi(\Theta_s)\dd \bar B_s
+\intot R_s^{-2} \Big[\frac92 \Psi(\Theta_s) + \Theta_s\Big] \dd s.
\end{align*}
Recalling \eqref{eqr} and that $\beta=4+d$, It\^o's formula tells us that
\begin{align*}
R_t^3=&r_0^3+ 3 \intot R_s^2 \dd \tB_s +\frac{3(d-1)}2\intot R_s\dd s 
- \frac{3\beta}2\intot \frac{\Gamma'(R_s)R_s^2}{\Gamma(R_s)}\dd s + 3\intot R_s \dd s\\
=&r_0^3+ 3 \intot R_s^2 \dd \tB_s - \frac92\intot R_s \dd s 
+ \frac{3(4+d)}{2}\intot \Big(R_s-\frac{R_s^2\Gamma'(R_s)}{\Gamma(R_s)} \Big) \dd s.
\end{align*}
We conclude that
\begin{align*}
R_t^3\Psi(\Theta_t)=&r_0^3\Psi(\theta_0)+\intot R_s^{2}\nabla_S^*\Psi(\Theta_s)\dd \bar B_s+
\intot R_s \Big[\frac92 \Psi(\Theta_s) + \Theta_s\Big] \dd s\\
&+3 \intot R_s^2 \Psi(\Theta_s) \dd \tB_s - \frac92\intot R_s \Psi(\Theta_s) \dd s 
+ \frac{3(4+d)}{2}\intot \Big(R_s-\frac{R_s^2\Gamma'(R_s)}{\Gamma(R_s)} \Big) \Psi(\Theta_s) \dd s.
\end{align*}
In other words, we have $R_t^3\Psi(\Theta_t)=r_0^3\Psi(\theta_0)-x_0+(X_t-m_\beta t)+M_t+Y_t$
as desired.
\end{proof}

We now treat the error term.

\begin{lem}\label{dc2}
Adopt the assumptions and notation of Lemma \ref{dc1}. Suppose the additional condition 
$\int_1^\infty r^{-1}|r\Gamma'(r)/\Gamma(r)-1|^2r^{-1}\dd r<\infty$. For each $t\geq 0$, in probability,
$$\lim_{\e\to 0}  |\log \e|^{-1/2}\e^{1/2} [R_{t/\e}^3 \Psi(\Theta_{t/\e})-r_0^3\Psi(\theta_0)+x_0-Y_{t/\e}]=0.$$ 
\end{lem}

\begin{proof}
First,
$\lim_{\e\to 0}  |\log \e|^{-1/2}\e^{1/2} [R_{t/\e}^3 \Psi(\Theta_{t/\e})-r_0^3\Psi(\theta_0)+x_0]=
\lim_{\e\to 0}  |\log \e|^{-1/2}\e^{1/2}\psi(V_{t/\e})=0$ in probability, where we have set
$\psi(v)=|v|^3\Psi(v/|v|)-r_0^3\Psi(\theta_0)+x_0$, because we know from Lemma \ref{ergo1}-(ii)
that $V_t$ converges in law as $t\to \infty$.

\vip

Next, we have $Y_t=\intot g(V_s)\dd s$, where 
we have set 
$$
g(v)=m_\beta + \frac{3(4+d)}{2} \Big(r-\frac{r^2\Gamma'(r)}{\Gamma(r)} \Big)\Psi(\theta),
$$ 
where $r=|v|$ and $\theta=v/|v|$. This function is of class $C^\infty$ on $\rd\setminus\{0\}$ and,
as we will see below,
$$
\hbox{(a)}\quad  \intrd |g(v)|^2(1+|v|)^{2}\mu_\beta(\dd v) < \infty
\quad \hbox{and} \quad \hbox{(b)}\quad \intrd g(v)\mu_\beta(\dd v)=0.
$$
Applying Lemma \ref{lmlm} (coordinate by coordinate), there exists $f:\rd\setminus\{0\}\mapsto \rd$
of class $C^\infty$, satisfying $\intrd |\nabla f(v)|^2 \mu_\beta(\dd v)<\infty$ and, for each $k=1,\dots,d$,
$\frac 12[\Delta f_k - \beta F \cdot \nabla f_k]=g_k$. By It\^o's formula, starting from \eqref{eds}, 
$$
f(V_t)=f(v_0)+N_t+Y_t \quad \hbox{where} \quad N_t=\intot \nabla^*f(V_s)\dd B_s.
$$
To conclude that $|\log \e|^{-1/2}\e^{1/2} Y_{t/\e} \to 0$ in probability,
we observe that $|\log \e|^{-1/2}\e^{1/2} [f(V_{t/\e})-f(v_0)]$ tends to $0$
in probability, which follows from the fact that  $V_t$ converges in law as $t\to \infty$,
and that 
$|\log \e|^{-1/2}\e^{1/2}N_{t/\e} \to 0$ in probability, which follows from the fact that
$(\e^{1/2}N_{t/\e})_{t\geq 0}$ converges in law by Jacod-Shiryaev 
\cite[Theorem VIII-3.11 p 473]{js}. Indeed, $(\e^{1/2} N_{t/\e})_{t\geq 0}$ 
is a continuous local martingale of which
the bracket matrix $\e\int_0^{t/\e} \nabla^*f(V_s) \nabla^*f(V_s) \dd s$ a.s.
converges to $[\intrd  \nabla^*f(v) \nabla f(v) \mu_\beta(\dd v)]t$ as $\e\to 0$ by Lemma \ref{ergo1}-(ii).

\vip

We now check (a).
Since $|g(v)|\leq C(1+|v|)|1-|v|\Gamma'(|v|)/\Gamma(|v|)|$ and since $\beta=4+d$,
$$
\intrd |g(v)|^2(1+|v|)^{2}\mu_\beta(\dd v)\leq C
\intrd |g(v)|^2(1+|v|)^{-2-d}\dd v\leq C \int_0^\infty \Big|1-\frac{r\Gamma'(r)}{\Gamma(r)}\Big|^2
\frac{r^{d-1}\dd r}{(1+r)^{d}},
$$
which converges since, by assumption, $\int_1^\infty r^{-1}|r\Gamma'(r)/\Gamma(r)-1|^2r^{-1}\dd r$.

\vip

We finally check (b), recalling the notation introduced in Subsection \ref{notadebut}:
$$
J=\intrd g(v)\mu_\beta(\dd v)=m_\beta + \frac{3(4+d)}2 \int_0^\infty\Big(r-\frac{r^2\Gamma'(r)}{\Gamma(r)} \Big)
\nu_\beta'(\dd r) \int_\Sd \Psi(\theta)\nu_\beta(\dd \theta)=m_\beta + \frac{3(4+d)}2 J_1J_2,
$$
the first and last equalities standing for definitions. First, 
$$
J_1=b_\beta \int_0^\infty\Big(r-\frac{r^2\Gamma'(r)}{\Gamma(r)} \Big) r^{d-1}[\Gamma(r)]^{-\beta}\dd r= 
b_\beta \int_0^\infty r^d[\Gamma(r)]^{-\beta} \dd r +\frac{b_\beta}\beta 
\int_0^\infty r^{1+d}([\Gamma(r)]^{-\beta})' \dd r$$
whence $J_1=b_\beta[1-(1+d)/\beta] \int_0^\infty r^d [\Gamma(r)]^{-\beta}$ and thus
$J_1=[1-(1+d)/\beta] m'_\beta$.

\vip

Next, recall that $\frac12\Delta_S \Psi(\theta)-\frac\beta2 [\gamma(\theta)]^{-1}\nabla_S\gamma(\theta)\cdot 
\nabla_S \Psi(\theta)=\frac92 \Psi(\theta)+\theta$ 
by Lemma \ref{Psi} and observe that for any smooth $\psi:\Sd\mapsto \rr$, we have
$$
\int_{\Sd} \Big[\Delta_S \psi(\theta)-\beta \frac{\nabla_S\gamma(\theta)}{\gamma(\theta)}\cdot 
\nabla_S \psi(\theta) \Big]\nu_\beta(\dd \theta)
=a_\beta \int_\Sd \ddiv_S([\gamma(\theta)]^{-\beta}\nabla_S\psi(\theta)  ) \varsigma(\dd \theta)=0.
$$
We conclude that $J_2=\int_\Sd \Psi(\theta)\nu_\beta(\dd \theta)=-(2/9)\int_\Sd \theta \nu_\beta(\dd \theta)=
-(2/9)M_\beta$, so that finally,
$$
J=m_\beta-\frac{3(4+d)}2\Big(1-\frac{1+d}\beta\Big)\frac 29 m'_\beta M_\beta=0
$$
because $\beta=4+d$ and $m_\beta=m'_\beta M_\beta$.
\end{proof}

We finally treat the main martingale term.

\begin{lem}\label{dc3}
With the assumptions and notation of Lemma \ref{dc1},
$( |\log \e|^{-1/2}\e^{1/2} M_{t/\e})_{t\geq 0} \stackrel{d}\longrightarrow  (\Sigma B_t)_{t\geq 0}$ as $\e\to 0$, 
for some $\Sigma \in \cS_d^+$, where $(B_t)_{t\geq 0}$ is a $d$-dimensional Brownian motion.
\end{lem}

\begin{proof}
Using one more time Jacod-Shiryaev \cite[Theorem VIII-3.11 p 473]{js}, 
it suffices to check that there is $\Sigma^2 \in \cS_d^+$ such that $\lim_{\e\to 0} Z^\e_t = \Sigma^2 t$
in probability for each $t\geq 0$,
where $Z^\e_t$ is the matrix of brackets of the martingale $|\log\e|^{-1/2}\e^{1/2}M_{t/\e}$, namely
$$
Z^\e_t = \frac{\e}{|\log \e|} \int_0^{t/\e} R_s^4 D(\Theta_s)\dd s,
$$
where $D(\theta)=\nabla^*_S \Psi(\theta) \nabla_S\Psi(\theta)+9\Psi(\theta)\Psi^*(\theta)$.
We proceed by coupling.

\vip

\underline{\it Step 1.} We recall Notation \ref{notaf} and use Lemma \ref{exr} with $a_\e=\kappa \e$,
where $\kappa = \int_\rr [\sigma(w)]^{-2}\dd w<\infty$, see Lemma \ref{fcts}-(i).
We consider a one-dimensional Brownian motion $(W_t)_{t\geq 0}$, introduce 
$A^\e_t=\e a_\e^{-2}\intot [\sigma(W_s/a_\e)]^{-2}\dd s$ and its inverse $\rho^\e_t$ and put
$R^\e_t = \sqrt \e h^{-1}(W_{\rho^\e_t}/a_\e)$. We know from Lemma \ref{exr} that
$S^\e_t=\e^{-1/2}R^\e_{\e t}$ solves \eqref{eqr}. We also consider the solution $(\hTheta_t)_{t\geq 0}$ of 
\eqref{eqt}, independent of $(W_t)_{t\geq 0}$.
We then know from Lemma \ref{rep} that, setting $H_t^\e=\intot [S_s^\e]^{-2} \dd s$,
$(S^\e_t \hTheta_{H^\e_t})_{t\geq 0}\stackrel{d}=(V_t)_{t\geq 0}$. In particular, for each $t\geq 0$,
$Z^\e_t\stackrel{d}= \tZ^\e_t$, where
$$
\tZ^\e_t=\frac{\e}{|\log \e|} \int_0^{t/\e} (S_s^\e)^4 D(\hTheta_{H^\e_s})\dd s.
$$

\underline{\it Step 2.} Here we verify that $\tZ^\e_t=K^\e_{\rho^\e_t}$, where, recalling Notation \ref{notaf},
$$
K^\e_t=\frac{\e}{|\log\e|a_\e^2}\int_0^t 
\frac{[h^{-1}(W_s/a_\e)]^4D(\hTheta_{T^\e_s})}{[\sigma(W_s/a_\e)]^2}\dd s
\quad \hbox{and} \quad T^\e_t = \frac{1}{a_\e^2}\int_0^{t} 
\frac{\dd u }{\psi(W_u/a_\e)}.
$$
Recalling that $S_s^\e=\e^{-1/2}R^\e_{\e t}=h^{-1}(W_{\rho^\e_{\e s}}/a_\e)$ and using the change of variables
$u=\rho^\e_{\e s}$, i.e. $s=\e^{-1}A^\e_u$, whence $\dd s = a_\e^{-2}[\sigma(W_u/a_\e)]^{-2}\dd u$, we find
$$
\tZ^\e_t=\frac{\e}{|\log \e|a_\e^2}\int_0^{\rho^\e_t} \frac{[h^{-1}(W_u/a_\e)]^4D\big(\hTheta_{H^\e_{\e^{-1}A^\e_u}}\big)}
{[\sigma(W_u/a_\e)]^2}\dd u,
$$
and it only remains to check that $H^\e_{\e^{-1}A^\e_t}=T^\e_t$. But, with the same change of variables,
$$
H^\e_{\e^{-1}A^\e_t}=\int_0^{\e^{-1}A^\e_t}\frac{\dd s}{[h^{-1}(W_{\rho^\e_{\e s}}/a_\e)]^{2}} = \frac{1}{a_\e^2}
\int_0^t \frac{\dd u}{[\sigma(W_u/a_\e)]^2[h^{-1}(W_u/a_\e)]^2}= \frac{1}{a_\e^2}\int_0^{t} 
\frac{\dd u }{\psi(W_u/a_\e)}.
$$

\underline{\it Step 3.} We now prove that there is $C>0$ such that $\E[|K^\e_t- G_D I^\e_t|^2 \vert \cW]\leq C t/|\log\e|^2$
for all $t\geq 0$, all $\e\in(0,1)$, where 
$\cW=\sigma(W_t,t\geq 0)$, where $G_D=\int_{\Sd} D(\theta)\nu_\beta(\dd \theta)$ and where
$$
I^\e_t = \frac{\e}{|\log\e|a_\e^2}\int_0^t 
\frac{[h^{-1}(W_s/a_\e)]^4}{[\sigma(W_s/a_\e)]^2}\dd s.
$$
We set $\Delta^\e_t= \E[|K^\e_t- G_D I^\e_t|^2 |\cW]$ and write
$$
\Delta^\e_t=\frac{\e^2}{|\log\e|^2a_\e^4}\int_0^t \int_0^t \frac{[h^{-1}(W_a/a_\e)]^4}{[\sigma(W_a/a_\e)]^2}
\frac{[h^{-1}(W_b/a_\e)]^4}{[\sigma(W_b/a_\e)]^2} \E([D(\hTheta_{T^\e_a})-G_D][D(\hTheta_{T^\e_b})-G_D]|\cW) \dd a \dd b.
$$
Using that $(T^\e_t)_{t\geq 0}$ is $\cW$-measurable, that $(\hTheta_t)_{t\geq 0}$ is independent of $\cW$,
that $D$ is bounded and that $G_D=\int_{\Sd}D \dd \nu_\beta$, 
we deduce from Lemma \ref{ergo2}-(ii) and the Markov property
that there are $C>0$ and $\lambda>0$ such that a.s.,
\begin{align*}
|\E([D(\hTheta_{T^\e_a})-G_D][D(\hTheta_{T^\e_b})-G_D]|\cW)|
\leq C \exp(-\lambda |T^\e_b-T^\e_a|).
\end{align*}
By Lemma \ref{fcts}-(iii) and since $a_\e=\kappa \e$, 
we have $\e a_\e^{-2}[h^{-1}(w/a_\e)]^4[\sigma(w/a_\e)]^{-2}\leq C(\e+|w|)^{-1}$, whence
$$
\Delta^\e_t\leq \frac{C}{|\log\e|^2}\int_0^t \int_0^t (\e+|W_a|)^{-1}(\e+|W_b|)^{-1}
\exp(-\lambda |T^\e_a-T^\e_b|) \dd a \dd b.
$$
Next, since $a_\e^2 \psi(w/a_\e) \leq C(\e+|w|)^2$ by Lemma \ref{fcts}-(iv),
$$
\lambda |T^\e_a-T^\e_b|= \lambda \Big|\frac{1}{a_\e^2}\int_a^b
\frac{\dd s }{\psi(W_s/a_\e)}\dd s \Big|\geq c \Big|\int_a^b (\e+|W_s|)^{-2}\dd s \Big|
$$
for some $c>0$. Using furthermore that $xy\leq x^2+y^2$ and a symmetry argument, we conclude that
\begin{align*}
\Delta^\e_t\leq \frac{C}{|\log\e|^2}\int_0^t \int_0^t (\e+|W_a|)^{-2}
\exp\Big(-c \Big|\int_a^b (\e+|W_s|)^{-2}\dd s \Big|\Big)\dd a \dd b \leq \frac{Ct}{|\log\e|^2}.
\end{align*}
We finally used \eqref{magic}.
%that for all $b\in [0,t]$, all continuous function $\varphi:\rr_+\mapsto \rr_+$,
%\begin{align}
%\int_0^t \varphi(a)\exp\Big(-\Big|\int_a^b \varphi(s)\dd s\Big|\Big)\dd a \leq 2. \label{magic}
%\end{align}

\vip

\underline{\it Step 4.} One can check precisely as in Lemma \ref{tl} that
for all $T\geq 0$, $\sup_{[0,T]}|A^\e_t-L^0_t|\to 0$ a.s. as $\e\to 0$,
where $(L^0_t)_{t\geq 0}$ is the local time at $0$ of $(W_t)_{t\geq 0}$.
Actually, the proof of Lemma \ref{tl} works (without any modification) for any $\beta>d$.

\vip

\underline{\it Step 5.} We next verify that for each $T\geq 0$, a.s., 
$\lim_{\e\to 0} \sup_{[0,T]}|I^\e_t-(\kappa/36)L^0_t|=0$.
This resembles the proof of Lemma \ref{tl}.
By Lemma \ref{fcts}-(iii), we know that $[h^{-1}(w)]^4/[\sigma(w)]^2 \leq C(1+|w|)^{-1}$ and that
\begin{equation}\label{eqlog}
\int_{-x}^x \frac{[h^{-1}(w)]^4\dd w} {[\sigma(w)]^2 } \stackrel{x\to \infty}\sim \frac{\log x}{36}.
\end{equation}
We fix $\delta>0$ and write $I^\e_t=J^{\e,\delta}_t+Q^{\e,\delta}_t$, where
$$
J^{\e,\delta}_t=\frac{\e}{|\log\e|a_\e^2}\int_0^t \frac{[h^{-1}(W_s/a_\e)]^4}
{[\sigma(W_s/a_\e)]^2}\indiq_{\{|W_s|> \delta\}}\dd s
\quad\hbox{and}\quad
Q^{\e,\delta}_t=\frac{\e}{|\log\e|a_\e^2}\int_0^t \frac{[h^{-1}(W_s/a_\e)]^4}
{[\sigma(W_s/a_\e)]^2}\indiq_{\{|W_s|\leq \delta\}}\dd s.
$$
Recalling that $a_\e=\kappa \e$ and using that $|w|> \delta$ implies 
$[h^{-1}(w/a_\e)]^4/[\sigma(w/a_\e)]^2 \leq C(1+|\delta/\e|)^{-1}$, we find that
$\sup_{[0,T]}J^{\e,\delta}_t\leq C T /[\delta |\log\e|]$, which tends to $0$ as $\e\to 0$.
We next use the occupation times formula to write
\begin{align*}
Q^{\e,\delta}_t=& \frac{\e}{|\log\e|a_\e^2}\int_{-\delta}^\delta \frac{[h^{-1}(x/a_\e)]^4 L^x_t \dd x}
{[\sigma(x/a_\e)]^2}\\
=&\frac{\e}{|\log\e|a_\e^2}\int_{-\delta}^\delta \frac{[h^{-1}(x/a_\e)]^4 \dd x}
{[\sigma(x/a_\e)]^2}  L^0_t+ \frac{\e}{|\log\e|a_\e^2}\int_{-\delta}^\delta \frac{[h^{-1}(x/a_\e)]^4 (L^x_t-L^0_t) \dd x}
{[\sigma(x/a_\e)]^2}\\
=&r_{\e,\delta} L_t^0+ R^{\e,\delta}_t,
\end{align*}
the last identity standing for a definition.
But a substitution and \eqref{eqlog} allow us to write  
$$
r_{\e,\delta}=\frac{\e}{|\log\e|a_\e} \int_{-\delta/{a_\e}}^{\delta/{a_\e}} \frac{[h^{-1}(y)]^4 \dd y}
{[\sigma(y)]^2} \stackrel{\e\to 0}\sim
\frac{\e \log (\delta/a_\e)}{36 |\log\e| a_\e} \longrightarrow \frac 1{36\kappa}
$$
as $\e\to 0$ because $a_\e=\kappa\e$. Recalling that $I^\e_t=r_{\e,\delta} L_t^0
+ R^{\e,\delta}_t+ J^{\e,\delta}_t$, we have proved that a.s., 
$$
\hbox{for all $\delta>0$,}\quad
\limsup_{\e\to 0}  \sup_{[0,T]}|I^\e_t - L^0_t/(36\kappa)|\leq \limsup_{\e\to 0}  \sup_{[0,T]}|R^{\e,\delta}_t|.
$$
But $|R^{\e,\delta}_t| \leq r_{\e,\delta} \times \sup_{[-\delta,\delta]}|L^x_t-L^0_t|$, whence
$$
\limsup_{\e\to 0}  \sup_{[0,T]}|I^\e_t - L^0_t/(36\kappa)|\leq \sup_{[0,T]\times[-\delta,\delta]}|L^x_t-L^0_t|/(36\kappa)
$$ a.s.
Letting $\delta\to 0$, using Revuz-Yor \cite[Corollary 1.8 p 226]{ry}, completes the step.

\vip

\underline{\it Step 6.} We finally conclude. We fix $t\geq 0$ and recall from Steps 1 and 2 that
$Z^\e_t\stackrel{d}= \tZ^\e_t=K^\e_{\rho^\e_t}$. By Step 4, we know that $A^\e_s\to L^0_s$ a.s. for each
$s\geq 0$, so that Lemma \ref{tc} tells us that $\rho^\e_t$ a.s. converges to 
$\tau_t=\inf\{u\geq 0 : L^0_u>t\}$, because $t$ is a.s. not a jump time of $(\tau_s)_{s\geq 0}$.
Using that $\rho^\e_t$ is $\cW$-measurable, we deduce from Step 3 that for any $A>0$,
$$
\E\big[|K^\e_{\rho^\e_t} - G_D I^\e_{\rho^\e_t}|\indiq_{\{\rho^\e_t \leq A\}}\big] \leq \frac{C A}{|\log \e|^2} \to 0.
$$
Since $\rho^\e_t$ a.s. tends to $\tau_t$, one deduces that 
$\tZ^\e_t - G_D I^\e_{\rho^\e_t}$ converges in probability to $0$.
We then infer from Step 5, using again that $\rho^\e_t$ a.s. tends to $\tau_t$, that
$|I^\e_{\rho^\e_t}-L^0_{\rho^\e_t}/(36\kappa)|$ a.s. tends to $0$. But $(L^0_s)_{s \geq 0}$ being continuous,
we see that $L^0_{\rho^\e_t}$ a.s. tends to $L^0_{\tau_t}=t$.
All this proves that $\tZ^\e_t$, and thus also $Z^\e_t$, converges in probability, as $\e\to 0$, 
to $\Sigma^2 t$, where
$$
\Sigma^2 = \frac{G_D}{36\kappa}.
$$
This symmetric matrix is positive definite: for $\xi \in \rd\setminus\{0\}$, setting
$\Psi_\xi(\theta)=\Psi(\theta)\cdot\xi$,
$$
\xi^* G_D\xi=\intrd \Big[ |\nabla_S\Psi_\xi (\theta)|^2+ 9|\Psi_\xi(\theta)|^2\Big]\nu_\beta(\dd \theta)
\geq 9 \intrd |\Psi_\xi(\theta)|^2 \nu_\beta(\dd \theta)
$$
which cannot vanish, because else we would have $\Psi_\xi(\theta)=0$ for all $\theta \in \Sd$,
which is impossible because $\Psi_\xi$ solves
$\frac12 \Delta_S \Psi_\xi(\theta)- \frac\beta2[\gamma(\theta)]^{-1}\nabla_S\gamma(\theta)\cdot\nabla_S\Psi_\xi(\theta)
=\frac92\Psi_\xi(\theta)+\xi\cdot\theta$
\end{proof}

We now have all the tools to give the

\begin{proof}[Proof of Theorem \ref{mr}-(b)]
We know from Lemma \ref{dc1} that 
$X_t-m_\beta t = [R_t^3\Psi(\Theta_t)-r_0^3\Psi(\theta_0)- Y_t] -M_t$, 
from Lemma \ref{dc2} that for each $t\geq 0$, 
$\lim_{\e\to 0}  |\log \e|^{-1/2}\e^{1/2} [R_{t/\e}^3 \Psi(\Theta_{t/\e})-r_0^3\Psi(\theta_0)+x_0-Y_{t/\e}]=0$
in probability, and from Lemma \ref{dc3} that
$( |\log \e|^{-1/2}\e^{1/2} M_{t/\e})_{t\geq 0} \stackrel{d}\longrightarrow  (\Sigma B_t)_{t\geq 0}$ as $\e\to 0$.
We conclude that, as desired, $( |\log \e|^{-1/2}\e^{1/2} (X_{t/\e}-m_\beta t/\e)_{t\geq 0} 
\stackrel{f.d.}\longrightarrow  (\Sigma B_t)_{t\geq 0}$ as $\e\to 0$.
\end{proof}

We know from Lemma \ref{fcts}-(i) that $\kappa$ can be computed a little more explicitly.

\begin{rk}\label{rkdc} Assume that $\beta=4+d$.

\vip

(i) In Theorem \ref{mr}-(b), $\Sigma\in \cS_d^+$ is the square root of 
$\frac1{36\kappa}\int_\Sd [\nabla^*_S \Psi(\theta) \nabla_S\Psi(\theta)+9\Psi(\theta)\Psi^*(\theta)] 
\nu_\beta(\dd \theta)$, with $\nu_\beta$ defined in Subsection \ref{notadebut},
$\Psi$ introduced in Lemma \ref{Psi} and $\kappa=\frac 16\int_0^\infty r^{d-1}[\Gamma(r)]^{-4-d}\dd r.$

\vip

(ii)
If $U(v)=(1+|v|^2)^{1/2}$, then $\mu_\beta(\dd v)=c_\beta (1+|v|^2)^{-\beta/2}\dd v$
and $m_\beta=0$,
so that we have $(\e^{1/2}|\log\e|^{-1/2} X_{t/\e})_{t\geq 0} \stackrel{f.d.} \longrightarrow (\Sigma B_t)_{t\geq 0}$.
Moreover, $\gamma\equiv 1$, whence $\nu_\beta(\dd\theta)=\varsigma(\dd\theta)$ and
$\Psi(\theta)=-a\theta$, where $a=2/(8+d)$ (a computation shows that
$\Delta_S \Psi(\theta)=a(d-1)\theta$,
whence $\frac12 \Delta_S \Psi(\theta)=\frac 92 \Psi(\theta)+\theta$).
Since now $\nabla_S\Psi(\theta)=-a \pi_{\theta^\perp}$, whence 
$\nabla^*_S \Psi(\theta) \nabla_S\Psi(\theta)=a^2\pi_{\theta^\perp}$, we find
$$
\Sigma^2=\frac{a^2}{36\kappa}\int_{\Sd} [\pi_{\theta^\perp} + 9\theta \theta^* ]\varsigma(\dd\theta)
=\frac{a^2}{36\kappa}\int_{\Sd} [I_d + 8\theta \theta^* ]\varsigma(\dd\theta)
=\frac{a^2}{36\kappa}\Big[\int_{\Sd} (1+8\theta_1^2) \varsigma(\dd\theta)\Big] I_d.
$$
Observing that $\int_{\Sd} \theta_1^2 \varsigma(\dd\theta)=1/d$, we conclude that
$\Sigma=qI_d$, with $q= [9\kappa d (8+d)]^{-1/2}$.
\end{rk}

\section{Appendix}\label{ap}

\subsection{Ergodicity and convergence in law}
We first recall some classical properties of the velocity process.

\begin{lem}\label{ergo1} Assume that $\beta>d$ and consider the 
$\rd\setminus\{0\}$-valued velocity process $(V_t)_{t\geq 0}$, see \eqref{eds}.

\vip

(i) The measure with density $\mu_\beta$ defined in Remark \ref{inv}
is its unique invariant probability measure.

\vip

(ii) For any $\phi \in L^1(\rd,\mu_\beta)$, $\lim_{T\to \infty} T^{-1}\int_0^T \phi(V_s)\dd s 
= \intrd \phi \dd \mu_\beta$ a.s.

\vip

(iii) It holds that $V_t$ goes in law to $\mu_\beta$ as $t\to\infty$.
\end{lem}

\begin{proof}
We denote by $\cL$ the generator of the velocity process, we have
$\cL\varphi(v)=\frac12[\Delta \varphi(v) - \beta F(v)\cdot\nabla\varphi(v)]$ 
for all $\varphi \in C^2(\rd\setminus\{0\})$, all $v\in \rd\setminus\{0\}$.
We also denote by $P_t(v,\dd w)$ its semi-group: for $t\geq 0$
and $v \in \rd\setminus\{0\}$, $P_t(v,\dd w)$ is the law of $V_t$ when $V_0=v$.

\vip

Recalling that $\mu_\beta(\dd v)=c_\beta [U(v)]^{-\beta}\dd v$ and observing that 
$\cL\varphi(v)=\frac12 [U(v)]^\beta\ddiv ([U(v)]^{-\beta}\nabla\varphi(v))$, 
we see that $\int_{\rd} \cL\varphi(v) \mu_\beta(\dd v)=0$
for all $\varphi \in C^2(\rd\setminus\{0\})$, and $\mu_\beta$ is an invariant probability measure.
The uniqueness of this invariant probability measure follows from point (iii).
In a few lines below, we will verify the two following points.

\vip

(a) There is $\Phi: \rd\setminus\{0\}\mapsto [0,\infty)$ of class $C^2$ such that 
$\lim_{|v|\to 0+}\Phi(v)=\lim_{|v|\to\infty} \Phi(v)=\infty$ and, for some $b,c>0$
and some compact set $C\subset \rd\setminus\{0\}$,
for all $v\in \rd\setminus\{0\}$, $\cL\Phi(v)\leq -b + c\indiq_{\{v\in C\}}.$

\vip

(b) There is $t_0>0$ such that for all compact set $C\subset \rd\setminus\{0\}$, there is $\alpha_C>0$
and a probability measure $\zeta_C$ on $\rd\setminus\{0\}$ such that for all 
$A\in \cB(\rd\setminus\{0\})$, $\inf_{x\in C} P_{t_0}(x,A)\geq \alpha_C\zeta_C(A)$.

\vip

These two conditions allow us to apply Theorems 4.4 and 5.1 of Meyn-Tweedie \cite{mt}, which tell
us that $(V_t)_{t\geq 0}$ is Harris recurrent, whence point (ii) (by Revuz-Yor \cite[Theorem 3.12 p 427]{ry}, 
any Harris recurrent process with an invariant probability measure satisfies the ergodic theorem)
and $\cL(V_t)\to \mu_\beta$, whence point (iii).
Indeed, in the terminology of \cite{mt}, (a) implies condition (CD2) and
(b) implies that all compact sets are {\it petite}.

\vip

\underline{\it Point (a).} For some $q>0$ to be chosen later, set, for 
$r\in (0,\infty)$, $g(r)=-q+\indiq_{\{r\in [1,3]\}}$ and 
$\varphi(r)= \int_2^r y^{1-d}[\Gamma(y)]^\beta \dd y \int_2^y 
g(x) x^{d-1}[\Gamma(x)]^{-\beta}\dd x$.
For $v \in \rd\setminus\{0\}$, set $\Phi(v)=\varphi(|v|)+m$, for some constant $m$
to be chosen later.

\vip

But $\varphi'(r)=r^{1-d}[\Gamma(r)]^\beta \int_2^r g(x)x^{d-1}[\Gamma(x)]^{-\beta}\dd x$,
$\varphi''(r)=g(r) - [\frac{d-1}r - \beta \frac{\Gamma'(r)}{\Gamma(r)}]\varphi'(r)$,
$\Delta \Phi(v)=\varphi''(|v|)+\frac{d-1}{|v|}\varphi'(|v|)$ and
$\nabla \Phi(v)=\frac{\varphi'(|v|)}{|v|}v$, whence $F(v)\cdot \nabla \Phi(v)=
\frac{\Gamma'(|v|)}{\Gamma(|v|)}\varphi'(|v|)$, see \eqref{fff}.
All in all, we find that $\cL \Phi(v) = g(|v|)/2$.

\vip

The converging integrals 
$\int_0^2 g(x) x^{d-1}[\Gamma(x)]^{-\beta}\dd x$ 
and $\int_2^\infty  g(x) x^{d-1}[\Gamma(x)]^{-\beta}\dd x$ are positive if $q>0$ is small enough,
and we conclude that 
$$
\lim_{r\to 0} \varphi(r) = \int_0^2y^{1-d}[\Gamma(y)]^\beta \dd y \int_y^2 
g(x) x^{d-1}[\Gamma(x)]^{-\beta}\dd x =\infty
$$
and
$$
\lim_{r\to \infty} \varphi(r) = \int_2^\infty y^{1-d}[\Gamma(y)]^\beta \dd y \int_2^y 
g(x) x^{d-1}[\Gamma(x)]^{-\beta}\dd x =\infty,
$$
whence 
$\lim_{|v|\to 0+} \Phi(v)=\lim_{|v|\to\infty}\Phi(v)=\infty$.
With the choice $m=-\min_{r>0} \varphi(r) \in \rr$, the function $\Phi$ is nonnegative and thus suitable.

\vip

\underline{\it Point (b).} We will prove, and this is sufficient, 
that for all compact set $C \subset \rd\setminus\{0\}$,
there exists a constant $\kappa_C>0$ such that for all $v \in C$
all measurable $A \subset C$, $P_1(v,A) \geq \kappa_C |A|$, where
$|A|$ is the Lebesgue measure of $A$.

\vip

Consider $a'>a>0$ such that the annulus $D=\{x \in \rd, a<|x|<a'\}$ contains $C$.
Recall \eqref{eds} and that the force $F$ is bounded on $D$, see Assumption \ref{as}. 
By the Girsanov theorem, for any $A\in \cB(\rd)$,
$$
P_1(v,A)=\PR_v( V_1 \in A)\geq 
\E_v[\indiq_{\{\forall s\in [0,1], V_s \in D\}}\indiq_{\{V_1 \in A\}}] 
\geq c\E[\indiq_{\{\forall s\in [0,1], v+B_s \in D\}}\indiq_{\{v+B_1 \in A\}}]
$$
for some constant $c>0$, where $(B_t)_{t\in [0,1]}$ is a $d$-dimensional Brownian motion issued from $0$.
But the density $g(v,w)$ of $v+B_1$ restricted to the event that $(v+B_s)_{s\in [0,1]}$ does not get out
of $D$ is bounded below, as a function of $(v,w)$, on $C\times C$,
whence the conclusion.
\end{proof}

We recall some facts about the total variation distance: for
two probability measures $P,Q$ on some measurable set $E$,
\begin{equation}\label{TV1}
||P-Q||_{TV}=\frac12\sup_{||\phi||_\infty\leq 1} \Big| \int_E \phi(x) (P-Q)(\dd x)\Big|
= \inf \Big\{\PR(X\neq Y) : \cL(X)=P,\;\cL(Y)=Q \Big\}.
\end{equation}
Furthermore, if $P$ and $Q$ have some densities $f$ and $g$ with respect to some measure $R$ on $E$, then
\begin{equation}\label{TV2}
||P-Q||_{TV}=\frac12\int_E |f(x)-g(x)|R(\dd x).
\end{equation}

\begin{lem}\label{ergo2} 
We consider the $\Sd$-valued process $(\hTheta_t)_{t\geq 0}$, solution to \eqref{eqt}.

\vip

(i) The measure $\nu_\beta(\dd \theta)=a_\beta [\gamma(\theta)]^{-\beta} \varsigma(\dd \theta)$ on $\Sd$
is its unique invariant probability measure.

\vip

(ii) There is $C>0$ and $\lambda>0$ such that for all $t\geq 0$, all measurable and bounded 
$\phi:\Sd \mapsto \rr$,
$$
\sup_{\theta_0 \in \Sd} \Big|\E_{\theta_0}[\phi(\hTheta_t)]-  \int_\Sd \phi \dd \nu_\beta\Big| 
\leq C ||\phi||_\infty e^{-\lambda t}.
$$

(iii) There exists a (unique in law) stationary eternal version $(\hTheta^\star_t)_{t\in \rr}$ 
of this $\Sd$-valued process process and it holds that $\cL(\hTheta^\star_t)=\nu_\beta$ 
for all $t\in \rr$. We denote by
$\Lambda \in \cP(\cH)$, where $\cH=C(\rr,\Sd)$, the law of this stationary process.

\vip

(iv) Consider the process $(\hTheta_t)_{t\geq 0}$ starting from some given $\theta_0\in \Sd$.
Fix $k\geq 1$ and consider some positive sequences $(t^1_n)_{n\geq 1}$, ..., $(t^k_n)_{n\geq 1}$, all
tending to infinity as $n\to \infty$. We can find, for each $A\geq 1$ and each $n\geq 1$, an i.i.d. family 
of $\Lambda$-distributed eternal processes $(\hTheta^{\star,1,n,A}_t)_{t\in \rr}$,
..., $(\hTheta^{\star,k,n,A}_t)_{t\in \rr}$ such that, setting 
$$
p_A(t^n_1,\dots,t^n_k)=\PR\Big[
\Big(\hTheta_{(t^n_1+t)\lor 0},\dots,\hTheta_{(t^n_1+\dots+t^n_k+t)\lor 0}\Big)_{t\in[-A,A]}
=(\hTheta^{\star,1,n,A}_t,\dots,\hTheta^{\star,k,n,A}_t)_{t\in [-A,A]}\Big],
$$
it holds that $\lim_{n\to \infty} p_A(t^n_1,\dots,t^n_k)=1$.
\end{lem}

\begin{proof}
We recall that the generator $\chL$ of the process $(\hTheta_t)_{t\geq 0}$
is given, for $\varphi\in C^2(\Sd)$ and $\theta\in\Sd$,
by $\chL \varphi(\theta)=\frac12[\Delta_S\varphi(\theta)-\beta \frac{\nabla_S\gamma(\theta)}
{\gamma(\theta)}\cdot\nabla_S\varphi(\theta)]=
\frac12 [\gamma(\theta)]^{\beta}\ddiv_S([\gamma(\theta)]^{-\beta}\nabla_S\varphi(\theta))$,
so that $\nu_\beta(\dd \theta)=a_\beta  [\gamma(\theta)]^{-\beta}\varsigma(\dd \theta)$ is an
invariant probability measure. The uniqueness of this invariant probability measure follows from point (ii).
We denote by $Q_t(x,\dd y)$ the semi-group, defined as the law of $\hTheta_t$ when 
$\hTheta_0=x \in \Sd$. Grigor'yan \cite[Theorem 3.3 p 103]{g} tells us that 
$Q_t(x,\dd y)$ has a density $q_t(x,y)$ with respect to the uniform measure $\varsigma$ on
$\Sd$, which is positive and smooth as a function of $(t,x,y)\in (0,\infty)\times\Sd\times\Sd$.

\vip

We now prove (ii). It suffices to show that $b=\sup_{x,x'\in \Sd}||Q_1(x,\cdot)-Q_1(x',\cdot)||_{TV}<1$, 
because then the semi-group property implies that $||Q_t(x,\cdot)- \nu_\beta||_{TV}\leq b^{\lfloor t\rfloor}$, 
whence the result by \eqref{TV1}. But, setting $a=\min\{q_1(x,y) : x,y \in \Sd\}>0$ and
recalling \eqref{TV2}, we have
\begin{align*}
||Q_1(x,\cdot)-Q_1(x',\cdot)||_{TV}=&\frac12 \int_{\Sd} |q_1(x,y)-q_1(x',y)|\varsigma(\dd y)\\
=&\frac12 \int_{\Sd} |(q_1(x,y)-a)-(q_1(x',y)-a)|\varsigma(\dd y),
\end{align*}
which is bounded by $\frac12 \int_{\Sd} [(q_1(x,y)-a)+(q_1(x',y)-a)]\varsigma(\dd y)=1-a<1$.

\vip

Point (iii) follows from the Kolmogorov extension theorem. Indeed,
consider, for each $n\geq 0$, the solution 
$(\hTheta_t^n)_{t\geq -n}$ starting at time $-n$ with initial law
$\nu_\beta$ and 
observe that for all $m>n$, $\cL((\hTheta_t^n)_{t\geq -n})=\cL((\hTheta_t^m)_{t\geq -n})$
because $\cL(\hTheta_{-n}^m) = \nu_\beta$.

\vip

Next, we consider $n$ large enough so that $\min\{t^n_1,\dots,t^n_k\}\geq 2A$.
We will check by induction that for all $\ell=1,\dots,k$,
$||\Gamma^{n,\ell}_A - \Lambda_A^{\otimes \ell}||_{TV}\leq p_{A,\ell,n}$ where
$\Lambda_A=\cL((\hTheta_t^\star)_{t\in [0,2A]})$, where
$\Gamma^{n,\ell}_A \in \cP(C([0,2A],\Sd)^\ell)$ is the law of 
$((\hTheta_{t_1^n-A+t})_{t\in[0,2A]}, \dots, (\hTheta_{t_1^n+\dots+t_k^\ell-A+t})_{t\in[0,2A]})$,
and where 
$$
p_{A,\ell,n} = C\sum_{i=1}^{\ell} \exp(-\lambda (t^n_{i}-2A)),$$ 
with $C>0$ and $\lambda>0$ introduced in (ii).
By \eqref{TV1}, this will prove point (iv).
We recall that we know from (ii) that $\sup_{\theta_0\in \Sd}||Q_t(\theta_0,\cdot)-\nu_\beta||_{TV}\leq C\exp(-\lambda t)$,
and we introduce $\Lambda_{A,x} \in \cP(C([0,2A],\Sd))$ the law of $(\hTheta_t)_{t\in [0,2A]}$
when starting from $\hTheta_0=x\in\Sd$.

\vip

Writing $\Gamma^{n,1}_A=\int_{\Sd}Q_{t^n_1-A}(\theta_0,\dd x) \Lambda_{A,x}(\cdot)$ and 
$\Lambda_A=\int_{\Sd} \nu_{\beta}(\dd x) \Lambda_{A,x}(\cdot)$, we find that indeed,
$$
||\Gamma^{n,1}_A - \Lambda_A||_{TV}\leq ||Q_{t^n_1-A}(\theta_0,\cdot)-\nu_\beta||_{TV}\leq C\exp(-\lambda (t^n_1-A))\leq
p_{A,1,n}.
$$

Assuming next that $||\Gamma^{n,\ell-1}_A - \Lambda_A^{\otimes (\ell-1)}||_{TV} \leq p_{A,\ell-1,n}$ for some
$\ell \in \{2,\dots,k\}$, we write
\begin{align*}
\Gamma^{n,\ell}_A(\dd \theta^{(1)},\dots,\dd \theta^{(\ell)})
=&\int_{x\in\Sd}\Gamma^{n,\ell-1}_A(\dd \theta^{(1)},\dots,\dd \theta^{(\ell-1)})Q_{t^n_\ell-2A}(\theta^{(\ell-1)}_{2A},\dd x) 
\Lambda_{A,x}(\dd \theta^{(\ell)}),\\
\Lambda^{\otimes \ell}(\dd \theta^{(1)},\dots,\dd \theta^{(\ell)})
=&\int_{x\in\Sd}\Lambda^{\otimes(\ell-1)}_A(\dd \theta^{(1)},\dots,\dd \theta^{(\ell-1)})\nu_\beta(\dd x) 
\Lambda_{A,x}(\dd \theta^{(\ell)}).
\end{align*}
We conclude that
$$
||\Gamma^{n,\ell}_A - \Lambda_A^{\otimes \ell}||_{TV}\leq \sup_{y\in\Sd}||Q_{t^n_\ell-2A}(y,\cdot)-\nu_\beta||_{TV}+
||\Gamma^{n,\ell-1}_A - \Lambda_A^{\otimes (\ell-1)}||_{TV} \leq C e^{-\lambda (t^n_\ell-2A)} + p_{A,\ell-1,n},
$$
which equals $p_{A,\ell,n}$ as desired.
\end{proof}

\subsection{On It\^o's measure}
We recall that It\^o's measure $\Xi \in \cP(\cE)$ was introduced 
in Notation \ref{notaex}.

\begin{lem}\label{exc}
(i) For $\Xi$-almost every $e\in\cE$, we have $\int_0^{\ell(e)/2}|e(u)|^{-2}\dd u=\infty$.

\vip

(ii) For all $\phi \in L^1(\rr)$, 
$\int_{\cE} [\int_0^{\ell(e)}\phi(e(u))\dd u]\Xi(\dd e) = \int_\rr \phi(x)\dd x$.

\vip

(iii) For all measurable $\phi:\rr\mapsto \rr_+$, 
$\int_{\cE} [\int_0^{\ell(e)}\phi(e(u))\dd u]^2\Xi(\dd e) \leq 4[\int_\rr \sqrt{|x|}\phi(x)\dd x]^2$.

\vip

(iv) For $q<3/2$, for $\Xi$-almost every $e\in\cE$, we have $\int_0^{\ell(e)} |e(u)|^{-q} \dd u<\infty$.
\end{lem}

\begin{proof}
Concerning point (i), it suffices to use that $\int_{0+} (r |\log r|)^{-1}\dd r= \infty$ together
with L\'evy's modulus of continuity, see Revuz-Yor
\cite[Theorem 2.7 p 30]{ry}, which implies that for $\Xi$-a.e. $e\in \cE$,
$\limsup_{t \searrow 0} \sup_{r \in [0,t]} (2r |\log r|)^{-1} |e(r)|^{2}=1$.

\vip

Next (iv) follows from (iii), since 
$\int_0^{\ell(e)} |e(u)|^{-q} \dd u<\infty$ if and only if $\int_0^{\ell(e)} |e(u)|^{-q}\indiq_{\{|e(u)|\leq 1\}} 
\dd u<\infty$ (for any $e\in \cE$)
and since $\int_{\cE} [\int_0^{\ell(e)} |e(u)|^{-q} \indiq_{\{|e(u)|\leq 1\}} \dd u]^2\Xi(\dd e)
\leq 4 [\int_{-1}^1 |x|^{1/2-q} \dd x]^2 <\infty$.

\vip

We now check points (ii) and (iii). We recall that for $(W_t)_{t\geq 0}$ a Brownian motion,
for $(L^x_t)_{t\geq 0,x\in \rr}$ its family of local times, for $(\tau_t)_{t\geq 0}$ the inverse of 
$(L^0_t)_{t\geq 0}$, the second Ray-Knight theorem, see Revuz-Yor \cite[Theorem 2.3 p 456]{ry}, 
tells us that $(L^w_{\tau_1})_{w\geq 0}$ is a squared Bessel process with dimension $0$ issued from $1$.
Hence, for some Brownian motion $(B_w)_{w\geq 0}$, we have 
$L^w_{\tau_1}=1+2\int_0^w\sqrt{L^w_{\tau_1}} \dd B_v$, so that
$\E[L^w_{\tau_1}]=1$ and $\E[(L^w_{\tau_1}-1)^2]=4\E[(\int_0^w\sqrt{L^w_{\tau_1}} \dd B_v)^2 ]
=4 \int_0^w \E[L^v_{\tau_1}]\dd v=4w$.
By symmetry, for any $w\in \rr$, we have $\E[L^w_{\tau_1}]=1$ and 
$\E[(L^w_{\tau_1}-1)^2]=4|w|$. Applying \eqref{superform} with $t=1$, we see that
\begin{align*}
\int_\cE \Big[\int_0^{\ell(e)} \phi(e(u))\dd u\Big] \Xi(\dd e)
=\E\Big[ \int_0^1 \int_\cE  \Big[\int_0^{\ell(e)} \phi(e(u))\dd u \Big] \bM(\dd s,\dd e)\Big]
=\E\Big[\int_0^{\tau_1} \phi(W_s)\dd s\Big].
\end{align*}
But finally, by the occupation times formula and the Fubini theorem, 
$$
\E\Big[\int_0^{\tau_1} \phi(W_s)\dd s\Big]=\E\Big[\int_\rr \phi(w)L^w_{\tau_1} \dd w\Big]=\int_\rr \phi(w)\dd w
$$
which proves (ii). Similarly,
\begin{align*}
\int_\cE \Big[\int_0^{\ell(e)} \phi(e(u))\dd u \Big]^2 \Xi(\dd e)=
&\E\Big[\Big(\int_0^1 \int_\cE  \Big[\int_0^{\ell(e)} \phi(e(u))\dd u \Big] \tilde\bM(\dd s,\dd e\Big)^2 \Big]\\
=&\E\Big[\Big(\int_0^{\tau_1} \phi(W_s)\dd s - \int_\rr \phi(w)\dd w \Big)^2\Big]\\
=&\E\Big[\Big(\int_\rr \phi(w) (L^w_{\tau_1}-1)\dd w\Big)^2\Big]\\
=&\int_\rr\int_\rr\phi(w)\phi(u)\E[(L^w_{\tau_1}-1)(L^u_{\tau_1}-1)] \dd w \dd u.
\end{align*}
We complete the proof of (iii) using that $\E[(L^w_{\tau_1}-1)^2]= 4 |x|$
and the Cauchy-Schwarz inequality.
\end{proof}

\subsection{On Bessel processes}

\begin{lem}\label{bess}
(i) Fix $\delta \in (0,2)$, consider a Brownian motion $(W_t)_{t\geq 0}$, introduce the 
inverse $\rho_t$ of $A_t=(2-\delta)^{-2} \intot W_s^{-2(1-\delta)/(2-\delta)} \indiq_{\{W_s>0\}}\dd s$
and set $\cR_t=(W_{\rho_t})_+^{1/(2-\delta)}$.
Then $(\cR_t)_{t\geq 0}$ is a Bessel process with dimension $2-\delta$ issued from $0$.

\vip

(ii) For $(\cR_t)_{t\geq 0}$ a Bessel process with dimension $\delta >0$, a.s.,
for all $t\geq 0$ such that $\cR_t=0$ and all $h>0$, we have $\int_{t}^{t+h} \cR_s^{-2}\dd s=\infty$.
\end{lem}

\begin{proof}
Point (i) is more or less included in Donati-Roynette-Vallois-Yor \cite[Corollary 2.2]{drvy},
who state that for $(R_t)_{t\geq 0}$ a Bessel process with dimension $\delta\in (0,2)$ issued from $0$,
for $C_t=(2-\delta)^{2}\int_0^t R_s^{2(1-\delta)}\dd s$ and for $D_t$ the inverse of $C_t$,
$(R_{D_t})^{2-\delta}$ is a reflected Brownian motion. Moreover, this is clearly an {\it if and only if} condition.

\vip

But for $\cC_t=(2-\delta)^{2}\intot  \cR_s^{2(1-\delta)}\dd s
=(2-\delta)^{2}\intot (W_{\rho_s})_+^{2(1-\delta)/(2-\delta)}\dd s = \int_0^{\rho_t} \indiq_{\{W_u>0\}}\dd u$
and for $\cD_t$ its inverse, we have $\cD_t=A_{\cE_t}$, where $\cE_t$ is the inverse of $\intot  
\indiq_{\{W_s>0\}}\dd s$. It is then clear that $\cR_{\cD_t}^{2-\delta}=(W_{\rho_{\cD_t}})_+=(W_{\cE_t})_+$ 
is a reflected Brownian motion.

\vip

Point (ii) follows from Khoshnevisan \cite[(2.1a) p 1299]{ko} that asserts that a.s., for all $T>0$,
$\limsup_{h \searrow 0}\sup_{t\in[0,T]} [h(1 \lor \log(1/h))]^{1/2} | \cR_{t+h} - \cR_t | = \sqrt2$.
Indeed, $\int_{0+} [h(1 \lor \log(1/h))]^{-1}\dd h = \infty$.
\end{proof}

\subsection{Inverting time changes}

We recall a classical result about the convergence of inverse functions.

\begin{lem}\label{tc}
Consider, for each $n\geq 1$, a continuous increasing bijective function $(a^n_t)_{t\geq 0}$ from $[0,\infty)$
into itself, as well as its inverse $(r^n_t)_{t\geq 0}$. 
Assume that $(a^n_t)_{t\geq 0}$ converges pointwise to some function $(a_t)_{t\geq 0}$ such that 
$\lim_{\infty} a_t=\infty$, denote by $r_t=\inf\{u\geq 0 : a_u>t\}$
its right-continuous generalized inverse and set $J=\{t\in [0,\infty) : r_{t-}<r_t\}$.
For all $t \in [0,\infty)\setminus J$, we have $\lim_{t\to \infty} r^n_t=r_t$.
\end{lem}

\subsection{Technical estimates}
Finally, we study the functions $h,\psi,\sigma$ introduced in Notation \ref{notaf}.
We recall that $h(r)=(\beta+2-d)\int_{r_0}^r u^{1-d}[\Gamma(u)]^\beta \dd u$ is
an increasing bijection from $(0,\infty)$ into $\rr$, that 
$h^{-1}:\rr \mapsto (0,\infty)$ is its inverse function. We have set
$\sigma(w)=h'(h^{-1}(w))$ and $\psi(w)=[\sigma(w) h^{-1}(w)]^{2}$, both being functions
from $\rr$ to $(0,\infty)$.

\begin{lem}\label{fcts}
Fix $\beta>d-2$ and set $\alpha=(\beta+2-d)/3$. 
There are some constants $0<c<C$ such that the results below are valid
for all $w\in\rr$ (except in point (v)).

\vip

(i) If $\beta>d$, $\kappa=\int_\rr [\sigma(z)]^{-2}\dd z =
(\beta+2-d)^{-1}\int_0^\infty r^{d-1}[\Gamma(r)]^{-\beta}\dd r < \infty$.

\vip

(ii) If $\beta>1+d$, $m_\beta'=(\int_\rr h^{-1}(z)[\sigma(z)]^{-2}\dd z)/(\int_\rr [\sigma(z)]^{-2}\dd z)$.

\vip

(iii) If $\beta=4+d$, $[h^{-1}(w)]^4/[\sigma(w)]^2\leq C(1+|w|)^{-1}$
and $\int_{-x}^x \frac{[h^{-1}(z)]^4\dd z} {[\sigma(z)]^2 } \stackrel{x\to \infty}\sim \frac{\log x}{36}$.

\vip

(iv) If $\beta\in [d,4+d]$, $c(1+w)^2\indiq_{\{w>0\}} \leq \psi(w)\leq C(1+|w|)^2$.

\vip
 
(v) If $\beta \in [d,4+d)$, $\lim_{\eta\to 0} \eta^{2}\psi(w/\eta)=(\beta+2-d)^{2} w^{2}$ for any $w>0$.

\vip

(vi) If $\beta>d-2$, $[\sigma(w)]^{-2}\leq C(1+|w|)^{-2(\beta+1-d)/(\beta+2-d)}$.

\vip

(vii) If $\beta=d$, $\int_{-x}^x [\sigma(z)]^{-2}\dd z 
\stackrel{x\to \infty}\sim \frac{\log x}{4}$.

\vip

(viii) If $\beta \in [d,4+d)$, $[1+h^{-1}(w)]/[\sigma(w)]^2 \leq C (1+w)^{1/\alpha -2}\indiq_{\{w\geq 0\}}
+C(1+|w|)^{-2}\indiq_{\{w<0\}}$.

\vip

(ix) If $\beta \in [d,4+d)$, $\forall \; m\in\rr$,  $\lim_{\eta\to 0}
\eta^{1/\alpha-2}[h^{-1}(w/\eta)-m]/[\sigma(w/\eta)]^2= (\beta+2-d)^{-2} w^{1/\alpha-2}\indiq_{\{w\geq 0\}}.$

\vip

(x) If $\beta \in (d-2,d)$ and $a_\e=\e^{(\beta+2-d)/2} $, $\sqrt\e h^{-1}(w/a_\e) \to  w_+^{1/(\beta+2-d)}$
uniformly on compact sets.

\vip

(xi) If $\beta \in (d-2,d)$ and $a_\e=\e^{(\beta+2-d)/2}$, 
$$\lim_{\e\to 0} \e [a_\e \sigma(w/a_\e)]^{-2} = (\beta+2-d)^{-2} w^{-2(\beta+1-d)/(\beta+2-d)}\indiq_{\{w>0\}}.$$
\end{lem}

\begin{proof}
The three following points will be of constant use.

\vip

(a) As $w\to\infty$, we have $h^{-1}(w)\sim w^{1/(\beta+2-d)}$, 
$\sigma(w) \sim (\beta+2-d) w^{(\beta+1-d)/(\beta+2-d)}$ and $\psi(w)\sim(\beta+2-d)^2w^2$.

\vip

(b) If $d\geq 3$, there are $c,c',c''>0$ such that, as $w\to -\infty$,
$h^{-1}(w)\sim c |w|^{-1/(d-2)}$, $\sigma(w) \sim c' |w|^{(d-1)/(d-2)}$ and $\psi(w)\sim c'' |w|^2$.

\vip

(c) If $d=2$, there are $c,c',c''>0$ and a function $\e(w)$ such that $\lim_{w\to -\infty} \e(w)=0$,
$h^{-1}(w)=\exp[-c|w|(1+\e(w))]$, $\sigma(w)\sim c' \exp[c|w|(1+\e(w))]$ as $w\to -\infty$
and $\lim_{w\to-\infty} \psi(w)=c''$.

\vip

To check (a), it suffices to note that by Assumption \ref{as}, $h(r)\sim r^{\beta+2-d}$ as $r\to \infty$.
Next, (b) follows from the fact that $h(r) \sim -\mathfrak{c} r^{2-d}$ as $r \to 0$ (with 
$\mathfrak{c}=[\Gamma(0)]^{\beta}(\beta+2-d)/(d-2)>0$), while (c)
uses that $h(r) \sim -\mathfrak{c} \log (1/r)$ (with $\mathfrak{c}=\beta [\Gamma(0)]^\beta$,
the result then holds with $\e(w)=\mathfrak{c} [\log h^{-1}(w)]/w - 1$,
$c=1/\mathfrak{c}$, $c'=\mathfrak{c}$ and $c''=\mathfrak{c}^2$).

\vip

We now prove (i). Using the substitution $r=h^{-1}(z)$,
$$
\kappa = \int_\rr  \frac{\dd z}{[h'(h^{-1}(z))]^2} = \int_0^\infty \frac{\dd r}{h'(r)}= 
\frac1{\beta+2-d}\int_0^\infty \frac{r^{d-1}}{[\Gamma(r)]^{\beta}}\dd r,
$$
which is finite if and only if $d-1-\beta<-1$, i.e. $\beta>d$. Recall that $\Gamma:[0,\infty)\mapsto(0,\infty)$
is supposed to be continuous and that $\Gamma(r)\sim r$ as $r\to\infty$.

\vip

We proceed similarly for (ii). Recalling that $m_\beta'$ was defined in Subsection \ref{notadebut},
$$
\frac{\int_\rr h^{-1}(z)[\sigma(z)]^{-2}\dd z}{\int_\rr [\sigma(z)]^{-2}\dd z}=
\frac{\int_0^\infty r[h'(r)]^{-1}\dd r}{\int_0^\infty [h'(r)]^{-1}\dd r}
=\frac{\int_0^\infty r^{d}[\Gamma(r)]^{-\beta}\dd r}{\int_0^\infty r^{d-1}[\Gamma(r)]^{-\beta}\dd r}=m_\beta'.
$$

For (iii), we observe that when $\beta=4+d$, (a) implies that 
$[h^{-1}(w)]^4/[\sigma(w)]^2 \sim 36^{-1} w^{-1}$ as $w\to \infty$, so that
we have the bound $[h^{-1}(w)]^4/[\sigma(w)]^2 \leq C(1+|w|)^{-1}$ on $\rr_+$
as well as the estimate $\int_0^x \frac{[h^{-1}(w)]^4\dd w} {[\sigma(w)]^2 } 
\stackrel{x\to \infty}\sim \frac{\log x}{36}$. If $d\geq 3$, (b) tells us that 
$[h^{-1}(w)]^4/[\sigma(w)]^2 \sim c |w|^{-2(d+1)/(d-2)}$ as $w\to -\infty$ (for some constant $c>0$),
and we conclude using that $2(d+1)/(d-2)>1$.
If $d=2$, (c) gives us $[h^{-1}(w)]^4/[\sigma(w)]^2 \sim [c']^{-2} \exp(-6 c |w|(1+\e(w)))$ as $w\to -\infty$,
from which the estimates follow.

\vip

Point (iv) immediately follows from (a) (concerning the lowerbound and the upperbound on $\rr_+$)
and (b) or (c) (concerning the upperbound on $\rr_-$). 

\vip

Point (v) is a consequence of (a).

\vip

Point (vi) follows from (a) (concerning the bound on $\rr_+$) and from (b) (and the fact that
$(d-1)/(d-2)>(\beta+1-d)/(\beta+2-d)$) or (c). 

\vip

With the same arguments as in (vi), we see that $\int_{-x}^x [\sigma(w)]^{-2}\dd w 
\stackrel{x\to \infty}\sim \int_0^x [\sigma(w)]^{-2}\dd w$, which is equivalent to  
$[\log x]/4$ as $x\to \infty$ by (a), whence (vii).

\vip

Points (viii) on (ix) follow from (a) and the fact that
$1/(\beta+2-d)-2(\beta+1-d)/(\beta+2-d)=1/\alpha-2$ (when $w\geq 0$) and (b) or (c) (when $w<0$).

\vip

Points (x) and (xi) follow from (a) (when $w\geq 0$) and (b) or (c) (when $w<0$).
Observe that in (x), the convergence is uniform on compact sets for free by the Dini theorem, since
for each $\e>0$, $w\mapsto \sqrt\e h^{-1}(w/a_\e)$ is non-decreasing and since the limit function
$w\mapsto w_+^{1/(\beta+2-d)}$ is continuous and non-decreasing.
\end{proof}

\end{document}